\documentclass[12pt]{article}
\pdfoutput=1

\usepackage{amsmath, amsthm, amssymb}
\usepackage{ifpdf}
\ifpdf
\usepackage[pdftex]{graphicx}
\else
\usepackage[dvips]{graphicx}
\fi
\usepackage{tikz}
 	 \usetikzlibrary{arrows,backgrounds} 
\usepackage[all]{xy}

\input xy
\xyoption{all}

\usepackage[pdftex,plainpages=false,hypertexnames=false,pdfpagelabels]{hyperref}
\newcommand{\arXiv}[1]{\href{http://arxiv.org/abs/#1}{\tt arXiv:\nolinkurl{#1}}}
\newcommand{\doi}[1]{\href{http://dx.doi.org/#1}{{\tt DOI:#1}}}

\newcommand{\googlebooks}[1]{(preview at \href{http://books.google.com/books?id=#1}{google books})}

\usepackage{xcolor}
\definecolor{dark-red}{rgb}{0.7,0.25,0.25}
\definecolor{dark-blue}{rgb}{0.15,0.15,0.55}
\definecolor{medium-blue}{rgb}{0,0,.8}
\hypersetup{
   colorlinks, linkcolor={purple},
   citecolor={medium-blue}, urlcolor={medium-blue}
}

\theoremstyle{plain}
\newtheorem{thm}{Theorem}[section]
\newtheorem*{thm*}{Theorem}
\newtheorem{cor}[thm]{Corollary}
\newtheorem{lem}[thm]{Lemma}

\newtheorem{prop}[thm]{Proposition}
\newtheorem{quest}[thm]{Question}

\theoremstyle{definition}
\newtheorem{defn}[thm]{Definition}
\newtheorem{nota}[thm]{Notation}

\newtheorem{exs}[thm]{Examples}
\newtheorem{ex}[thm]{Example}
\newtheorem{rem}[thm]{Remark}
\newtheorem{rems}[thm]{Remarks}

\DeclareMathOperator{\Aut}{Aut}
\DeclareMathOperator{\Dom}{Dom}
\DeclareMathOperator{\End}{End}

\DeclareMathOperator{\ev}{ev}
\DeclareMathOperator{\id}{id}

\DeclareMathOperator{\OP}{op}

\DeclareMathOperator{\spann}{span}
\DeclareMathOperator{\Stab}{Stab}

\DeclareMathOperator{\Tr}{Tr}
\DeclareMathOperator{\tr}{tr}

\newcommand{\D}{\displaystyle}
\newcommand{\comment}[1]{}
\newcommand{\hs}{\hspace{.07in}}

\newcommand{\be}{\begin{enumerate}}
\newcommand{\ee}{\end{enumerate}}
\newcommand{\itm}[1]{\item[\underline{\ensuremath{#1}:}]} 
\newcommand{\itt}[1]{\item[\underline{\text{#1}:}]} 
\newcommand{\N}{\mathbb{N}}
 
\newcommand{\R}{\mathbb{R}}
\newcommand{\C}{\mathbb{C}}
\newcommand{\B}{\mathbb{B}}
\renewcommand{\P}{\mathbb{P}}
\newcommand{\I}{\infty} 
\newcommand{\set}[2]{\left\{#1 \middle| #2\right\}}
\newcommand{\thh}{^{\text{th}}}
\renewcommand{\a}{\mathfrak{a}}
\renewcommand{\b}{\mathfrak{b}}
\renewcommand{\c}{\mathfrak{c}}
\newcommand{\n}{\mathfrak{n}}
\newcommand{\m}{\mathfrak{m}}

\newcommand{\cB}{\mathcal{B}}

\newcommand{\cF}{\mathcal{F}}

\newcommand{\cP}{\mathcal{P}}
\newcommand{\cT}{\mathcal{T}}
\newcommand{\cS}{\mathcal{S}}
\newcommand{\cR}{\mathcal{R}}
\newcommand{\cU}{\mathcal{U}}
\newcommand{\op}{^{\OP}}   

\newcommand{\noshow}[1]{}

\tikzstyle{shaded}=[fill=red!10!blue!20!gray!30!white]
\tikzstyle{unshaded}=[fill=white]
\tikzstyle{empty box}=[circle, draw, thick, fill=white, opaque, inner sep=2mm]
\tikzstyle{annular}=[scale=.7, inner sep=1mm, baseline]
\tikzstyle{rectangular}=[scale=.75, inner sep=1mm, baseline=-.1cm]

\newcommand{\Pn}[2]{
\begin{tikzpicture}[rectangular,baseline=.5cm]
	\clip (1,1) --(-1,1) -- (-1,-.6) -- (1,-.6);
	\draw[ultra thick] (0,1.5)--(0,-.25);
	\filldraw[thick, unshaded] (-.5,.5) --(-.5,-.5) -- (.5,-.5) -- (.5,.5)--(-.5,.5);
	\node at (0,0) {$#2$};
	\node at (0.2,.75) {{\scriptsize{$#1$}}};
\end{tikzpicture}}

\newcommand{\PnStar}[2]{
\begin{tikzpicture}[rectangular]
	\clip (1,-1) --(-1,-1) -- (-1,.6) -- (1,.6);
	\draw[ultra thick] (0,-1.5)--(0,.25);
	\filldraw[thick, unshaded] (-.5,.5) --(-.5,-.5) -- (.5,-.5) -- (.5,.5)--(-.5,.5);
	\node at (0,0) {$#2$};
	\node at (.2,-.75) {{\scriptsize{$#1$}}};
\end{tikzpicture}}

\newcommand{\idn}[1]{
\begin{tikzpicture}[rectangular]
	\clip (1,1) --(-1,1) -- (-1,-1) -- (1,-1);
	\draw[ultra thick] (0,1)--(0,-1);
	\node at (.2,0) {{\scriptsize{$#1$}}};
\end{tikzpicture}}

\newcommand{\tensor}[4]{
\begin{tikzpicture}[rectangular]
	\clip (2.5,1) --(-1,1) -- (-1,-1) -- (2.5,-1);
	\draw[ultra thick] (0,1.5)--(0,-1.5);
	\draw[ultra thick] (1.5,1.5)--(1.5,-1.5);
	\filldraw[thick, unshaded] (-.5,.5) --(-.5,-.5) -- (.5,-.5) -- (.5,.5)--(-.5,.5);
	\node at (0,0) {$#2$};
	\node at (.25,.75) {{\scriptsize{$#1$}}};
	\node at (.25,-.75) {{\scriptsize{$#1$}}};
	\filldraw[thick, unshaded] (1,.5) --(1,-.5) -- (2,-.5) -- (2,.5)--(1,.5);
	\node at (1.5,0) {$#4$};
	\node at (1.7,.75) {{\scriptsize{$#3$}}};
	\node at (1.7,-.75) {{\scriptsize{$#3$}}};
\end{tikzpicture}}

\newcommand{\TensorPn}[4]{
\begin{tikzpicture}[rectangular,baseline=0cm]
	\clip (2.5,1) --(-1,1) -- (-1,-.75) -- (2.5,-.75);
	\draw[ultra thick] (0,1.5)--(0,-.5);
	\draw[ultra thick] (1.5,1.5)--(1.5,-.5);
	\filldraw[thick, unshaded] (-.5,.5) --(-.5,-.5) -- (.5,-.5) -- (.5,.5)--(-.5,.5);
	\node at (0,0) {$#2$};
	\node at (.25,.75) {{\scriptsize{$#1$}}};
	\filldraw[thick, unshaded] (1,.5) --(1,-.5) -- (2,-.5) -- (2,.5)--(1,.5);
	\node at (1.5,0) {$#4$};
	\node at (1.7,.75) {{\scriptsize{$#3$}}};
\end{tikzpicture}}

\renewcommand{\include}[2]{
\begin{tikzpicture}[rectangular]
	\clip (1,1) --(-1,1) -- (-1,-1) -- (1,-1);
	\draw[ultra thick] (0,1.5)--(0,-1.5);
	\draw (.8,1.5)--(.8,-1.5);
	\filldraw[thick, unshaded] (-.5,.5) --(-.5,-.5) -- (.5,-.5) -- (.5,.5)--(-.5,.5);
	\node at (0,0) {$#2$};
	\node at (.2,.75) {{\scriptsize{$#1$}}};
	\node at (.2,-.75) {{\scriptsize{$#1$}}};
\end{tikzpicture}}

\newcommand{\includeop}[2]{
\begin{tikzpicture}[rectangular]
	\clip (1,1) --(-1,1) -- (-1,-1) -- (1,-1);
	\draw (-.8,1.5)--(-.8,-1.5);
	\draw[ultra thick] (0,1.5)--(0,-1.5);
	\filldraw[thick, unshaded] (-.5,.5) --(-.5,-.5) -- (.5,-.5) -- (.5,.5)--(-.5,.5);
	\node at (0,0) {$#2$};
	\node at (.2,.75) {{\scriptsize{$#1$}}};
	\node at (.2,-.75) {{\scriptsize{$#1$}}};
\end{tikzpicture}}

\newcommand{\OperatorValuedWeight}[2]{
\begin{tikzpicture}[rectangular]
	\clip (1.5,1) --(-1,1) -- (-1,-1) -- (1.5,-1);
	\draw[ultra thick] (0,1.5)--(0,-1.5);
	\draw (1.1,.3)--(1.3,0)--(1.5,.3);
	\draw (.7,.5) arc (180:0:.3cm) --(1.3,-.5) arc (0:-180:.3cm);
	\filldraw[thick, unshaded] (-.5,.5) --(-.5,-.5) -- (1,-.5) -- (1,.5)--(-.5,.5);
	\node at (.25,0) {$#2$};
	\node at (.2,.75) {{\scriptsize{$#1$}}};
	\node at (.2,-.75) {{\scriptsize{$#1$}}};
\end{tikzpicture}}

\newcommand{\OperatorValuedWeightOp}[2]{
\begin{tikzpicture}[rectangular]
	\clip (1,1) --(-1.5,1) -- (-1.25,-1) -- (1,-1);
	\draw[ultra thick] (0,1.5)--(0,-1.5);
	\draw (-1.1,.3)--(-1.3,0)--(-1.5,.3);
	\draw (-.7,.5) arc (0:180:.3cm) --(-1.3,-.5) arc (180:360:.3cm);
	\filldraw[thick, unshaded] (-1,.5) --(-1,-.5) -- (.5,-.5) -- (.5,.5)--(-1,.5);
	\node at (-.25,0) {$#2$};
	\node at (.2,.75) {{\scriptsize{$#1$}}};
	\node at (.2,-.75) {{\scriptsize{$#1$}}};
\end{tikzpicture}}

\newcommand{\TraceOfTwo}[3]{
\begin{tikzpicture}[rectangular,baseline=.5cm]
	\clip (1.3,2.6) --(-1,2.6) -- (-1,-1.1) -- (1.3,-1.1);
	\draw[ultra thick] (0,2)--(0,-.5);
	\draw[ultra thick] (.8,.8)--(1,.5)--(1.2,.8);
	\draw[ultra thick] (0,2) arc (180:0:.5cm) --(1,-.5) arc (0:-180:.5cm);
	\filldraw[thick, unshaded] (-.5,.5) --(-.5,-.5) -- (.5,-.5) -- (.5,.5)--(-.5,.5);
	\filldraw[thick, unshaded] (-.5,1) --(-.5,2) -- (.5,2) -- (.5,1)--(-.5,1);
	\node at (0,1.5) {$#2$};
	\node at (0,0) {$#3$};
	\node at (1.2,-.3) {{\scriptsize{$#1$}}};
\end{tikzpicture}}

\newcommand{\TraceOfTwoOp}[3]{
\begin{tikzpicture}[rectangular,baseline=.5cm]
	\clip (1,2.6) --(-1.3,2.6) -- (-1.3,-1.1) -- (1,-1.1);
	\draw[ultra thick] (0,2)--(0,-.5);
	\draw[ultra thick] (-.8,.8)--(-1,.5)--(-1.2,.8);
	\draw[ultra thick] (0,2) arc (0:180:.5cm) --(-1,-.5) arc (180:360:.5cm);
	\filldraw[thick, unshaded] (-.5,.5) --(-.5,-.5) -- (.5,-.5) -- (.5,.5)--(-.5,.5);
	\filldraw[thick, unshaded] (-.5,1) --(-.5,2) -- (.5,2) -- (.5,1)--(-.5,1);
	\node at (0,1.5) {$#2$};
	\node at (0,0) {$#3$};
	\node at (-1.2,-.3) {{\scriptsize{$#1$}}};
\end{tikzpicture}}

\newcommand{\TraceOfTwoBig}[3]{
\begin{tikzpicture}[rectangular,baseline=.75cm]
	\clip (1.3,3.7) --(-1,3.7) -- (-1,-1.2) -- (1.3,-1.2);
	\draw[ultra thick] (0,2)--(0,-.5);
	\draw[ultra thick] (.8,.8)--(1,.5)--(1.2,.8);
	\draw[ultra thick] (0,3) arc (180:0:.5cm) --(1,-.5) arc (0:-180:.5cm);
	\filldraw[thick, unshaded] (-.5,.5) --(-.5,-.5) -- (.5,-.5) -- (.5,.5)--(-.5,.5);
	\filldraw[thick, unshaded] (-.5,1) --(-.5,3) -- (.5,3) -- (.5,1)--(-.5,1);
	\draw[thick] (-.5,2)--(.5,2);
	\node at (0,2.5) {$#1$};
	\node at (0,1.5) {$#2$};
	\node at (0,0) {$#3$};
\end{tikzpicture}}

\newcommand{\TraceOfTwoBigOp}[3]{
\begin{tikzpicture}[rectangular,baseline=.75cm]
	\clip (-1.3,3.7) --(1,3.7) -- (1,-1.2) -- (-1.3,-1.2);
	\draw[ultra thick] (0,2)--(0,-.5);
	\draw[ultra thick] (-.8,.8)--(-1,.5)--(-1.2,.8);
	\draw[ultra thick] (0,3) arc (0:180:.5cm) --(-1,-.5) arc (-180:0:.5cm);
	\filldraw[thick, unshaded] (-.5,.5) --(-.5,-.5) -- (.5,-.5) -- (.5,.5)--(-.5,.5);
	\filldraw[thick, unshaded] (-.5,1) --(-.5,3) -- (.5,3) -- (.5,1)--(-.5,1);
	\draw[thick] (-.5,2)--(.5,2);
	\node at (0,2.5) {$#1$};
	\node at (0,1.5) {$#2$};
	\node at (0,0) {$#3$};
\end{tikzpicture}}

\newcommand{\ClosedLoop}[1]{
\begin{tikzpicture}[rectangular]
	\clip (1,1) --(-1,1) -- (-1,-1) -- (1,-1);
	\draw[ultra thick] (.29,.2)--(.49,-.1)--(.69,.2);
	\draw[ultra thick] (0,0) circle (.5cm);
	\node at (-.7,0) {{\scriptsize{$#1$}}};
\end{tikzpicture}}

\newcommand{\ClosedLoopOp}[1]{
\begin{tikzpicture}[rectangular]
	\clip (1,1) --(-1,1) -- (-1,-1) -- (1,-1);
	\draw[ultra thick] (-.28,.2)--(-.48,-.1)--(-.68,.2);
	\draw[ultra thick] (0,0) circle (.5cm);
	\node at (.7,0) {{\scriptsize{$#1$}}};
\end{tikzpicture}}

\newcommand{\InnerProduct}[2]{
\begin{tikzpicture}[rectangular,baseline=.5cm]
	\clip (1,2.25) --(-1,2.25) -- (-1,-.75) -- (1,-.75);
	\draw[ultra thick] (0,1.5)--(0,0);
	\filldraw[thick, unshaded] (-.5,.5) --(-.5,-.5) -- (.5,-.5) -- (.5,.5)--(-.5,.5);
	\filldraw[thick, unshaded] (-.5,1) --(-.5,2) -- (.5,2) -- (.5,1)--(-.5,1);
	\node at (0,1.5) {$#1$};
	\node at (0,0) {$#2$};
\end{tikzpicture}}

\newcommand{\TripleInnerProduct}[3]{
\begin{tikzpicture}[rectangular]
	\clip (1,2.25) --(-1,2.25) -- (-1,-2.25) -- (1,-2.25);
	\draw[ultra thick] (0,1.5)--(0,-1.5);
	\filldraw[thick, unshaded] (-.5,.5) --(-.5,-.5) -- (.5,-.5) -- (.5,.5)--(-.5,.5);
	\filldraw[thick, unshaded] (-.5,1) --(-.5,2) -- (.5,2) -- (.5,1)--(-.5,1);
	\filldraw[thick, unshaded] (-.5,-1) --(-.5,-2) -- (.5,-2) -- (.5,-1)--(-.5,-1);
	\node at (0,1.5) {$#1$};
	\node at (0,-1.5) {$#3$};
	\node at (0,0) {$#2$};
\end{tikzpicture}}

\newcommand{\WideTripleInnerProduct}[3]{
\begin{tikzpicture}[rectangular]
	\clip (1,2.25) --(-1,2.25) -- (-1,-2.25) -- (1,-2.25);
	\draw[ultra thick] (0,1.5)--(0,-1.5);
	\filldraw[thick, unshaded] (-.7,.5) --(-.7,-.5) -- (.7,-.5) -- (.7,.5)--(-.7,.5);
	\filldraw[thick, unshaded] (-.5,1) --(-.5,2) -- (.5,2) -- (.5,1)--(-.5,1);
	\filldraw[thick, unshaded] (-.5,-1) --(-.5,-2) -- (.5,-2) -- (.5,-1)--(-.5,-1);
	\node at (0,1.5) {$#1$};
	\node at (0,-1.5) {$#3$};
	\node at (0,0) {$#2$};
\end{tikzpicture}}

\newcommand{\BigInnerProduct}[3]{
\begin{tikzpicture}[rectangular]
	\clip (1.5,2.25) --(-1.5,2.25) -- (-1.5,-2.25) -- (1.5,-2.25);
	\draw[ultra thick] (0,1.5)--(0,-1.5);
	\filldraw[thick, unshaded] (-1,.5) --(-1,-.5) -- (1,-.5) -- (1,.5)--(-1,.5);
	\filldraw[thick, unshaded] (-1,1) --(-1,2) -- (1,2) -- (1,1)--(-1,1);
	\filldraw[thick, unshaded] (-1,-1) --(-1,-2) -- (1,-2) -- (1,-1)--(-1,-1);
	\node at (0,1.5) {$#1$};
	\node at (0,-1.5) {$#3$};
	\node at (0,0) {$#2$};
\end{tikzpicture}}

\newcommand{\ActOnPn}[2]{
\begin{tikzpicture}[rectangular,baseline=.5cm]
	\clip (1,2.5) --(-1,2.5) -- (-1,-.75) -- (1,-.75);
	\draw[ultra thick] (0,2.5)--(0,-.25);
	\filldraw[thick, unshaded] (-.5,.5) --(-.5,-.5) -- (.5,-.5) -- (.5,.5)--(-.5,.5);
	\filldraw[thick, unshaded] (-.5,1) --(-.5,2) -- (.5,2) -- (.5,1)--(-.5,1);
	\node at (0,1.5) {$#1$};
	\node at (0,0) {$#2$};
\end{tikzpicture}}

\newcommand{\CentralVectorOperator}[3]{
\begin{tikzpicture}[rectangular,baseline=.5cm]
	\clip (1,2.5) --(-1,2.5) -- (-1,-.5) -- (1,-.5);
	\draw[ultra thick] (0,3)--(0,-1);
	\filldraw[thick, unshaded] (-.5,0) --(-.5,2) -- (.5,2) -- (.5,0)--(-.5,0);
	\draw[thick] (-.5,1)--(.5,1);
	\node at (0,.5) {$#3$};
	\node at (0,1.5) {$#2$};
	\node at (0.2,2.25) {{\scriptsize{$#1$}}};
	\node at (0.2,-.25) {{\scriptsize{$#1$}}};
\end{tikzpicture}}

\newcommand{\TraceEllipse}[2]{
\begin{tikzpicture}[rectangular]
	\clip (1.5,1.5) --(-1,1.5) -- (-1,-1.5) -- (1.5,-1.5);
	\draw[ultra thick] (.8,.5)--(1,.2)--(1.2,.5);
	\draw[ultra thick] (0,.5) arc (180:0:.5cm) --(1,-.5) arc (0:-180:.5cm);
	\filldraw[thick, unshaded] (-.5,.5) --(-.5,-.5) -- (.5,-.5) -- (.5,.5)--(-.5,.5);
	\node at (0,0) {$#2$};
	\node at (.5,-1.25) {{\scriptsize{$#1$}}};
\end{tikzpicture}}

\newcommand{\ExtremalityOne}[2]{
\begin{tikzpicture}[rectangular]
	\clip (2.25,1.6) --(-1,1.6) -- (-1,-1.6) -- (2.25,-1.6);
	\draw[ultra thick] (1.8,.5)--(2,.2)--(2.2,.5);
	\draw[ultra thick] (0,.5) arc (180:0:1cm) -- (2,-.5) arc (0:-180:1cm);
	\draw (1.1,.3)--(1.3,0)--(1.5,.3);
	\draw (.7,.5) arc (180:0:.3cm) --(1.3,-.5) arc (0:-180:.3cm);
	\filldraw[thick, unshaded] (-.5,.5) --(-.5,-.5) -- (1,-.5) -- (1,.5)--(-.5,.5);
	\node at (.25,0) {$#2$};
	\node at (-.1,-1) {{\scriptsize{$#1$}}};
\end{tikzpicture}}

\newcommand{\ExtremalityTwo}[2]{
\begin{tikzpicture}[rectangular]
	\clip (1.5,1.5) --(-1.3,1.5) -- (-1.3,-1.5) -- (1.5,-1.5);
	\draw[ultra thick] (-1.2,.5)--(-1,.2)--(-.8,.5);
	\draw[ultra thick] (0,.5) arc (0:180:.5cm) -- (-1,-.5) arc (-180:0:.5cm);
	\draw (1.1,.3)--(1.3,0)--(1.5,.3);
	\draw (.7,.5) arc (180:0:.3cm) --(1.3,-.5) arc (0:-180:.3cm);
	\filldraw[thick, unshaded] (-.5,.5) --(-.5,-.5) -- (1,-.5) -- (1,.5)--(-.5,.5);
	\node at (.25,0) {$#2$};
	\node at (0,-1) {{\scriptsize{$#1$}}};
\end{tikzpicture}}

\newcommand{\TraceOpEllipse}[2]{
\begin{tikzpicture}[rectangular]
	\clip (1,1.5) --(-1.5,1.5) -- (-1.5,-1.5) -- (1,-1.5);
	\draw[ultra thick] (-.8,.5)--(-1,.2)--(-1.2,.5);
	\draw[ultra thick] (0,.5) arc (0:180:.5cm) --(-1,-.5) arc (180:360:.5cm);
	\filldraw[thick, unshaded] (-.5,.5) --(-.5,-.5) -- (.5,-.5) -- (.5,.5)--(-.5,.5);
	\node at (0,0) {$#2$};
	\node at (-.5,-1.25) {{\scriptsize{$#1$}}};
\end{tikzpicture}}

\newcommand{\RotationOp}[3]{
\begin{tikzpicture}[rectangular,baseline=-2cm]
	\clip (-1.65,-1.5) --(2,-1.5) -- (2,-3.5) -- (-1.65,-3.5);
	\draw[ultra thick] (-.6,-2) arc (0:180:.4cm) -- (-1.4,-2.5) .. controls ++(270:1.4cm) and ++(270:2.2cm) ..  (1.5,-1.5)--(1.5,-.5);
	\draw[ultra thick] (0,-.5)--(0,-2.5);
	\draw[ultra thick] (-1.2,-2.2)--(-1.4,-2.5)--(-1.6,-2.2);
	\filldraw[thick, unshaded] (.5,-3)--(.5,-2)--(-1,-2)--(-1,-3)--(.5,-3);
	\node at (-.25,-2.5) {$#3$};
	\node at (0.25,-1.75) {{\scriptsize{$#2$}}};
	\node at (1.7,-1.75) {{\scriptsize{$#1$}}};
\end{tikzpicture}}

\newcommand{\Rotation}[3]{
\begin{tikzpicture}[rectangular,baseline=-2cm]
	\clip (1.65,-1.5) --(-2,-1.5) -- (-2,-3.5) -- (1.65,-3.5);
	\draw[ultra thick] (.6,-2) arc (180:0:.4cm) -- (1.4,-2.5) .. controls ++(270:1.4cm) and ++(270:2.2cm) ..  (-1.5,-1.5)--(-1.5,-.5);
	\draw[ultra thick] (0,-.5)--(0,-2.5);
	\draw[ultra thick] (1.2,-2.2)--(1.4,-2.5)--(1.6,-2.2);
	\filldraw[thick, unshaded] (-.5,-3)--(-.5,-2)--(1,-2)--(1,-3)--(-.5,-3);
	\node at (.25,-2.5) {$#3$};
	\node at (0.2,-1.75) {{\scriptsize{$#2$}}};
	\node at (-1.25,-1.75) {{\scriptsize{$#1$}}};
\end{tikzpicture}}

\newcommand{\Switch}[4]{
\begin{tikzpicture}[rectangular,baseline=-1.5cm]
	\clip (1.6,-1.4) --(-3,-1.4) -- (-3,-3.7) -- (1.6,-3.7);
	\draw[ultra thick] (0,-2) arc (180:0:.5cm) -- (1,-2.5) .. controls ++(270:1.4cm) and ++(270:1.4cm) ..  (-2.5,-2.5)--(-2.5,-1.5);
	\draw[ultra thick] (-1.5,-1.25)--(-1.5,-2.5);
	\draw[ultra thick] (.8,-2.2)--(1,-2.5)--(1.2,-2.2);
	\filldraw[thick, unshaded] (-.5,-3)--(-.5,-2)--(.5,-2)--(.5,-3)--(-.5,-3);
	\filldraw[thick, unshaded] (-2,-3)--(-2,-2)--(-1,-2)--(-1,-3)--(-2,-3);
	\node at (0,-2.5) {$#2$};
	\node at (-1.5,-2.5) {$#4$};
	\node at (-1.25,-1.75) {{\scriptsize{$#3$}}};
	\node at (-2.3,-1.75) {{\scriptsize{$#1$}}};
\end{tikzpicture}}

\newcommand{\Insert}[5]{
\begin{tikzpicture}[rectangular,baseline=-1.5cm]
	\clip (1.6,0) --(-2.5,0) -- (-2.5,-3.1) -- (1.6,-3.1);
	\draw[ultra thick] (1.2,.5)--(1.2,-2.5);
	\draw[ultra thick] (-1.2,.5)--(-1.2,-2.5);
	\draw[ultra thick] (0,.5)--(0,-1);
	\filldraw[thick, unshaded] (-1.5,-3)--(-1.5,-2)--(1.5,-2)--(1.5,-3)--(-1.5,-3);
	\filldraw[thick, unshaded] (-.5,-1.5)--(-.5,-.5)--(.5,-.5)--(.5,-1.5)--(-.5,-1.5);
	\node at (0,-1) {$#2$};
	\node at (-.15,-2.5) {$#5$};
	\node at (-1.8,-1.25) {{\scriptsize{$#3$}}};
	\node at (1.4,-1.25) {{\scriptsize{$#4$}}};
	\node at (.2,-.25) {{\scriptsize{$#1$}}};
\end{tikzpicture}}

\begin{document}
\title{A planar calculus for infinite index subfactors}
\author{ David Penneys }
\date{\today}
\maketitle
\begin{abstract}
We develop an analog of Jones' planar calculus for  $II_1$-factor bimodules with arbitrary left and right von Neumann dimension. We generalize to bimodules Burns' results on rotations and extremality for infinite index subfactors. These results are obtained without Jones' basic construction and the resulting Jones projections.
\end{abstract}
\tableofcontents
%%%%%%%%%%%%%%%%%%%%%%%%%%%%%%%%%%%%%%%%%%%%%%%%%%

\section{Introduction}

Jones initiated the modern theory of subfactors in \cite{MR696688}. Given a finite index $II_1$-subfactor $A_0\subseteq A_1$, he used the \underline{basic construction} to obtain the Jones tower $(A_n)_{n\geq 0}$, obtained iteratively by adding the Jones projections $(e_n)_{n\geq 1}$ which satisfy the Temperley-Lieb relations. Jones used this structure to show the index  lies in the range $\set{4\cos^2(\pi/n)}{n\geq 3}\cup [4,\I)$, and he found an example for each value.

Much initial subfactor research classified hyperfinite subfactors of small index ($[A_1\colon A_0] \leq 4$) by studying the \underline{standard invariant}, i.e., the two towers of higher relative commutants $(A_i'\cap A_j)_{i=0,1;j\geq 0}$ \cite{MR996454,MR999799,MR1145672,MR1278111}. This combinatorial data was axiomatized in three slightly different structures: paragroups \cite{MR996454}, $\lambda$-lattices \cite{MR1334479}, and planar algebras \cite{math/9909027}. When combined, these viewpoints produce strong results, e.g., standard invariants with index in $(4,5)$ are completely classified, excluding the $A_\I$ standard invariant at each index value \cite{MR1198815} (see \cite{1007.1730,1007.2240,1109.3190,1010.3797} for more details).

Some finite index results generalize to infinite index subfactors, such as discrete, irreducible, ``depth $2$" subfactors correspond to outer (cocylce) actions of Kac algebras \cite{MR1055223,MR1387518}, and the classical Galois correspondence still holds for outer actions of infinite discrete groups and minimal actions of compact groups \cite{MR1622812}. 

In his Ph.D. thesis \cite{burns}, Burns studied rotations and extremality for infinite index, since the key to isotopy invariance of Jones' planar calculus in \cite{math/9909027} is the rotation operator (also known to Ocneanu). Burns' essential observation for finite index was that the centralizer algebras $A_0'\cap A_n$ coincide with the central $L^2$-vectors:
$$
A_0'\cap L^2(A_n)=\set{\zeta\in L^2(A_n)}{a\zeta=\zeta a \text{ for all }a\in A_0}.
$$
Burns found an elegant formula for the rotation on $P_{n,+}=A_0'\cap \bigotimes^n_{A_0} L^2(A_1)$:
$$
\rho = \sum_{\beta} L_\beta R_\beta^*
$$
where $\{\beta\}$ is a Pimsner-Popa basis for $A_1$ over $A_0$, $L_\beta$ is the left creation operator, and $R_\beta^*$ is the right annihilation operator (see Definition \ref{defn:RelativeTensorProduct}). This approach was generalized in \cite{1007.3173} to define a canonical planar $*$-algebra associated to a strongly Markov inclusion of finite von Neumann algebras. Burns adapted his formula to infinite index, and he showed existence of the rotation on the central $L^2$-vectors is equivalent to approximate extremality of the subfactor. 

In infinite index, $A_0'\cap A_n$ and $A_0'\cap L^2(A_n)$ do not coincide. One naturally asks:

\begin{quest}\label{quest:StandardInvariant}
What is a suitable standard invariant for infinite index subfactors?
\end{quest}

A definitive answer to Question \ref{quest:StandardInvariant} is not yet known. On one hand, we have the two towers of centralizer algebras $(A_i'\cap A_j)_{i=0,1;j\geq 0}$ in which we can multiply (the shift isomorphisms $A_i'\cap A_j\cong A_{i+2}'\cap A_{j+2}$ still hold by \cite{MR1387518}). On the other hand, we have the central $L^2$-vectors on which we have Burns' rotation (in the approximately extremal case) and graded multiplication in the sense of \cite{MR2732052} (tensoring of central vectors). However, the operator valued weights which replace the conditional expectations do not preserve these spaces and may not be well-defined. All this structure is necessary for a good planar calculus. We ask:

\begin{quest}\label{quest:PlanarCalculus}
What is the strongest planar calculus we can define for infinite index subfactors?
\end{quest}

In this paper, we propose an answer to Question \ref{quest:PlanarCalculus} using both centralizer algebras and central $L^2$-vectors. We do so in more generality, starting with a bimodule $\sb{A}H_A$ over a $II_1$-factor $A$ (one recovers the subfactor case when $A=A_0$ and $H=L^2(A_1)$). First, we set $H^n=\bigotimes_A^n H$, $Q_n=A'\cap (A\op)'\cap B(H^n)$ (the centralizer algebras), and $P_n = A'\cap H^n = \set{\zeta\in H^n}{a\zeta=\zeta a\text{ for all }a\in A}$ (the central $L^2$-vectors). As mentioned above, the $P_n$'s naturally form a graded algebra $P_\bullet$ in the sense of \cite{MR2732052} under relative tensor product. We represent central vectors in $P_n$ as in \cite{MR2732052} by boxes with $n$ strings emanating from the top, and we denote graded multiplication (relative tensor product) of $\zeta_m\in P_m$ and $\zeta_n\in P_n$ by
$$
\zeta_m\otimes \zeta_n = \TensorPn{m}{\zeta_m}{n}{\zeta_n}\in P_{m+n}.
$$
We represent elements of $Q_n$ as boxes with strings emanating from top and bottom. For $\zeta\in P_n$, note that the creation-annihilation operator $L(\zeta)L(\zeta)^*=R(\zeta)R(\zeta)^*$ lies in $Q_n$, which we represent as
$$
L(\zeta)L(\zeta)^* = \CentralVectorOperator{n}{\zeta}{\zeta}\in Q_n.
$$

\begin{thm}
The  extended positive cones $\widehat{Q_n^+}$ (in the sense of \cite{MR534673}) naturally form an algebra $\widehat{Q_\bullet^+}$ over the operad $\B\P$ generated by the oriented tangles
$$
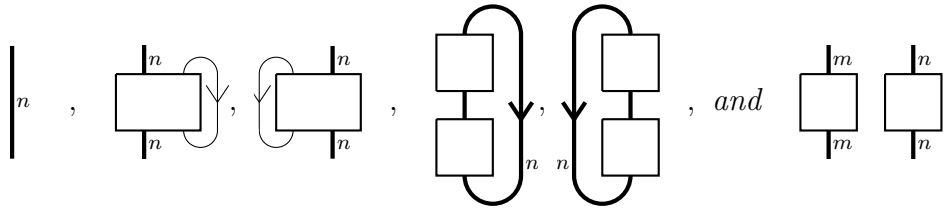
\idn{n},\, \OperatorValuedWeight{n}{},\,\OperatorValuedWeightOp{n}{},\,\TraceOfTwo{n}{}{},\,\TraceOfTwoOp{n}{}{},\text{ and }\tensor{m}{}{n}{}
$$
for $m,n\geq 0$ up to planar isotopy. (We suppress external disks, draw one thick string labelled $n$ for $n$ individual strings, and orient all strings upward unless otherwise specified.)

Moreover, the $\B\P$-algebra $\widehat{Q_\bullet^+}$ and graded algebra $P_\bullet$ are compatible: if $z\in \widehat{Q_n^+}$ and $\zeta\in P_n$, then
$$
 z(\omega_\zeta)=\TripleInnerProduct{\zeta}{z}{\zeta}= \TraceOfTwoBig{\zeta}{\zeta}{z} =\Tr_n(L(\zeta)L(\zeta)^*\cdot z)
$$
where $\Tr_n$ is the canonical trace on $Q_n$ coming from the right $A$-action on $H^n$. (Note that the multiplication tangle only makes sense once we take the trace by \cite{MR534673}. See Theorem \ref{thm:BilinearExtension} for more details.)
\end{thm}

We generalize to bimodules Burns' work on rotations: an operator $\rho$ on the central $L^2$-vectors $P_n$ is a \underline{Burns rotation} if for all left and right bounded vectors $b_1,\dots,b_n\in H$, (omitting the subscript $A$ on the tensors,)
$$
\langle \rho(\zeta),b_1\otimes\cdots\otimes b_n\rangle = \langle \zeta, b_2\otimes\cdots\otimes b_n\otimes b_1\rangle.
$$
Note this equation implies the uniqueness and periodicity of $\rho$ if it exists. We generalize Burns' notion of (approximate) extremality, and we prove the following theorem:

\begin{thm}\label{thm:MainTheorem}
Consider the following statements (include all or none of the parenthetical statements):
\item[(1)] $H^n$ is (approximately) extremal for some $n\geq 1$,
\item[(2)] $H^n$ is (approximately) extremal for all $n\geq 1$,
\item[(3)] The (possibly non-)unitary $\rho$ exists on $P_{2n}$ for all $n\geq 1$, and
\item[(4)] The (possibly non-)unitary $\rho$ exists on $P_{2n}$ for some $n\geq 1$.
\item
Then $(1)\Rightarrow(2)\Rightarrow(3)\Rightarrow (4)$. If $H$ is symmetric, then $(4)\Rightarrow (1)$.

When $\rho$ exists, we represent it diagrammatically by
$$
\rho^m(\zeta)=\Rotation{m}{n}{\zeta} \text{ for } \zeta\in P_{m+n},
$$ 
(well-defined by Corollary \ref{cor:POperad}) and these diagrams are compatible with the diagrams above in the sense of Theorem \ref{thm:MoveAround}.
\end{thm}

Interestingly, we find our planar structure without the use of Jones' basic construction and resulting Jones projections! 

%%%%%%%%%%%%%%%%%%%%%%%%%%%%%%%%%

\paragraph{Outline:\\}
\hspace{.03in}
In Section \ref{sec:preliminaries}, we give a brief introduction to modules, the relative tensor product, extended positive cones, and operator valued weights. Subsections \ref{sec:FactsAboutRelativeTensorProduct} and \ref{sec:LemmataPositiveCones} provide some helpful, well-known results for the convenience of the reader.

In Subsection \ref{sec:TowersOfBimodules}, starting with our $A-A$ bimodule $H$, we introduce $H^n$ along with two towers of algebras $C_n,C_n\op$, a tower of centralizer algebras $Q_n=C_n\cap C_n\op$, and the central $L^2$-vectors $P_n$. We then compute formulas for the various canonical maps associated with these towers. In Subsection \ref{sec:PAoverPositiveCones}, we show  the extended positive cones (in the sense of \cite{MR534673}) of the centralizer algebras $\widehat{Q_n^+}$ naturally form an algebra over an operad $\B\P$ (we use positive cones so we can ``conditionally expect" using operator valued weights). In Subsection \ref{sec:CentralVectors}, we show that the vectors in $P_\bullet$ are left and right $A$-bounded and form a graded algebra in the sense of \cite{MR2732052}. We then show the compatibility of $\widehat{Q_\bullet^+}$ and $P_\bullet$ in Subsection \ref{sec:compatible}.

Subsection \ref{sec:Extremality} defines extremality for bimodules and Burns rotations. In Subsection \ref{sec:DiagramsOfRotation}, we show how the Burns rotation fits in our planar calculus, and in Subsection \ref{sec:ExtremalityImpliesRotations}, we show that (approximate) extremality implies the existence of the Burns rotation (Theorem \ref{thm:rotation}). A converse of this theorem for symmetric bimodules is obtained in Subsection \ref{sec:symmetric}, which finishes the proof of Theorem \ref{thm:MainTheorem}.

In Section \ref{sec:Examples}, we discuss the centralizer algebras $Q_n$ and central $L^2$-vectors $P_n$ for some basic examples. In particular, in Corollaries \ref{cor:FiniteDimensional} and \ref{cor:InfiniteSymmetric}, we find an infinite index subfactor for which $\dim(Q_n)<\I$ and $\dim(P_n)=1$ for all $n\in\N$. This example contrasts Burns' example of an infinite index subfactor with a type $III$ summand in a higher relative commutant \cite{burns}.

Throughout the paper, we need some technical results which have been included in a few appendices. Appendix \ref{sec:TensorUnbounded} shows that the relative tensor product of extended positive cones is well-defined and associative, which is necessary for our planar calculus. Appendix \ref{sec:BP} discusses the operad $\B\P$ which acts on the positive cones $\widehat{Q_n^+}$, including results on generating sets of tangles, standard form of tangles, and well-definition of the action. In Appendix \ref{sec:cones}, we axiomatize the notion of \underline{extended positive cone} to make rigorous the idea of a planar algebra over such objects. The main intricacy is that we must make multiplication by $\I_\R$ well-defined.

%%%%%%%%%%%%%%%%%%%%%%%%%%%%%%%%%

\paragraph{Future research:\\}
\hspace{.01in}
The annular Temperley-Lieb category, especially the rotation, played an important role in the construction of certain exotic finite index subfactors \cite{0902.1294,0909.4099}. In a future paper with Jones, we will incorporate the odd Jones projections for infinite index (see \cite{burns}) into the planar calculus, and we will give the analog of the annular Tempeley-Lieb category for infinite index. We hope this viewpoint will be as fruitful as in the finite index case.

The results of this paper should generalize to bimodules over an arbitrary finite von Neumann algebra. As it requires substantial calculations while obscuring the main new ideas presented here, this generalization will appear in a future paper.

Finally, it would be interesting to try to connect Connes' results on self-dual positive cones \cite{MR0377533} to the extended positive cones axiomatized in Appendix \ref{sec:cones}.

%%%%%%%%%%%%%%%%%%%%%%%%%%%%%%%%%

\paragraph{Acknowledgements:\\}
\hspace{.09in}The author would like to thank Stephen Curran, Steven Deprez, Michael Hartglass, Vaughan Jones, Scott Morrison, Jean-Luc Sauvageot, J. Owen Sizemore, and Makoto Yamashita for many helpful conversations. The author would like to thank Vaughan Jones again for giving him this project and for his supervision while completing his Ph.D. at the University of California, Berkeley.

This majority of this work was completed at the Institut Henri Poincar\'{e} during the trimester on von Neumann algebras and ergodic theory of groups actions. The author would like to thank the organizers Damien Gaboriau, Sorin Popa, and Stefaan Vaes for their support during this time. The author was also supported by DOD-DARPA grant HR0011-11-1-0001.

%%%%%%%%%%%%%%%%%%%%%%%%%%%%%%%%%%%%%%%%%%%%%%%%%%
\section{Preliminaries}\label{sec:preliminaries}

\begin{nota}
\item[$\bullet$] Throughout this paper, a \underline{trace} on a finite von Neumann algebra means a faithful, normal, tracial state unless otherwise specified.
\item[$\bullet$] $A$ will always denote a finite von Neumann algebra with trace $\tr_A$. 
\item[$\bullet$] We use the notation $\widehat{a}$ to denote the image of $a\in A$ in $L^2(A,\tr_A)$.
\item[$\bullet$] For a semifinite von Neumann algebra $M$ with normal, faithful, semifinite (n.f.s.) trace $\Tr_M$, we write
\begin{align*}
\n_{\Tr_M}&=\set{x\in M}{\Tr_M(x^*x)<\I}\text{ and}\\
\m_{\Tr_M}&= \n_{\Tr_M}^*\n_{\Tr_M} = \spann\set{x^* y}{x,y\in \n_{\Tr_M}}.
\end{align*}
\end{nota}

%%%%%%%%%%%%%%%%%%%%%%%%%%%%%%%%%%%%%%%%%%%%%%%%%%
\subsection{Modules and the relative tensor product}\label{sec:RelativeTensorProduct}
This exposition follows \cite{MR561983,MR703809,MR1278111,MR1387518,MR1424954,MR1753177,burns}.

\begin{defn}[Left modules]\label{defn:LeftModule}
If $\sb{A}K$ is a left Hilbert A-module, then the set of left $A$-bounded vectors is given by 
$$D(\sb{A}K)=\set{\eta\in K}{\|a\eta\|_2\leq \lambda \|a\|_2 \text{ for some }\lambda\geq 0},$$
and each $\eta\in D(\sb{A}K)$ gives a bounded map $R(\eta)\colon L^2(A)\to H$ by the extension of $\widehat{a}\mapsto a\eta$.

For $\eta_1,\eta_2\in D(\sb{A}K)$, we have an $A$-valued inner product given by $\sb{A}\langle \eta_1,\eta_2\rangle = JR(\eta_1)^*R(\eta_2)J\in A$ satisfying
\be
\item[(1)] ${\sb{A}\langle} a\eta_1 + \eta_2,\eta_3\rangle= a {\sb{A}\langle}\eta_1,\eta_3\rangle+{\sb{A}\langle}\eta_2,\eta_3\rangle$,
\item[(2)] ${\sb{A}\langle} \eta_1, \eta_2\rangle^* = {\sb{A}\langle} \eta_2,\eta_1\rangle$, and
\item[(3)] ${\sb{A}\langle} x \eta_1, \eta_2\rangle={\sb{A}\langle} \eta_,x^* \eta_2\rangle$
\ee
for all $a\in A$, $x\in A'\cap B(K)$, and $\eta_1,\eta_2,\eta_3\in D(\sb{A}K)$ (note $x\eta_i\in D(\sb{A}K)$).

An $\sb{A}K$-basis is a set of vectors $\{\alpha\}\subset D(\sb{A}K)$ such that 
$$
\sum\limits_\alpha R(\alpha)R(\alpha)^* = 1_K \Longleftrightarrow \sum_\alpha {\sb{A}\langle}\eta ,\alpha\rangle\alpha = \eta\text{ for all }\eta\in D(\sb{A}K).
$$
$\sb{A}K$-bases exist by \cite{MR561983}.

The canonical trace on $A'\cap B(K)$ is given by $\Tr_{A'\cap B(K)}(x)= \sum_\alpha \langle x \alpha,\alpha\rangle$ where $\{\alpha\}$ is any $\sb{A}K$ basis.

If $\eta\in D(\sb{A}K)$, then $\Tr_{A'\cap B(K)}(R(\eta)R(\eta)^*)=\tr_A(\sb{A}{\langle} \eta,\eta\rangle)=\|\eta\|^2_2$.
\end{defn}

\begin{defn}[Right modules]\label{defn:RightModule}
A right Hilbert $A$-module is the same as a left Hilbert $A\op$-module. If $H_A$ is a right Hilbert $A$-module, we write $\xi a$ for $a\op \xi$ for all $a\op\in A\op$. We get parallel definitions:

The set of right $A$-bounded vectors is given by
$$D(H_A)=\set{\xi\in H}{\|\xi a\|_2\leq \lambda \|a\|_2 \text{ for some }\lambda\geq 0}.$$
Each $\xi \in D(H_A)$ defines a bounded map $L(\xi)\colon L^2(A)\to H$ by the extension of $\widehat{a}\mapsto \xi a$.

For $\xi_1,\xi_2\in D(H_A)$, we have an $A$-valued inner product given by $\langle \xi_1|\xi_2\rangle_A = L(\xi_1)^*L(\xi_2)\in A$ satisfying
\be
\item[(1)] $\langle \xi_1 | \xi_2 a + \xi_3\rangle_A=\langle \xi_1 | \xi_2\rangle_A a+\langle \xi_1 |\xi_3\rangle_A$,
\item[(2)] $\langle \xi_1 | \xi_2\rangle_A^* = \langle \xi_2 | \xi_1\rangle_A$, and
\item[(3)] $\langle x\xi_1 | \xi_2\rangle_A=\langle \xi_1 | x^*\xi_2\rangle_A$
\ee
for all $a\in A$, $x\in (A\op)'\cap B(H)$, and $\xi_1,\xi_2,\xi_3\in D(H_A)$ (note $x\xi_i\in D(H_A)$).

An $H_A$-basis is a set of vectors $\{\beta\}\subset D(H_A)$ such that 
$$
\sum\limits_\beta L(\beta)L(\beta)^* = 1_H\Longleftrightarrow \sum_\beta \beta \langle \beta|\xi\rangle_A = \xi \text{ for all } \xi\in D(H_A).
$$
$H_A$-bases exist by \cite{MR561983}.

The canonical trace on on $(A\op)'\cap B(H)$ is given by  $\Tr_{(A\op)'\cap B(H)}(x)= \sum_\beta \langle x \beta,\beta\rangle$ where $\{\beta\}$ is any $H_A$ basis.

If $\xi\in D(H_A)$, then $\Tr_{(A\op)'\cap B(H)}(L(\xi)L(\xi)^*)=\tr_A(\langle \xi|\xi\rangle_A)=\|\xi\|^2_2$.
\end{defn}

\begin{defn}[Relative tensor product]\label{defn:RelativeTensorProduct}
The relative tensor product $H\otimes_A K$ is given by one of the three equivalent definitions:
\item[(1)] the completion of the algebraic tensor product $D(H_A)\odot_A K$ under the pseudo-norm induced by the sesquilinear form $\langle \xi\odot \eta, \xi'\odot \eta'\rangle = \langle \langle \xi'|\xi\rangle_A \eta,\eta'\rangle$,
\item[(2)] the completion of the algebraic tensor product $H \odot_A D(\sb{A} K)$ under the pseudo-norm induced by the sesquilinear form $\langle \xi\odot \eta, \xi'\odot \eta'\rangle =  \langle \xi_1{\sb{A}\langle} \eta_1,\eta_2\rangle, \xi_2\rangle_H$, or
\item[(3)] the completion of the algebraic tensor product $D(H_A)\odot_A D(\sb{A} K)$ under the pseudo-norm induced by the sesquilinear form
$$
\langle \xi_1\odot \eta_1,\xi_2\odot \eta_2\rangle = \langle \xi_1{\sb{A}\langle} \eta_1,\eta_2\rangle, \xi_2\rangle_H= \langle \langle \xi_2|\xi_1\rangle_A \eta_1,\eta_2\rangle_K .
$$
The image of $\xi\odot \eta$ in $H\otimes_A K$ is denoted $\xi\otimes \eta$. (This notation avoids confusion with the operators $x\otimes_A y$ as in Lemma \ref{lem:binormal}.)

Given $\xi\in D(H_A)$ and $\eta\in D(\sb{A} K)$, we get bounded creation operators $L_\xi \colon K\to H\otimes_A K$ by $\eta'\mapsto \xi\otimes \eta'$ and $R_\eta \colon H\to H\otimes_A K$ by $\xi'\mapsto \xi'\otimes \eta$, whose adjoints are the annihilation operators given by
$L_\xi^*(\xi'\otimes\eta')=\langle \xi|\xi'\rangle_A \eta'$ and $R_\eta^*(\xi'\otimes \eta')=\xi'{\sb{A}\langle} \eta',\eta\rangle$.
\end{defn}

\begin{defn}[Fiber product, \cite{MR799587,MR1753177}]\label{defn:FiberProduct}
Suppose $A\op\subset M_1\subset B(H)$ and $A\subset M_2\subset B(K)$. Then we define
$$
M_1'\otimes_A M_2' = \set{x\otimes_A y}{x\in M_1'\text{ and } y\in M_2'}\subset B(H\otimes_A K)
$$
(see Appendix \ref{sec:TensorUnbounded} and Lemma \ref{lem:binormal}), and the fiber product of $M_1$ and $M_2$ over $A$ is given by $M_1 \star_A M_2 = (M_1' \otimes_A M_2')'$. The fiber product satisfies:
\item[$\bullet$] $(M_1 \star_A M_2) \cap (N_1 \star_A N_2) = (M_1\cap N_1)\star_A (M_2\cap N_2)$ and
\item[$\bullet$] $M_1 \star_A A = ((A\op)'\cap M_1) \otimes_A 1_K$ and $A\op\star_A M_2 = 1_H \otimes_A (A'\cap M_2)$.
\item In particular,
$$
(B(H)\star_A A)'=((A\op)'\otimes_A 1_K)'=A\op \star_A B(K)=1_H \otimes_A A'.
$$
\end{defn}

%%%%%%%%%%%%%%%%%%%%%%%%%%%%%%%%%%%%%%%%%%%%%%%%%%%%%%%%%%%%%%
\subsection{Some easy facts about the relative tensor product}\label{sec:FactsAboutRelativeTensorProduct}
The following are well-known to experts, but we reproduce them here for the sake of completeness and the reader's convenience. For this subsection, $H_A$ is a right Hilbert $A$-module, and $\sb{A}K$ is a left Hilbert $A$-module unless otherwise stated.

\begin{lem}\label{lem:UnitaryBases}
Suppose $\{\beta\}$ is an $H_A$-basis. Then if $u\in U((A\op)'\cap B(H))$, $\{u\beta\}$ is another $H_A$-basis. If $v\in U(A)$, then $\{\beta v\}$ is also an $H_A$-basis. A similar result holds for left modules.
\end{lem}
\begin{proof}
For $u\in (A\op)'\cap B(H)$, $L( u \beta)L(u\beta)^*=uL(\beta) L(\beta)^* u^*$. Thus
$$
\sum_{u\beta}L(u\beta)L(u\beta)^*=u\left(\sum_{\beta} L(\beta)L(\beta)^*\right) u^* =  1_H.
$$
If $v\in U(A)$, then $L(\beta v) L(\beta v^*)=L(\beta) v v^*L(\beta)^*=L(\beta)L(\beta)^*$, and the result follows.
\end{proof}

\begin{lem}\label{lem:remove}
Let $\xi_1,\xi_2\in D(H_A)$ and $\eta_1,\eta_2\in D(\sb{A}K)$. Then $L_{\xi_1}^*L_{\xi_2}\in B(K)$ is left multiplication by $\langle \xi_1|\xi_2\rangle_A$ and $R_{\eta_1}^*R_{\eta_2}\in B(H)$ is right multiplication by $\sb{A}{\langle} \eta_1,\eta_2\rangle$.
\end{lem}
\begin{proof}
$\langle L_{\xi_1}^* L_{\xi_2} \eta_1,\eta_2\rangle 
= \langle \xi_2\otimes \eta_1, \xi_1\otimes \eta_2\rangle 
=  \langle \langle \xi_1 | \xi_2\rangle_A \eta_1, \eta_2\rangle 
$. The other is as trivial.
\end{proof}

\begin{lem}
If $\{\beta\}$ is an $H_A$-basis, then $\sum_\beta L_\beta L_\beta^* = 1_{H\otimes_A K}$. Similarly, if $\{\alpha\}$ is an $\sb{A}H$-basis, then $\sum_\alpha R_\alpha R_\alpha^*=1_{H\otimes_A K}$.
\end{lem}
\begin{proof}
We prove the first statement. Suppose $\xi\in D(H_A)$ and $\eta\in D(\sb{A}K)$. Then
$$
\sum_\beta L_\beta L_\beta^*( \xi\otimes \eta )= \sum_\beta L_\beta (L_\beta^* L_\xi) \eta=\sum_\beta\beta \langle \beta|\xi\rangle_A \otimes \eta=\xi\otimes \eta.
$$
\end{proof}

\begin{lem}\label{lem:RN}
Suppose $\eta\in {\sb{A}K}$ and $\eta'\in D(\sb{A}K)$. Then there is a unique $\sb{A}\langle \eta',\eta\rangle\in L^2(A)\subset L^1(A)$ such that $\langle a \eta,\eta'\rangle_K = \langle a, {\sb{A}\langle} \eta',\eta\rangle\rangle_{L^2(A)}$ for all $a\in A$. A similar result holds for right modules.
\end{lem}
\begin{proof}
If $\xi\in D(\sb{A}K)$, this is just the usual Radon-Nikodym derivative, and
\begin{align*}
\|{\sb{A}\langle} \eta',\eta \rangle\|_2
& =\sup_{a\in A, \|\widehat{a}\|_2\leq 1} |\langle \widehat{a}, {\sb{A}\langle}\eta',\eta\rangle^{\widehat{\hs}}\rangle_{L^2(A)}| 
= \sup_{a\in A, \|\widehat{a}\|_2\leq 1} \tr ({\sb{A}\langle}\eta,\eta'\rangle a) \\
& = \sup_{a\in A, \|\widehat{a}\|_2\leq 1} |\langle a \eta,\eta'\rangle_K| 
\leq \left(\sup_{a\in A, \|\widehat{a}\|_2\leq 1} \|a^* \eta'\|_2 \right)\|\eta\|_2 
\leq \lambda \|\eta\|_2
\end{align*}
for some $\lambda>0$ depending only on $\eta'$ as $\eta'\in D(\sb{A}K)$. Now if $\eta\notin D(\sb{A}K)$, take $\eta_n\in D(\sb{A}K)$ with $\eta_n\to \eta$ in $\|\cdot\|_2$, and define 
$$
\sb{A}\langle \eta' ,\eta\rangle = \lim_n {\sb{A}\langle} \eta',\eta_n\rangle
$$
which exists by the above estimate. Now $\langle a \eta,\eta'\rangle_K = \langle \widehat{a}, {\sb{A}\langle} \eta',\eta\rangle\rangle_{L^2(A)}$ for all $a\in A$ by construction.
\end{proof}

\begin{cor}
Each $\eta \in {\sb{A}K}$ gives a closable operator $R(\eta)^0\colon \widehat{A}\to {\sb{A}K}$ by $\widehat{a}\mapsto a\eta$. A similar result holds for right modules.
\end{cor}
\begin{proof}
We need only show its adjoint is densely defined. If $\eta'\in D(\sb{A}K)$, then 
$$
\langle R(\eta)^0 \widehat{a} , \eta'\rangle_K 
= \langle a\eta,\eta'\rangle_K 
= \langle \widehat{A} , {\sb{A}\langle} \eta',\eta\rangle \rangle_{L^2(A)}
$$
by Lemma \ref{lem:RN}, and the result follows as $D(\sb{A}K)$ is dense in $K$.
\end{proof}

\begin{cor}\label{cor:UnboundedLR}
Each $\eta\in {\sb{A}K}$ gives a closable unbounded operator $R_\eta^0\colon D(H_A)\to H\otimes_A K$ by $\xi\mapsto \xi\otimes \eta$. A similar result holds for each $\xi'\in H_A$.
\end{cor}
\begin{proof}
Once again, we show its adjoint is densely defined. If $\xi'\in D(H_A)$ and $\eta'\in D(\sb{A}K)$, then by Lemma \ref{lem:RN},
\begin{align*}
\langle R_\eta^0 \xi , \xi'\otimes \eta' \rangle_{H\otimes_A K}
& = \langle \xi\otimes \eta, \xi'\otimes \eta'\rangle_{H\otimes_A K}
= \langle \langle \xi'|\xi\rangle_A \eta,\eta'\rangle_K 
= \langle \langle \xi'|\xi\rangle_A^{\widehat{\hs}}, {\sb{A}\langle} \eta',\eta\rangle^{\widehat{\hs}} \rangle_{L^2(A)}\\
& = \langle  L(\xi')^* \xi, {\sb{A}\langle} \eta',\eta\rangle^{\widehat{\hs}} \rangle_{L^2(A)}
= \langle \xi, L(\xi'){\sb{A}\langle} \eta',\eta\rangle^{\widehat{\hs}} \rangle_{H}.
\end{align*}
The result now follows as $D(H_A)\otimes_A D(\sb{A}K)$ is dense in $H\otimes_A K$.
\end{proof}

%%%%%%%%%%%%%%%%%%%%%%%%%%%%%%%%%%%%%%%%%%%%%%%%%%%%%
\subsection{Haagerup's extended positive cones and operator valued weights}\label{sec:MR534673}

For this subsection, $M$ is a von Neumann algebra acting on a Hilbert space $H$. 

\begin{defn}[Section 1 of \cite{MR534673}]
The extended positive cone of $M$, denoted $\widehat{M^+}$, is the set of weights on the predual of $M$, i.e., maps $m\colon M_*^+\to [0,\I]$ such that
\item[(1)] $m(\lambda \phi+\psi) = \lambda m(\phi)+m(\psi)$ for all $\lambda\geq 0$ and $\phi,\psi\in M_*^+$, and
\item[(2)] $m$ is lower semicontinuous.
\item
The extended positive cone has additional structure:
\item[$\bullet$] There is a natural inclusion $M^+\to\widehat{M^+}$ by $m\mapsto (\phi\mapsto \phi(m))$.
\item[$\bullet$] For $m\in \widehat{M^+}$ and $a\in M$, we define $a^*ma\in \widehat{M^+}$ by 
$$a^*ma(\phi)=m(a\phi a^*)=m(\phi(a^*\, \cdot\,a)).$$
We write $\lambda m$ for $\lambda^{1/2}m\lambda^{1/2}$ for $\lambda\geq 0$.
\item[$\bullet$] 
There is a natural partial ordering on $\widehat{M^+}$ given by
$m_1\leq m_2$ if $m_1(\phi)\leq m_2(\phi)$ for all $\phi\in M_*^+$. 
\item[$\bullet$] 
If $I$ is a directed set, we say $(m_i)_{i\in I}\subset\widehat{M^+}$ \underline{increases} to $m\in\widehat{M^+}$ if $i\leq j$ implies $m_i\leq m_j$ and $\sup_{i} m_i(\phi)= m(\phi)$ for all $\phi\in M_*^+$. Hence we can define the sum of elements of $\widehat{M^+}$ pointwise. 
\item[$\bullet$] Each $\phi\in M_*^+$ extends uniquely to a map $\widehat{M^+}\to[0,\I]$ by $\phi(m)=m(\phi)$.
\end{defn}

\begin{rem}[Section 1 of \cite{MR534673}]
There are equivalent definitions of $\widehat{M^+}$:
\item[$\bullet$] Given a projection $p\in P(M)$ and a densely-defined positive, self-adjoint operator $S$ in $K=pH$ affiliated with $M$, we can define 
\begin{equation}\label{eq:KS}
m_{(K,S)}(\omega_\xi) = 
\begin{cases}
\|S^{1/2}\xi\| & \text{ if }\xi\in D(S^{1/2})\\
\I & \text{ else}
\end{cases}
\end{equation}
where $\omega_\xi = \langle \cdot\, \xi,\xi\rangle$. Conversely, given $m\in \widehat{M^+}$, there are unique $(K,S)$ such that Equation \eqref{eq:KS} holds. In the sequel, we will write $m=(K,S)$ when we use this bijective correspondence.
\item[$\bullet$] Each $m\in \widehat{M^+}$ has a unique spectral resolution
$$
m(\phi) = \int_0^\I \lambda d \phi(e_\lambda) + \I \phi(p)
$$
where $\{e_\lambda\}_{\lambda\in [0,\I)}$ are increasing family of projections in $M$ such that:
\be
\item[(1)] $\lambda\mapsto e_\lambda$ is strongly continuous from the right, and
\item[(2)] $p=1-\lim_{\lambda\to\I} e_\lambda$
\ee
Moreover, 
\begin{align*}
e_0 = 0 &\Longleftrightarrow m(\phi)>0 \text{ for all } \phi \in M_*^+\setminus\{0\}\\
p=0 &\Longleftrightarrow \set{\phi\in M_*^+}{m(\phi)<\I} \text{ is dense in } M_*^+.
\end{align*}
\item[$\bullet$] Every $m\in \widehat{M^+}$ is a pointwise limit of an increasing sequence of operators in $M^+$.
\item[$\bullet$] $\widehat{M^+}$ is the set of all $m\in \widehat{B(H)^+}$ affiliated to $M$ ($umu^*=m$ for all $u\in U(M')$).
\end{rem}

\begin{thm}[\cite{MR534673}, Proposition 1.11, Theorem 1.12]\label{thm:BilinearExtension}
Suppose $M$ is a semifinite von Neumann algebra with n.f.s. trace $\Tr_M$. For $x,y\in M^+$, let $\Tr_M(x\cdot y) =\Tr_M(x^{1/2}yx^{1/2})$. Then the map $(x,y)\mapsto \Tr_M(x\cdot y)$ has a unique extension to $\widehat{M^+}\times\widehat{M^+}$ such that
\item[$\bullet$] $\Tr_M(x\cdot y) = \Tr_M(y\cdot x)$ for all $x,y\in \widehat{M^+}$,
\item[$\bullet$] $\Tr_M$ is additive and homogeneous in both variables,
\item[$\bullet$] if $(x_i),(y_j)\subset \widehat{M^+}$ with $x_i\nearrow x$ and $y_j\nearrow y$, then $\Tr_M(x_i\cdot y_j)\nearrow \Tr_M(x\cdot y)$, and
\item[$\bullet$] $\Tr_M((a^*xa)\cdot y) = \Tr_M(x\cdot(aya^*))$ for all $x,y\in \widehat{M^+}$ and $a\in M$.
\item
Moreover
\item[$\bullet$] The map $x\mapsto \Tr(x\,\cdot\,)$ is a homogeneous, additive bijection from $\widehat{M^+}$ onto the set of normal weights of $M$,
\item[$\bullet$] $x\leq y\Longleftrightarrow \Tr(x\,\cdot\,)\leq \Tr(y\,\cdot\,)$ and  $x_i\nearrow x\Longleftrightarrow \Tr(x_i\,\cdot\,)\nearrow \Tr(x\,\cdot\,)$, and
\item[$\bullet$] If $x = \int_0^\I \lambda \, de_\lambda+\I p$, then $\Tr(x\,\cdot\,)$ is faithful if and only if $e_0=0$ and semifinite if and only if $p=0$.
\end{thm}

\begin{defn}[\cite{MR534673}, Definitions 2.1 and 2.2]
Let $M$ and $N$ be von Neumann algebras $N\subseteq M$. An operator valued weight from $M\to N$ is a map $T\colon M^+\to \widehat{N^+}$ which satisfies the following conditions:
\item[(1)] $T(\lambda x+ y) = \lambda T(x)+T(y)$ for all $\lambda \geq 0$ and $x,y\in M^+$, and
\item[(2)]  $T(a^*xa)=a^*T(x)a$ for all $x\in M^+$ and $a\in N$.
\item
As in the case of ordinary weights, we set
\begin{align*}
\n_T &= \set{x\in M}{T(x^*x)\in N^+}\text{ and }\\
\m_T &=\n_T^*\n_T=\spann\set{x^*y}{x,y\in\n_T}.
\end{align*}
\item Moreover, we say $T$ is:
\item[$\bullet$] \underline{normal} if  $x_i\nearrow x \Rightarrow T(x_i)\nearrow T(x)$ for all $x_i,x\in M^+$,
\item[$\bullet$] \underline{faithful} if $T(x^*x)=0 \Rightarrow x=0$ for all $x\in M^+$, and
\item[$\bullet$] \underline{semifinite} if $\n_T$ is $\sigma$-weakly dense in $M$.
\item
We will abbreviate normal, faithful, semifinite by the acronym n.f.s.
\end{defn}

\begin{rems}
\item[(1)] $T$ is a conditional expectation if and only if $T(1)=1$.
\item[(2)] If $T$ is normal, it has a unique extension to $\widehat{M^+}$ satisfying (1) and (2).
\item[(3)] $\n_T$ is a left-ideal and $\n_T,\m_T$ are algebraic $N-N$ bimodules. By polarization, $T$ extends to a map $T\colon\m_T\to N$, and $T(axb)=aT(x)b$ for all $x\in \m_T$ and $a,b\in N$.
\end{rems}

\begin{thm}[\cite{MR534673}, Theorem 2.7]\label{thm:Texists}
Given an inclusion $N\subseteq M$ of semifinite von Neumann algebras with n.f.s. traces $\Tr_N,\Tr_M$ respectively. Then there is a unique n.f.s. trace-preserving operator valued weight $T\colon M^+\to \widehat{N^+}$. Moreover, if $x\in M^+$, $T(x)$ is the unique element of $\widehat{N^+}$ such that
\begin{equation}\label{eq:ovwCondition}
\Tr_M(y\cdot x)=\Tr_N(y\cdot T(x))\text{ for all }y\in N^+
\end{equation}
(where we also write $\Tr_N$ for the unique extension of $\Tr_N$ to $\widehat{N^+}$).
\end{thm}

\begin{defn}
For $N\subseteq M$ an inclusion of von Neumann algebras, we write
\item[$\bullet$] $\cP(M,N)$ for the set of n.f.s. operator valued weights $M^+\to \widehat{N^+}$, and
\item[$\bullet$] $\cP_0(M,N)\subseteq \cP(M,N)$ for the set of operator valued weights whose restriction to $N'\cap M$ is semifinite.
\end{defn}

\begin{lem}[\cite{MR1622812}, Lemma 2.5 and Proposition 2.8, \cite{MR1297671}, Corollary 28]\label{lem:SemifiniteRestriction}
Let $N\subset M$ be an inclusion of semifinite von Neumann algebras.
\item[(1)] 
There is a unique central projection $z\in N'\cap M$ such that
\be
\item[$\bullet$] $\cP_0(pMp,pN)=\emptyset$ for all $p\in N'\cap M$, $p\leq (1-z)$ and
\item[$\bullet$] $\cP_0(zMz,zN)=\cP(zMz,zN)$.
\ee
Moreover, for all $T\in\cP(M,N)$,
\be
\item[$\bullet$] $(1-z)(N'\cap M)\cap \m_T=\{0\}$, and
\item[$\bullet$] $T|_{z(N'\cap M)}$ is semifinite.
\ee
\item[(2)] 
If $\cP_0(M,N)\neq \emptyset$ and $\cP_0(N',M')\neq \emptyset$, then $N'\cap M$ is a direct sum of type $I$ factors, and $pN\subset pMp$ has finite index for every finite rank $p\in N'\cap M$.
\end{lem}

%%%%%%%%%%%%%%%%%%%%%%%%%%%%%%%%%%%%%%%%%%%%%%%%%%
\subsection{Useful lemmata on extended positive cones}\label{sec:LemmataPositiveCones}

For this subsection, $M$ is a von Neumann algebra acting on a Hilbert space $H$. 

\begin{lem}\label{lem:parallelogram}
For $m\in \widehat{M^+}$ and $\eta,\xi\in H$, the parallelogram identity holds:
$$
m(\omega_{\eta+\xi})+m(\omega_{\eta-\xi})= 2m(\omega_{\eta})+2m(\omega_{\xi}).
$$
\end{lem}
\begin{proof}
Take $(x_i)\subset M^+$ with $x_i$ increasing to $m$. Then
\begin{align*}
m(\omega_{\eta+\xi})+m(\omega_{\eta-\xi})
& = \sup_{i,j} \bigg( x_i(\omega_{\eta+\xi})+x_j(\omega_{\eta-\xi}) \bigg)\\
& \leq  \sup_{i,j} \bigg(\sup_{k\geq i,j}\bigg(x_k(\omega_{\eta+\xi})+x_k(\omega_{\eta-\xi})  \bigg)\bigg)\\
& =  \sup_{i,j} \bigg(\sup_{k\geq i,j}\bigg( 2x_k(\omega_{\eta})+2x_k(\omega_{\xi})\bigg)\bigg)\\
& \leq   \sup_{i',j'} \bigg( 2x_{i'}(\omega_{\eta})+2x_{j'}(\omega_{\xi})\bigg)
= 2m(\omega_{\eta})+2m(\omega_{\xi}).
\end{align*}
The other inequality is proved similarly.
\end{proof}

\begin{lem}\label{lem:VectorStates}
\item[(1)]  $m_1\leq m_2$ if and only if $m_1(\omega_\xi)\leq m_2(\omega_\xi)$ for all $\xi\in H$.
\item[(2)] $(m_i)_{i\in I}$ increases to $m$ if and only if $i\leq j$ implies $m_i\leq m_j$ and $\sup_i m_i(\omega_\xi) = m(\omega_\xi)$ for all $\xi\in H$.
\item[(3)] If $(m_i)_{i\in I}$ increases to $m$ and $a\in M^+$, then $a^*m_i a$ increases to $a^*ma$.
\end{lem}
\begin{proof}
First, note every $\phi\in M_*^+$ is a sum of functionals $\omega_{\xi_k} = \langle \cdot\, \xi_k,\xi_k\rangle$ for $\xi_i\in H$.
\item[(1)] Follows immediately by lower semicontinuity of $m\in \widehat{M^+}$.
\item[(2)] Suppose $\phi = \sum_k \omega_{\xi_k}$. By lower semicontinuity,  
\begin{align*}
m(\phi) 
& = \sum_{k} m(\omega_{\xi_k}) 
= \sum_k \sup_i m_i(\omega_{\xi_k})\\
& \geq \sup_i  \sum_km_i(\omega_{\xi_k}) 
= \sup_i m_i\left(\sum_k \omega_k\right)
= \sup_i m_i(\phi).
\end{align*}
There are two cases:
\itt{Case 1} Suppose $m(\phi)=\I$. Then there is a $\varepsilon>0$ such that $\sup_i m_i(\omega_{\xi_k})>\varepsilon$ for infinitely many $k$, say $(k_n)$. Let $N>0$, and let $M>0$ such that $M\varepsilon>N$. Choose $j_1\in I$ such that $i\geq j_1$ implies $m_i(\omega_{k_1})>\varepsilon$. For $n=2,\dots,M$, inductively choose $j_n> j_{n-1}$ such that $i\geq j_n$ implies $m_i(\omega_{k_n})>\varepsilon$. Then for all $i>j_M$, 
$$
\sum_k m_i(\omega_{\xi_k}) \geq \sum_{n=1}^M m_i(\omega_{\xi_{k_n}}) \geq \sum_{n=1}^M \varepsilon = M\varepsilon>N.
$$
Since $N$ was arbitrary, we must have 
$$\sup_i m_i(\phi)=\sup_i m_i \left( \omega_k\right)=\sup_i \sum_k m_i(\omega_k)=\I.$$
\itt{Case 2} Suppose $m(\phi)<\I$. Let $\varepsilon>0$. Then there is an $N\in\N$ such that $\sum_{k>N} m(\omega_{\xi_k})<\varepsilon$. Now as in the proof of Lemma \ref{lem:parallelogram},
$$
m(\phi)-\varepsilon  <  \sum_{k=1}^N \sup_i m_i(\omega_{\xi_k})= \sup_i  \sum_{k=1}^N m_i(\omega_{\xi_k})\leq \sup_i \sum_k m_i(\omega_k) =\sup_i m_i(\phi),
$$
and the result follows as $\varepsilon$ was arbitrary.
\item[(3)] We use (2). Let $\xi\in H$.
\begin{align*}
a^*m_i a (\omega_\xi) &= m_i(\omega_{a\xi})\leq m_j(\omega_{a\xi})= a^* m_j a (\omega_\xi)\text{ for all } i\leq j \text{ and}\\
\sup_i a^*m_ia(\omega_\xi) &= \sup_i m_i(\omega_{a\xi})=m(\omega_{a\xi}) = a^* m a (\omega_\xi).
\end{align*}
\end{proof}

\begin{rem}\label{rem:sup}
Suppose $(x_i)_{i\in I},(y_i)_{i\in I}\subset M^+$ are directed families and $\lambda\geq 0$. Then by Lemma \ref{lem:VectorStates} and techniques similar to those used in the proof of Lemma \ref{lem:parallelogram},
$$
\sup_{i} (\lambda x_i+y_i) = \lambda \sup_i x_i + \sup_j y_j.
$$
\end{rem}

\begin{lem}\label{lem:QuadraticForm}
Suppose $F\subset \widehat{M^+}$ is a directed family, i.e., if $x,y\in F$, then there is a $z\in F$ with $z\geq x$ and $z\geq y$. Then there is a unique $m_F=(K_F,S_F)\in \widehat{M^+}$ with $K_F=\overline{\Dom(S_F^{1/2})}$ such that
\begin{align*}
m_F(\omega_\xi)=\langle S_F^{1/2}\xi,S_F^{1/2}\xi\rangle =\sup_{x\in F} x(\omega_\xi)  \text{ for all }\\
\xi \in \Dom(S_F^{1/2})=\set{\xi \in H}{ \sup_{x\in F} x(\omega_\xi)<\I}.
\end{align*}
We denote $m_F$ by $\sup_{x\in F} x$.
\end{lem}
\begin{proof}
As in \cite{MR534673, MR561983,MR1943006}, one checks that the extended quadratic form $s_F\colon H\to [0_\R,\I_\R]$ given by $s_F(\xi) = \sup_{x\in F}x(\omega_\xi)$ satisfies
\be
\item[(1)] $s_F(\lambda \xi)= |\lambda|^2 s_F(\xi)$,
\item[(2)] $s_F(\eta+\xi)+s_F(\eta-\xi) = 2s_F(\eta)+2s_F(\xi)$,
\item[(3)] $s_F$ is lower semicontinuous, and
\item[(4)] $s_F(u\xi)=s_F(\xi)$ for all $u\in M'$.
\ee
(1) and (4) are trivial. (3) follows as sups of lower semicontinuous maps are lower semicontinuous. (2) is similar to the proof of Lemma \ref{lem:parallelogram}.
\end{proof}

\begin{defn}\label{defn:UnboundedRN}
Suppose $M$ is a semifinite von Neumann algebra with n.f.s. trace $\Tr_M$ acting on the right of $H$. Let $\xi\in D(H_M)$, and suppose $(x_i)\in (M'\cap B(H))^+$ with $x_i\nearrow x\in \widehat{(M'\cap B(H))^+}$. Then each $L(\xi)^*x_iL(\xi)\in M^+$ as it commutes with the right $M$-action on $L^2(M,\Tr_M)$, so we define
$$
L(\xi)^*xL(\xi)\ = \sup_i L(\xi)^*x_iL(\xi)\in \widehat{M^+}.
$$
Note that if $\kappa\in L^2(M,\Tr_M)$, then
$$
\bigg(L(\xi)^*xL(\xi)\bigg)(\omega_\kappa)=\sup_i \bigg(L(\xi)^*x_iL(\xi)\bigg)(\omega_\kappa) = \sup_i x_i(\omega_{\xi\otimes \kappa})=x(\omega_{\xi\otimes \kappa}),
$$
which is independent of the choice of $(x_i)$. Hence $L(\xi)^*xL(\xi)$ is well-defined by Lemma \ref{lem:VectorStates}. Similarly, we may define operators of the form $R(\eta)^*yR(\eta)$, $L_\xi^*xL_\xi$, and $R_\eta^*yR_\eta$.
\end{defn}

%%%%%%%%%%%%%%%%%%%%%%%%%%%%%%%%%%%%%%%%%%%%%%%%%%
\section{Planar calculus for bimodules}
For this section, let $A$ be a $II_1$-factor, and let $\sb{A}H_A$ be an $A-A$ Hilbert bimodule, i.e., $H$ has commuting actions of $A$ and $A\op$.

%%%%%%%%%%%%%%%%%%%%%%%%%%%%%%%%%%%%%%%%%%%%%%%%%%%%%
\subsection{Centralizer algebras, central $L^2$-vectors, and canonical maps}\label{sec:TowersOfBimodules}

\begin{defn}
For an $A-A$ bimodule $K$ (algebraic or Hilbert), we define 
$$A'\cap K=\set{\xi\in K}{a\xi = \xi a\text{ for all }a\in A}.$$
\end{defn}

\begin{nota}
For $n\geq 0$, let
\item[$\bullet$] $H^n = \bigotimes^n_A H$, with the convention that $H^0=L^2(A)$,
\item[$\bullet$] $B^n=D(\sb{A}H^n)\cap D(H^n_A)$, which is dense in $H^n$ by Lemma 1.2.2 of \cite{correspondences}. We also use the convention $B=B^1$. Note $B^0=A$.
\item[$\bullet$] $\{\alpha\}\subset B$ be an $\sb{A}H$ basis (possible due to the density of $B$ in $H$), with 
$$\{\alpha^n\}=\set{\alpha_{1}\otimes\cdots\otimes\alpha_{n} }{\alpha_{i}\in\{\alpha\}\text{ for all } i=1,\dots,n}\subset B^n$$ 
the corresponding $\sb{A}H^n$ basis (as $R_{\alpha_{1}\otimes\cdots\otimes \alpha_{n}} =R_{\alpha_{1}}\cdots R_{\alpha_{n}}$). We let $\{\beta\}\subset B$ be an $H_A$ basis, with $\{{\beta}^n\}\subset B^n$ the corresponding $H_A^n$ basis.
\item[$\bullet$] (central $L^2$-vectors) $P_n = A'\cap H^n$. Note $P_0=A'\cap L^2(A) = \C\widehat{1}$.
\item[$\bullet$] $C_n = (A\op)'\cap B(H^n)$ (the commutant of the right $A$-action on $H^n$) with canonical trace $\Tr_n=\sum_{{\beta}^n} \langle \, \cdot \, {\beta}^n,{\beta}^n\rangle$,
\item[$\bullet$] $C_n\op = A'\cap B(H^n)$ with canonical trace $\Tr_n\op=\sum_{{\alpha}^n} \langle \, \cdot \, {\alpha}^n,{\alpha}^n\rangle$,
\item[$\bullet$] (centralizer algebras) $Q_n = C_n\cap C_n\op$.
\end{nota}

\begin{rem}\label{rem:central}
Note that $A\subset C_n$ and $A\op\subset C_n\op$.
\end{rem}

\begin{defn}\label{defn:symmetric}
$H$ is called \underline{symmetric} if there is a conjugate-linear isomorphism $J\colon H\to H$ such that $J(a\xi b)= b^* (J\xi) a^*$ for all $a,b\in A$ and $\xi\in H$. 
\end{defn}

\begin{rem}\label{rem:symmetric}
If $H$ is symmetric, then for $n\geq 1$, $H^n$ is symmetric with conjugate-linear isomorphism $J_n\colon H^n\to H^n$ given by the extension of
$$
J_n( \xi_1\otimes \cdots \otimes \xi_n)=(J\xi_1)\otimes\cdots \otimes (J\xi_n).
$$
for $\xi_i\in B$ for all $i$. Note that $J_nAJ_n=A\op$, $J_n C_nJ_n=C_n\op$, and $J_n B^n=B^n$. On $B(H^n)$, we define $j_n$ by $j_n(x)=J_nx^*J_n$. Note that $j_n^2=\id$ and $\Tr_n=\Tr_n\op\circ j_n$. 

If $H$ is not symmetric, then in general, $C_n\op$ is not the opposite algebra of $C_n$, e.g. $\sb{R\otimes 1} L^2(R\otimes R)_{R\otimes R}$ where $R$ is the hyperfinite $II_1$-factor.
\end{rem}

\begin{rem}\label{rem:StayBounded}
It is clear that $B^n$ is an $A-A$ bimodule. If $\eta\in B^n$ and $c\in C_n$, then $c\xi\in D(H^n_A)$, but in general, $c\xi\notin D(\sb{A}H^n)$. However, if $c\in Q_n$, then clearly $c\xi\in B^n$.
\end{rem}

\begin{prop}
\label{prop:include}
We have natural inclusions:
\begin{align*}
&i_n\colon C_n\to C_{n+1} \text{ by } x\mapsto x \otimes_A \id_H =  (\eta\otimes \xi\mapsto (x\eta)\otimes \xi \text{ for } \eta\in B^n \text{ and }\xi\in B)\text{ and}\\
&i_n\op\colon C_n\op\to C_{n+1}\op \text{ by } y\mapsto \id_H \otimes_A y =  (\xi\otimes \eta\mapsto \xi\otimes (y \eta) \text{ for } \xi\in B \text{ and }\eta\in B^{n}).
\end{align*}
Both maps include $Q_n\to Q_{n+1}$.
\end{prop}
\begin{proof}
If $z\in Q_n$, then $i_n(z)\in Q_{n+1}$ as for all $a,b\in A$,
$$
(z\otimes_A \id_H)[a(\xi\otimes \eta)b] = (z(a\xi))\otimes (\eta b)=(a(z\xi))\otimes (\eta b)=a[(z\eta)\otimes \xi]b.
$$
The result is similar for $i_n\op$.
\end{proof}

\begin{prop}\label{prop:iFormula}
If $x\in C_n$, then $i_n(x)=\sum_\alpha R_\alpha x R_\alpha^*$. If $y\in C_n\op$, then $i_n\op(y)=\sum_\beta L_\beta y L_\beta^*$.
\end{prop}
\begin{proof}
We prove the first statement. If $\xi_1,\dots,\xi_{n+1}\in B$, we have
\begin{align*}
\left(\sum_\alpha R_\alpha x R_\alpha^*\right) \xi_1\otimes\cdots \otimes\xi_n
& = \sum_\alpha R_\alpha x( \xi_1\otimes\cdots \otimes \xi_{n-1} {\sb{A}\langle} \xi_n,\alpha\rangle)\\
& = \sum_\alpha  \big(x(\xi_1\otimes\cdots \otimes \xi_{n-1}{\sb{A}\langle} \xi_n,\alpha\rangle)\otimes \alpha \big)\\
& = \sum_\alpha \big(x(\xi_1\otimes\cdots \otimes \xi_{n-1})\big)\otimes{\sb{A}\langle} \xi_n,\alpha\rangle \alpha\\
& = [x(\xi_1\otimes\cdots \otimes \xi_{n-1})]\otimes \xi_n
= i_n(x) (\xi_1\otimes\cdots\otimes \xi_n).
\end{align*}
\end{proof}

\begin{rem}\label{rem:FiberProduct}
By Definition \ref{defn:FiberProduct}, $(C_k\otimes_A \id_{n-k})'\cap B(H^n)=\id_k\otimes_A C_{n-k}\op$.
\end{rem}

\begin{lem}\label{lem:AlsoClosable}
Suppose $\xi\in H^n$ and $y\in (C_{n+1}\op)^+$. Recall the operator $R_\xi^0\colon B\to H^{n+1}$ by $\eta\mapsto \eta\otimes\xi$ is closable by Corollary \ref{cor:UnboundedLR}. Then $y^{1/2}R_\xi^0\colon B\to H^{n+1}$ is also closable.
\end{lem}
\begin{proof}
Let $p$ be the range/kernel perp projection of $y^{1/2}$. By the spectral theorem, there are projections $p_k\in C_{n+1}\op$ such that $y^{1/2}p_k=p_ky^{1/2}$ is invertible on $p_kH^{n+1}$ and $p_k\nearrow p$ (strongly). Fix $k\geq 0$.  Vectors of the form $\zeta=\sum_{i=1}^j \sigma_i \otimes \kappa_i \in p_kH^{n+1}$ where $\sigma_1,\dots,\sigma_j\in B$ and $\kappa_1,\dots,\kappa_j\in B^n$ are dense in $p_k H^{n+1}$ by the density of $B\otimes_A B^n \subset H^{n+1}$. Then for such $\zeta$ and all $\eta\in B$,
$$
\langle y^{1/2}R_\xi^0 \eta, y^{-1/2}p_k \zeta\rangle=\sum_{i=1}^j \langle \eta\otimes \xi, \sigma_i\otimes\kappa_i\rangle = \sum_{i=1}^j \langle \eta, L_{\sigma_i} ({\sb{A}\langle} \kappa_i,\xi\rangle) \rangle=\left\langle \eta , \sum_{i=1}^j L_{\sigma_i} ({\sb{A}\langle} \kappa_i,\xi\rangle)\right\rangle
$$
(see Corollary \ref{cor:UnboundedLR}). Finally, the span of vectors of the form $y^{-1/2}p_k \zeta$ where $\zeta$ is as above and $k\geq 0$ is dense in $pH^{n+1}$. 
\end{proof}

The following proposition and its proof are similar to Theorem 3.2.26 and Proposition 3.2.27 of \cite{burns}.

\begin{prop}\label{prop:T}
Recall from Proposition \ref{prop:include} that $i_n(C_n)\subset C_{n+1}$ and $i_n\op (C_n\op)\subset C_{n+1}\op$. The unique trace-preserving operator valued weight 
$$T_{n+1}\colon (C_{n+1}^+,\Tr_{n+1})\to (\widehat{C_n^+} ,\Tr_n)\text{ is given by }x\mapsto \sum_\beta R_\beta^* x R_\beta.$$
The unique trace-preserving operator valued weight 
$$T_{n+1}\op\colon \left((C_{n+1}\op)^+,\Tr_{n+1}\op\right) \to \left(\widehat{(C_{n}\op)^+},\Tr_n\op\right)\text{ is given by }y\mapsto \sum_\alpha L_\alpha^* y L_\alpha.$$
In particular, $T_{n+1}$ and $T_{n+1}\op$ are independent of the choice of basis.
\end{prop}
\begin{proof}
We prove the result for the second statement. 

Suppose $y\in (C_{n+1}\op)^+$ and $\xi\in H^{n}$. 
By Lemma \ref{lem:AlsoClosable}, $y^{1/2}R_\xi^0$ is closable, so we set $S=(y^{1/2}R_\xi^0)^*\overline{y^{1/2} R_\xi^0}$, which is affiliated with $C_{1}\op$, and define $m_{S}\in \widehat{(C_{1}\op)^+}$ as in Equation \eqref{eq:KS} by
$$
m_{S}(\omega_\eta) = 
\begin{cases}
\| S^{1/2} \eta \| &\text{ if } \eta \in D(S^{1/2})\supset B\\
\I & \text{ else.}
\end{cases}
$$

Now we calculate that
\begin{align*}
\Tr_{1}\op (m_{S})
& = \sum_{\alpha} m_{S}(\omega_\alpha) 
= \sum_\alpha \|S^{1/2} \alpha\|^2_2
= \sum_\alpha \|y^{1/2} R_\xi^0 \alpha\|^2_2\\
& = \sum_\alpha \langle y (\alpha\otimes \xi) , (\alpha\otimes \xi)\rangle_{H^{n+1}}
= \left\langle \left( \sum_\alpha L_\alpha^* y L_\alpha\right) \xi,\xi\right\rangle_{H^n} = T_{n+1}\op(y)(\omega_\xi).
\end{align*}
As all elements of $B(H)^+_*$ are sums $\sum_{i} \omega_{\xi_i}$, $T_{n+1}\op$ is well-defined and independent of the choice of $\{\alpha\}$. 

Note that $T_{n+1}\op((C_{n+1}\op)^+)\subset \widehat{(C_{n}\op)^+}$ as if $y\in (C_{n+1}\op)^+$, $\xi\in H^{n}$, and $u\in U(A)$, then
\begin{align*}
\sum_\alpha L_\alpha^* y L_\alpha (\omega_{u\xi}) 
& = \sum_\alpha \langle y (\alpha\otimes u\xi), \alpha \otimes u\xi \rangle
 = \sum_\alpha \langle y (\alpha u\otimes \xi), \alpha u\otimes \xi \rangle \\
& = \sum_\alpha L_{\alpha u}^* y L_{\alpha u} (\omega_{\xi}) 
= \sum_\alpha L_\alpha^* y L_\alpha (\omega_{\xi }) 
\end{align*}
as $\{\alpha u\}$ is another $\sb{A}H$ basis by Lemma \ref{lem:UnitaryBases}.

Finally, if $x\in (C_n\op)^+$ and $y\in (C_{n+1}\op)^+$, then 
\begin{align*}
\Tr_{n+1}\op \left( [i_n\op(x^{1/2})] y [i_n\op(x^{1/2})]\right)
& =  \sum_{{\alpha}^{n+1}} \left\langle [i_n\op(x^{1/2})] y [i_n\op(x^{1/2})]{\alpha}^{n+1},{\alpha}^{n+1}\right\rangle\\
& =  \sum_{\alpha,{\alpha}^{n}} \left\langle y (\alpha\otimes (x^{1/2}{\alpha}^{n})),(\alpha\otimes (x^{1/2}{\alpha}^{n}))\right\rangle \\
& =  \sum_{{\alpha}^{n}} \left\langle \sum_\alpha L_\alpha^* yL_\alpha  (x^{1/2}{\alpha}^{n}), (x^{1/2}{\alpha}^{n})\right\rangle \\
& = \Tr_n\op\left(x^{1/2} T_{n+1}\op(y) x^{1/2} \right),
\end{align*}
so $T_{n+1}\op$ is the unique trace-preserving operator valued weight by Equation \eqref{eq:ovwCondition} in Theorem \ref{thm:Texists}.
\end{proof}

\begin{rem}\label{rem:TForZ}
If $z\in Q_{n+1}^+$, then $T_{n+1}\op(z)\in \widehat{Q_{n}^+}$ as if $\xi\in H^{n}$ and $u\in U(A)$,
$$
\sum_\alpha L_\alpha^* z L_\alpha (\omega_{\xi u}) 
= \sum_\alpha \langle z (\alpha\otimes \xi u), \alpha \otimes \xi u \rangle
= \sum_\alpha \langle (z(\alpha \otimes\xi))uu^*, \alpha \otimes \xi\rangle
= \sum_\alpha L_{\alpha}^* z L_{\alpha} (\omega_{\xi}).
$$
A similar result holds for $T_{n+1}$.
\end{rem}

\begin{cor}\label{cor:Z1Traces}
If $z\in Q_1^+$, then $\sum_\alpha L(\alpha)^* z L(\alpha) = \Tr_1\op(z) 1_{L^2(A)}$. Similarly, $\sum_\alpha R(\beta)^* z R(\beta) = \Tr_1(z) 1_{L^2(A)}$.
\end{cor}
\begin{proof}
We prove the first formula. First, $\sum_\alpha L(\alpha)^* z L(\alpha) \in \widehat{Q_0^+}=[0,\I]$. Now
$$
\left(\sum_\alpha L(\alpha)^* z L(\alpha)\right) (\omega_{\widehat{1}})
= \sum_\alpha\langle  L(\alpha)^* z L(\alpha) \widehat{1}, \widehat{1} \rangle 
= \sum_\alpha \langle z \alpha, \alpha \rangle 
= \Tr_1\op(z).
$$
\end{proof}

\begin{prop}\label{prop:CrossT}
The unique trace-preserving operator valued weight 
$$\widetilde{T_{n+1}}\colon (Q_{n+1}^+,\Tr_{n+1})\to (i_n\op(\widehat{Q_n^+}) ,\Tr_n)\text{ is given by }x\mapsto \sum_\beta L_\beta^* x L_\beta.$$
The unique trace-preserving operator valued weight 
$$\widetilde{T_{n+1}\op}\colon \left(Q_n^+,\Tr_{n+1}\op\right) \to \left(i_n(\widehat{Q_{n}^+}),\Tr_n\op\right)\text{ is given by }y\mapsto \sum_\alpha R_\alpha^* y R_\alpha.$$
In particular, $\widetilde{T_{n+1}}$ and $\widetilde{T_{n+1}\op}$ are independent of the choice of basis.
\end{prop}
\begin{proof}
Similar to the proof of Proposition \ref{prop:T} using Remark \ref{rem:TForZ}. Note that if $u\in U(A)$, then $\{u\alpha\},\{\beta u\}$ are also $\sb{A}H,H_A$-bases respectively by Lemma \ref{lem:UnitaryBases}.
\end{proof}

%%%%%%%%%%%%%%%%%%%%%%%%%%%%%%%%%%%%%%%%%%%%%%%%%
\subsection{Planar algebra over extended positive cones of centralizer algebras}\label{sec:PAoverPositiveCones}

The following theorem is necessary to show the planar calculus is well-defined.

\begin{thm}\label{thm:BPRelations}
The following relations hold among the maps $T_n, T_n\op, \otimes_A,\Tr_n,\Tr_n\op$ for $m,n\geq 1$ (compare with Theorem \ref{thm:relations}):
\item[(1)] $T_n T_{n+1}\op (z) = T_{n}\op T_{n+1}(z)$ for all $z\in \widehat{Q_{n+1}^+}$,
\item[(2)] $z_1\otimes_A( z_2\otimes_A z_3) = (z_1\otimes_A z_2)\otimes_A z_3$ for all $z_i\in \widehat{Q_{n_i}^+}$, $i=1,2,3$,
\item[(3)] $z_1\otimes_A (T_n z_2) = T_{m+n}(z_1\otimes z_2)$ and $(T_m\op z_1)\otimes_A z_2 = T_{m+n}\op (z_1\otimes z_2)$ for all $z_1\in \widehat{Q_{m}^+}$ and $z_2\in \widehat{Q_{n}^+}$, and
\item[(4)] $\Tr_n(z_1\cdot z_2)=\Tr_n(z_2\cdot z_1)$ and $\Tr_n\op(z_1\cdot z_2)=\Tr_n\op(z_2\cdot z_1)$ for all $z_1,z_2\in \widehat{Q_n^+}$.
\end{thm}
\begin{proof}
\item[(1)] For all $\xi \in H^{n}$ and $z\in \widehat{Q_{n+1}^+}$,
\begin{align*}
\big(T_n T_{n+1}\op (z) \big)(\omega_\xi)
& =  \left(\sum_{\beta} R_\beta^*\left(\sum_\alpha L_\alpha^* z L_\alpha \right)R_\beta\right)(\omega_\xi)
= \sum_{\alpha,\beta}  (R_\beta^* L_\alpha^* z L_\alpha R_\beta  )(\omega_\xi) \\
& =  \sum_{\alpha,\beta} z(\omega_{\alpha \otimes \xi\otimes\beta})
= \left(\sum_{\alpha} L_\alpha^*\left(\sum_\beta R_\beta^* z R_\beta \right)L_\alpha \right) (\omega_\xi)\\
& = \big(T_{n}\op T_{n+1}(z)\big)(\omega_\xi).
\end{align*}
\item[(2)] This is Corollary \ref{cor:associative}.
\item[(3)] Suppose $z_{1,j}\in Q_m^+$ increases to $z_1$ and $z_{2,k}\in Q_{n}^+$ increases to $z_2$. Then 
\begin{align*}
T_{m+n} ( z_{1,j}\otimes_A z_{2,k} ) 
& = \sum_\beta R_\beta^* ( z_{1,j}\otimes_A z_{2,k} ) R_\beta 
= \sum_\beta  z_{1,j}\otimes_A \left(R_\beta^* z_{2,k} R_\beta\right)\\
& =   z_{1,j}\otimes_A \left(\sum_\beta R_\beta^* z_{2,k} R_\beta\right)
= z_{1,j} \otimes_A (T_nz_{2,k})
\end{align*}
Now $T_nz_{2,k}$ increases to $T_n z_2$, and we are finished by Theorem \ref{thm:UnboundedIncrease}. The other equality is similar.
\item[(4)] This is Theorem \ref{thm:BilinearExtension}.
\end{proof}

\begin{cor}\label{cor:BPRelationsCor} The following relations also hold:
\item[(1)] $i_{n+1}i_{n}\op(z)=i_{n+1}\op i_{n}(z)$ for all $z\in \widehat{Q_n^+}$.
\item[(2)] $i_{m+n} (z_1\otimes_A z_n) = z_1 \otimes_A i_n(z_2)$ and $i_{m+n}\op(z_1\otimes_A z_2) = i_m\op(z_1)\otimes_A z_2$ for all $z_1\in \widehat{Q_m^+}$ and $z_2\in \widehat{Q_n^+}$,
\item[(3)] $i_{n-1}\op T_n (z)= T_{n+1} i_{n}\op (z)$ and $i_{n-1} T_n\op (z)= T_{n+1}\op i_n (z)$ for all $z\in \widehat{Q_n^+}$, and
\item[(4)] for all $z_1\in \widehat{Q_m^+}$ and $z_2\in \widehat{Q_n^+}$,
\begin{align*}
(T_{n+1}\circ\cdots\circ T_{m+n})(z_1\otimes_A z_2)&=\Tr_n(z_2)z_1\text{ and}\\
(T_{m+1}\circ\cdots\circ T_{m+n})(z_1\otimes_A z_2)&=\Tr_m\op(z_1)z_2.
\end{align*}
In particular, $\Tr_{m+n}(z_1\otimes z_2)=\Tr_m(z_1)\Tr_n(z_2)$ and $\Tr_{m+n}\op(z_1\otimes z_2)=\Tr_m\op(z_1)\Tr_n\op(z_2)$.
\end{cor}

\begin{defn}
The \underline{bimodule planar operad $\B\P$} is the operad of oriented, unshaded planar tangles (up to planar isotopy) generated by
$$
\idn{n},\, \OperatorValuedWeight{n}{},\,\OperatorValuedWeightOp{n}{},\,\TraceOfTwo{n}{}{},\,\TraceOfTwoOp{n}{}{},\text{ and }\tensor{m}{}{n}{}
$$
for $m,n\geq 0$ up to planar isotopy. (We draw all disks as boxes, suppress external disks, draw one thick string labelled $n$ for $n$ individual strings, and orient all strings upward unless otherwise specified.) Some topological properties of tangles in $\B\P$ are given in Appendix \ref{sec:BP}.

A \underline{$\B\P$-algebra (of extended positive cones)} $V_\bullet$ is a sequence $\{V_n\}_{n\geq 0}$ of extended positive cones (see Appendix \ref{sec:cones}) and an action by multilinear maps
$$Z\colon \B\P \to ML\{V_n\}$$
($Z$ is the \underline{partition function}) which is well-behaved under composition. 
\item
A $\B\P$-algebra is called:
\item[$\bullet$]  \underline{connected} if $V_0=[0_\R,\I_\R]$,
\item[$\bullet$] \underline{normal} if $Z(\cT)$ is normal for all $\cT\in\B\P$, and
\item[$\bullet$] \underline{self-dual} if $V_n$ is self-dual for all $n$, and for all annular tangles $\cT\in\B\P$, flipping it inside out gives the adjoint map (see Definition \ref{defn:ExtendedConeAdjoint}).
\end{defn}

\begin{thm}\label{thm:ConePA}
Given an $A-A$ bimodule $H$, the extended positive cones $\widehat{Q_n^+}$ form a unique connected, normal, self-dual $\B\P$-algebra $\widehat{Q_\bullet^+}$ such that:
\item[(1)] $\id_{H^n}=\id_n=\idn{n}$\,,
\item[(2)] $T_{n+1}(z)=\OperatorValuedWeight{n}{z}$\, and $T_{n+1}\op(z)=\OperatorValuedWeightOp{n}{z}$ for all $z\in \widehat{Q_{n+1}^+}$
\item[(3)] $z_1\otimes_A z_2 = \tensor{m}{z_1}{n}{z_2}$ (defined in Appendix \ref{sec:TensorUnbounded}) for all $z_1\in \widehat{Q_m^+}$ and $z_2\in \widehat{Q_n^+}$, and
\item[(4)] $\Tr_n(z_1\cdot z_2) = \TraceOfTwo{n}{z_1}{z_2}$ and $\Tr_n\op(z_1\cdot z_2) = \TraceOfTwoOp{n}{z_1}{z_2}$ for all $z_1,z_2\in \widehat{Q_n^+}$.
\item
Moreover, the following hold:
\item[(5)] $i_n(z) = \include{n}{z}$\, and $i_n\op(z) = \includeop{n}{z}$ for all $z\in \widehat{Q_n^+}$ and
\item[(6)] $\dim_{-A}(H)=T_1(1)=\ClosedLoop{1}$\, and $\dim_{A-}(H)=T_1\op(1)=\ClosedLoopOp{1}$.
\item
Note that the well-definition of the partition function $Z$ means that any closed diagram counts for a multiplicative factor in $\widehat{Q_0^+}=\widehat{Z(A)^+}=[0_\R,\I_\R].$
\item
We call $\widehat{Q_\bullet^+}$ the canonical $\B\P$-algebra associated to $H$.
\end{thm}
\begin{proof}
It suffices to show (1)-(4) uniquely determine the action of any tangle in $\B\P$. This follows from Theorem \ref{thm:BPRelations} and Appendix \ref{sec:BP}. Note that $\widehat{Q_\bullet^+}$ is connected since $\widehat{Q_0^+}=\widehat{Z(A)^+}=[0_\R,\I_\R]$, normal by Theorem \ref{thm:BilinearExtension} and Remark \ref{rem:normal}, and self-dual by Proposition \ref{prop:adjoint}. 
\end{proof}

\begin{rem}
Given some operad $\P$ of (shaded, unshaded, oriented, disoriented, etc.) planar tangles, it is not always possible to define an (extended) positive cone planar algebra over $\P$. For example, the rotation does not always map positive elements to positive elements in a subfactor planar algebra.
\end{rem}

%%%%%%%%%%%%%%%%%%%%%%%%%%%%%%%%%%%%%%%%%%%%%%%%%
\subsection{Graded algebra of central $L^2$-vectors}\label{sec:CentralVectors}

\begin{lem}\label{lem:CentralBounded}
Suppose $K$ is a Hilbert $A-A$ bimodule. Then $A'\cap K \subseteq D(\sb{A}K)\cap D(K_A)$.
\end{lem}
\begin{proof}
Suppose $\zeta\in A'\cap K$, $\zeta\neq 0$. Define $\varphi\colon A_+\to \C$ by $a\mapsto \langle a\zeta,\zeta\rangle$. Note that $\varphi$ is traicial as
$$
\varphi(a^*a)=\langle a^*a\zeta,\zeta\rangle = \langle a^* \zeta a,\zeta\rangle =  \langle a^* \zeta ,\zeta a^*\rangle =  \langle a^* \zeta ,a^* \zeta\rangle=  \langle aa^* \zeta ,\zeta\rangle=\varphi(aa^*).
$$
Hence there is a $\lambda\geq 0$ such that $\varphi=\lambda \tr_A$ by the uniqueness of the trace on a $II_1$-factor. Now for all $a\in A$,
$$
\| a \zeta\|^2_2 =\|\zeta a\|^2_2= \varphi(a^*a)= \lambda \tr_A(a^*a)=\lambda \|a\|^2_2,
$$
and $\zeta$ is left and right $A$-bounded.
\end{proof}

\begin{rem}
In the sequel, we will confuse elements $\zeta\in P_n$ and the operators $L(\zeta)=R(\zeta)\colon L^2(A)\to H^n$. We will omit $R(\zeta)$ and only write $L(\zeta)$.
\end{rem}

\begin{thm}
Representing elements $\zeta\in P_n$ by boxes with $n$ strings emanating from the top, 
$$
\zeta,L(\zeta)=\Pn{n}{\zeta},
$$
the $P_n$'s form a graded algebra $P_\bullet$ in the sense of  \cite{MR2732052} where the graded multiplication is given by relative tensor product (over $A$) of central vectors. We denote the product of $\zeta_m\in P_m$ and $\zeta_n\in P_n$ by
$$
\zeta_m\otimes \zeta_n = \TensorPn{m}{\zeta_m}{n}{\zeta_n}\in P_{m+n}.
$$
\end{thm}

\begin{rem}
If $z\in Q_n$ and $\zeta\in P_n$, then $z\zeta\in P_n$, which we denote as:
$$
z\zeta,L(z\zeta)=\ActOnPn{z}{\zeta}.
$$
The dual version of these diagrams denotes the functionals
$\langle \,\cdot\,,\zeta\rangle,L(\zeta)^*=\PnStar{n}{\zeta}$.
The inner product $\langle \,\cdot\,,\,\cdot\,\rangle\colon P_n\times P_n^*\to \C$ is given by $\langle \xi,\zeta\rangle = \InnerProduct{\zeta}{\xi}$.
\end{rem}

%%%%%%%%%%%%%%%%%%%%%%%%%%%%%%%%%%%%%%%%%%%%%%%%%
\subsection{Compatibility}\label{sec:compatible}

We now show how $\widehat{Q_\bullet^+}$ and $P_\bullet$ are compatible.

\begin{lem}\label{lem:ZetaRelations} 
\item[(1)] If $\zeta\in P_n$ and $\xi\in B^n$, then $\sb{A}\langle \zeta,\xi\rangle=\langle \xi|\zeta\rangle_A$.
\item[(2)] If $\zeta,\xi\in P_n$, $\sb{A}\langle \zeta,\xi\rangle=\langle \xi|\zeta\rangle_A=\langle \zeta,\xi\rangle 1_{L^2(A)}\in \C 1_{L^2(A)}$.
\item[(3)] For $\zeta\in P_n$, $L(\zeta)L(\zeta)^*=R(\zeta)R(\zeta)^*\in Q_n^+$. We denote the common operator as:
$$
\CentralVectorOperator{n}{\zeta}{\zeta}\in Q_n^+.
$$
\item[(4)] If $\zeta\in P_n$ and $\|\zeta\|_2=1$, $L(\zeta)L(\zeta)^*|_{P_n}=p_\zeta$, the projection onto $\C\zeta$. 
\end{lem}
\begin{proof}
\item[(1)]
Suppose $a_1,a_2\in A$. Then 
\begin{align*}
\langle \sb{A}\langle \zeta,\xi\rangle \widehat{a_1},\widehat{a_2}\rangle
&=\langle J R(\zeta)^*R(\xi)J \widehat{a_1}, \widehat{a_2}\rangle
=\langle \widehat{a_2^*}, R(\zeta)^*R(\xi) \widehat{a_1^*}\rangle
=\langle a_2^*\zeta,a_1^*\xi\rangle\\
&=\langle \zeta a_2^*, a_1^*\xi\rangle
=\langle a_1\zeta, \xi a_2\rangle
=\langle \zeta a_1,\xi a_2\rangle
=\langle L(\zeta)\widehat{a_1},L(\xi)\widehat{a_2}\rangle\\
&=\langle \langle \xi|\zeta\rangle_A \widehat{a_1},\widehat{a_2}\rangle.
\end{align*}
\item[(2)]
Since $\zeta,\xi \in P_n$, for all $a,b,a_1,a_2\in A$,
\begin{align*}
\langle \langle \xi|\zeta\rangle_A (a\widehat{a_1}b),\widehat{a_2}\rangle
& = \langle \zeta aa_1b,\xi a_2\rangle 
= \langle \zeta a_1, \xi a^* a_2 b^*\rangle\\
& =\langle \langle \xi|\zeta\rangle_A \widehat{a_1}, a^* \widehat{a_2} b^*\rangle
=\langle a(\langle \xi|\zeta\rangle_A \widehat{a_1})b, \widehat{a_2} \rangle,
\end{align*}
so $\langle \xi|\zeta\rangle_A \in Z(A)=\C 1_{A}$. Now setting $a=b=a_1=a_2=1_A$ gives the result.
\item[(3)] 
For $\xi\in B^n$, by (1),
$$
L(\zeta)L(\zeta)^*\xi 
= \zeta \langle \zeta|\xi\rangle_A 
= \langle \zeta|\xi\rangle_A \zeta 
= {\sb{A}\langle} \xi,\zeta\rangle \zeta 
= R(\zeta)R(\zeta)^* \xi,
$$
so the two are equal on $H^n$. We have $C_n\ni L(\zeta)L(\zeta)^* = R(\zeta)R(\zeta)^*\in C_n\op$, so $L(\zeta)L(\zeta)^*\in Q_n^+$. 
\item[(4)] Trivial from (2) and (3).
\end{proof}

\begin{thm}\label{thm:CloseOffZetas}
Suppose $\zeta\in P_n$ and $z\in \widehat{Q_n^+}$.
\item[(1)] $L(\zeta)^* z L(\zeta)=z(\omega_\zeta)1_{L^2(A)}=R(\zeta)^*zR(\zeta)$. We denote this diagrammatically as:
$$
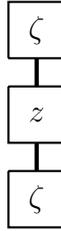
\TripleInnerProduct{\zeta}{z}{\zeta}
$$
\item[(2)] In the notation of Theorem \ref{thm:BilinearExtension},
\begin{align*}
z(\omega_\zeta)
& = \tr_A(L(\zeta)^* z L(\zeta))
= \Tr_n(L(\zeta)L(\zeta)^*\cdot z)\\
& = \tr_{A\op}(R(\zeta)^*zR(\zeta))
= \Tr_n\op(z\cdot R(\zeta)R(\zeta)^*).
\end{align*}
In diagrams:
$$
\TripleInnerProduct{\zeta}{z}{\zeta}= \TraceOfTwoBig{\zeta}{\zeta}{z}= \TraceOfTwoBigOp{\zeta}{\zeta}{z}.
$$
\end{thm}
\begin{proof}
\item[(1)]
We show the first equality. If $z\in Q_n^+$, this is just (2) of Lemma \ref{lem:ZetaRelations} with $\zeta_1=\zeta_2=z^{1/2}\zeta$. Now for $z\in \widehat{Q_n^+}$, pick $(z_m)\subset Q_n^+$ with $z_m\nearrow z$ to get
$$
L(\zeta)^*zL(\zeta)=\lim_{m\to\I} L(\zeta)^*z_mL(\zeta)=\lim_{m\to\I} z_m(\omega_\zeta)1_{L^2(A)}= z(\omega_\zeta)1_{L^2(A)}.
$$
The second equality is similar.
\item[(2)]
We show the second equality. We may assume $z\in Q_n^+$, after which we may take sups to get the full result. Then as $z^{1/2}\zeta\in P_n$, we have
\begin{align*}
\Tr_n(z\cdot L(\zeta)L(\zeta)^*) 
& = \Tr_n(z^{1/2} L(\zeta)L(\zeta)^* z^{1/2})
= \Tr_n(L(z^{1/2}\zeta)L(z^{1/2}\zeta)^*)\\
& = \tr_A(L(z^{1/2}\zeta)^*L(z^{1/2}\zeta))
= \tr_A(L(\zeta)^* zL(\zeta)).
\end{align*}
The other equality is similar.
\end{proof}

\begin{rem} If $a\in Q_n$, $z\in \widehat{Q_n^+}$, and $\zeta\in P_n$, 
$$
\WideTripleInnerProduct{\zeta}{a^*za}{\zeta}=(a^*za)(\omega_\zeta)=z(\omega_{a\zeta})=\TripleInnerProduct{a\zeta}{z}{a\zeta}.
$$
\end{rem}

\begin{cor}\label{cor:RemoveSubdiagram} If $\zeta_1\in P_m$, $\zeta_2\in P_n$, $z_1\in Q_m^+$, and $z_2\in Q_n^+$, then
$$
\BigInnerProduct{\zeta_1\otimes\zeta_2}{z_1\otimes_A z_2}{\zeta_1\otimes\zeta_2}
=\langle (z_1\otimes_A z_2)( \zeta_1\otimes \zeta_2),(\zeta_1\otimes\zeta_2)\rangle
=\langle z_1 \zeta_1,\zeta_1\rangle\langle z_2\zeta_2,\zeta_2\rangle
=\TripleInnerProduct{\zeta_1}{z_1}{\zeta_1}\TripleInnerProduct{\zeta_2}{z_2}{\zeta_2}.
$$
For $z_1\in \widehat{Q_m^+}$, and $z_2\in\widehat{Q_n^+}$, taking sups gives
$$
\BigInnerProduct{\zeta_1\otimes\zeta_2}{z_1\otimes_A z_2}{\zeta_1\otimes\zeta_2}
=(z_1\otimes_A z_2)(\omega_{\zeta_1\otimes\zeta_2})
=z_1(\omega_{\zeta_1})z_2(\omega_{\zeta_2})
=\TripleInnerProduct{\zeta_1}{z_1}{\zeta_1}\TripleInnerProduct{\zeta_2}{z_2}{\zeta_2}.
$$
\end{cor}

\begin{thm}[$P_\bullet$ acts on $\widehat{Q_\bullet^+}$]\label{thm:action}
Given a tangle $\cT\in \B\P$ with $2n$ boundary points and a $\zeta\in P_n$, we have 
$$
\TripleInnerProduct{\zeta}{\cT}{\zeta}:=\ev_{\omega_\zeta}\circ \cT \colon V_{i_1}\times\cdots\times V_{i_k}\to [0_\R,\I_\R].
$$
In this sense, we say $P_\bullet$ \underline{acts as weights} on $\widehat{Q_\bullet^+}$. By Theorems \ref{thm:BPRelations} and \ref{thm:CloseOffZetas} and Corollary \ref{cor:RemoveSubdiagram}, we may remove closed subdiagrams and multiply by the appropriate scalar in $[0,\I]$.
\end{thm}

\begin{rem}
If $A\subset (B,\tr_B)$ is an inclusion of $II_1$-factors and $H=L^2(B)$, then one can also define a shaded bimodule planar operad which works similarly to the above construction. This will be explored in a future paper.
\end{rem}

%%%%%%%%%%%%%%%%%%%%%%%%%%%%%%%%%%%%%%%%%%%%%%%%%%
\section{Extremality and rotations}\label{sec:ExtremalityRotations}
For this section, $A$ is a $II_1$-factor. Assume the notation of the last section.

%%%%%%%%%%%%%%%%%%%%%%%%%%%%%%%%%%%%%%%%%%%%%%%%%%
\subsection{Extremality}\label{sec:Extremality}

\begin{defn}
$H$ is \underline{approximately extremal with constant $\lambda>0$} if on $Q_1^+$,
$$\lambda^{-1} \Tr_1 \leq \Tr_1\op \leq \lambda\Tr_1.$$
$H$ is \underline{extremal} if $\Tr_1= \Tr_1\op $ on  $Q_1^+$.
\end{defn}

The following proposition is almost identical to Proposition 2.8 in \cite{MR1622812}.
\begin{prop}[Structure of $Q_n$] \label{prop:decompose} 
$Q_n = \a_{n}\oplus \b_{n}\oplus \b_{n}\op\oplus \c_{n}$ such that
\item[$\bullet$] $\a_{n}$ is a direct sum of type $I$ factors, and for each finite rank $p\in \a_{n}$, $pA\subset pC_np$ has finite index.
\item[$\bullet$] $\Tr_n|_{\a_{n}\oplus\b_{n}}$ and $\Tr_n\op|_{\a_{n}\oplus \b_{n}\op}$ are semifinite,
\item[$\bullet$] $\b_{n}\op\oplus\c_{n}\cap \m_{\Tr_n}=\{0\}=\b_{n}\oplus \c_{n}\cap \m_{\Tr_n\op}$, and
\item[$\bullet$] If $H^n$ is symmetric, then $j_n$ fixes $\a_{n},\c_{n}$ and $j_n(\b_{n})=\b_{n}\op$.
\end{prop}
\begin{proof}
By Lemma \ref{lem:SemifiniteRestriction}, let $z_{n},z_{n}\op\in Q_n$ be the unique central projections corresponding to $A\subset C_n$ and $A\op\subset C_n\op$. Set
\begin{align*}
\a_{n}&= z_{n}z_{n}\op Q_n & \b_{n}&= z_{n}(1-z_{n}\op)Q_n\\
\b_{n}\op &= (1-z_{n})z_{n}\op Q_n & \c_{n} &= (1-z_{n})(1-z_{n}\op) Q_n,
\end{align*}
and the rest follows immediately.
\end{proof}

\begin{prop}  Let $Q_1 = \a_1\oplus\b_1\oplus\b_1\op\oplus\c_1$ as in Proposition \ref{prop:decompose}. The following are equivalent:
\item[(1)] $H$ is approximately extremal with constant $\lambda>0$, and
\item[(2)] $\b_1=\b_1\op=\{0\}$ and there is a $\lambda>0$ such that on $Q_1^+\cap \a_1$,
$$\lambda^{-1} \Tr_1 \leq \Tr_1\op \leq \lambda\Tr_1.$$
\item A similar result holds for the extremal case.
\end{prop}
\begin{proof}
\itm{(1)\Rightarrow(2)} Suppose $H$ is approximately extremal. We show $\b_1=\{0\}$. As $\Tr_1|_{\a_1\oplus\b_1}$ is semifinite by Proposition \ref{prop:decompose}, we choose $z\in \b_1$ such that $z\geq 0$ and $z\in \m_{\Tr_1}$. Then  $z\in \m_{\Tr_1\op}$, but $\b_1\cap \m_{\Tr_1\op}=\{0\}$.  
Similarly $\b_1\op=\{0\}$.
\itm{(2)\Rightarrow(1)} 
$\Tr_1|_{\c_1\cap Q_1^+}=\Tr_1\op|_{\c_1\cap Q_1^+}=\I$.
\end{proof}

\begin{cor}
$H$ is extremal if and only if for each Hilbert $A-A$ bimodule $K\subset H$, the left and right von Neumann dimensions agree.
\end{cor}

\begin{rem}
If $H$ has a two-sided basis $\{\gamma\}$, then $H$ is extremal as 
$$
\Tr_1 = \sum_\gamma \langle \,\cdot\,\gamma, \gamma\rangle = \Tr_1\op.
$$
\end{rem}

\begin{rem}\label{lem:SwitchBasis}
If $H$ is approximately extremal and $z\in \widehat{Q_1^+}$, then there is a $\lambda >0$ such that
$$
\lambda^{-1} \sum_\beta z(\omega_\beta) \leq  \sum_\alpha z(\omega_\alpha) \leq \lambda \sum_\beta z(\omega_\beta).
$$
If $H$ is extremal, then the above holds with $\lambda = 1$.
\end{rem}

\begin{thm}\label{thm:extremal}
\item[(1)] If $H$ is (approximately) extremal (with constant $\lambda>0$), then $H^n$ is (approximately) extremal for all $n\geq 1$ (with constant $\lambda^n$). 
\item[(2)] If $H^n$ is (approximately) extremal for some $n\geq 1$, then $H$ is (approximately) extremal.
\end{thm}
\begin{proof}
We prove the extremal case, and the approximately extremal case is similar. 
\item[(1)]
We use strong induction on $n$. Suppose $H^1$ and $H^n$ are extremal. If $z\in Q_{n+1}^+$,
$$
\Tr_{n+1}(z)=\TraceEllipse{n+1}{z}=\ExtremalityOne{n}{z}=\ExtremalityTwo{n}{z}=\TraceOpEllipse{n+1}{z}=\Tr_{n+1}\op(z).
$$
Hence $H^{n+1}$ is extremal. 
\item[(2)] Suppose $H^n$ is extremal and $z\in Q_1^+$. Then $z\otimes_A \cdots \otimes_A z\in Q_n^+$. By the bimodule planar calculus,
\begin{align*}
\left(\begin{tikzpicture}[rectangular]
	\clip (1.25,1.5) --(-.55,1.5) -- (-.55,-1.5) -- (1.25,-1.5);
	\draw[] (.8,.5)--(1,.2)--(1.2,.5);
	\draw[] (0,.5) arc (180:0:.5cm) --(1,-.5) arc (0:-180:.5cm);
	\filldraw[thick, unshaded] (-.5,.5) --(-.5,-.5) -- (.5,-.5) -- (.5,.5)--(-.5,.5);
% 	\draw[ultra thick, unshaded] (.5,-1) ellipse (.65cm and .25cm);
	\node at (0,0) {$z$};
%	\node at (.5,-1.25) {{\scriptsize{$z$}}};
\end{tikzpicture}\right)^n
& =
\begin{tikzpicture}[rectangular]
	\clip (2.5,1.5) --(-3,1.5) -- (-3,-1.5) -- (2.5,-1.5);
	\draw[] (1.8,.5)--(2,.2)--(2.2,.5);
	\draw[] (-2,.5) .. controls ++(90:1.2cm) and ++(90:1.2cm) ..(2,.5)--(2,-.5) .. controls ++(270:1.2cm) and ++(270:1.2cm) ..(-2,-.5);
	\draw[] (.8,.5)--(1,.2)--(1.2,.5);
	\draw[] (0,.5) arc (180:0:.5cm) --(1,-.5) arc (0:-180:.5cm);
	\filldraw[thick, unshaded] (-2.5,.5) --(-2.5,-.5) -- (-1.5,-.5) -- (-1.5,.5)--(-2.5,.5);
	\filldraw[thick, unshaded] (-.5,.5) --(-.5,-.5) -- (.5,-.5) -- (.5,.5)--(-.5,.5);
% 	\draw[ultra thick, unshaded] (.5,-1) ellipse (.65cm and .25cm);
	\node at (0,0) {$z$};
	\node at (-2,0) {$z$};
	\node at (1.5,0) {{\scriptsize{$\cdots$}}};
	\node at (-1,0) {{\scriptsize{$\cdots$}}};
\end{tikzpicture}
=
\begin{tikzpicture}[rectangular]
	\clip (1.95,1.5) --(-1.5,1.5) -- (-1.5,-1.5) -- (1.95,-1.5);
	\draw[ultra thick] (1.5,.5)--(1.7,.2)--(1.9,.5);
	\draw[ultra thick] (.2,.5) arc (180:0:.75cm) --(1.7,-.5) arc (0:-180:.75cm);
	\filldraw[thick, unshaded] (-1.3,.5) --(-1.3,-.5) -- (1.3,-.5) -- (1.3,.5)--(-1.3,.5);
% 	\draw[ultra thick, unshaded] (.5,-1) ellipse (.65cm and .25cm);
	\node at (0,0) {{\scriptsize{$z\otimes_A\cdots\otimes_A z$}}};
	\node at (.4,-1.25) {{\scriptsize{$n$}}};
\end{tikzpicture}
\\
& =
\begin{tikzpicture}[rectangular]
	\clip (-1.95,1.5) --(1.5,1.5) -- (1.5,-1.5) -- (-1.95,-1.5);
	\draw[ultra thick] (-1.5,.5)--(-1.7,.2)--(-1.9,.5);
	\draw[ultra thick] (-.2,.5) arc (0:180:.75cm) --(-1.7,-.5) arc (-180:0:.75cm);
	\filldraw[thick, unshaded] (-1.3,.5) --(-1.3,-.5) -- (1.3,-.5) -- (1.3,.5)--(-1.3,.5);
% 	\draw[ultra thick, unshaded] (.5,-1) ellipse (.65cm and .25cm);
	\node at (0,0) {{\scriptsize{$z\otimes_A\cdots\otimes_A z$}}};
	\node at (-.4,-1.25) {{\scriptsize{$n$}}};
\end{tikzpicture}
=
\begin{tikzpicture}[rectangular]
	\clip (-2.5,1.5) --(3,1.5) -- (3,-1.5) -- (-2.5,-1.5);
	\draw[] (-1.8,.5)--(-2,.2)--(-2.2,.5);
	\draw[] (2,.5) .. controls ++(90:1.2cm) and ++(90:1.2cm) ..(-2,.5)--(-2,-.5) .. controls ++(270:1.2cm) and ++(270:1.2cm) ..(2,-.5);
	\draw[] (-.8,.5)--(-1,.2)--(-1.2,.5);
	\draw[] (0,.5) arc (0:180:.5cm) --(-1,-.5) arc (-180:0:.5cm);
	\filldraw[thick, unshaded] (2.5,.5) --(2.5,-.5) -- (1.5,-.5) -- (1.5,.5)--(2.5,.5);
	\filldraw[thick, unshaded] (-.5,.5) --(-.5,-.5) -- (.5,-.5) -- (.5,.5)--(-.5,.5);
% 	\draw[ultra thick, unshaded] (.5,-1) ellipse (.65cm and .25cm);
	\node at (0,0) {$z$};
	\node at (2,0) {$z$};
	\node at (-1.5,0) {{\scriptsize{$\cdots$}}};
	\node at (1,0) {{\scriptsize{$\cdots$}}};
\end{tikzpicture}
=
\left(\begin{tikzpicture}[rectangular]
	\clip (-1.25,1.5) --(.55,1.5) -- (.55,-1.5) -- (-1.25,-1.5);
	\draw[] (-.8,.5)--(-1,.2)--(-1.2,.5);
	\draw[] (0,.5) arc (0:180:.5cm) --(-1,-.5) arc (-180:0:.5cm);
	\filldraw[thick, unshaded] (-.5,.5) --(-.5,-.5) -- (.5,-.5) -- (.5,.5)--(-.5,.5);
% 	\draw[ultra thick, unshaded] (.5,-1) ellipse (.65cm and .25cm);
	\node at (0,0) {$z$};
%	\node at (.5,-1.25) {{\scriptsize{$z$}}};
\end{tikzpicture}\right)^n.
\end{align*}
In equations:
$$
\Tr_1(z)^n = \Tr_n (z\otimes_A \cdots \otimes_A z) = \Tr_n\op(z\otimes_A \cdots \otimes_A z)=\Tr_1\op(z)^n.
$$
Taking $n\thh$ roots gives the desired result.
\end{proof}

\begin{prop}
If $H$ is extremal and $z\in \widehat{Q_n^+}$, then $\sum_\beta R_\beta^* z R_\beta=\sum_\alpha R_\alpha^* zR_\alpha$ and $\sum_\alpha L_\alpha^* z L_\alpha = \sum_\beta L_\beta^*zL_\beta$.
\end{prop}
\begin{proof}
Immediate from Propositions \ref{prop:T} and \ref{prop:CrossT}.
\end{proof}

%%%%%%%%%%%%%%%%%%%%%%%%%%%%%%%%%%%%%%%%%%%%%%%%%%
\subsection{Rotations}\label{sec:Rotations}

\begin{defn}[Inspired by \cite{burns}]\label{defn:rotation}
A \underline{Burns rotation} is a map $\rho\colon P_n\to P_n$ such that for all $\zeta\in P_n$ and $b_1,\dots,b_n\in B$,
\begin{equation}\label{eq:rotation}
\langle \rho(\zeta),b_1\otimes\cdots\otimes b_n\rangle = \langle \zeta, b_2\otimes\cdots\otimes b_n\otimes b_1\rangle.
\end{equation}
An \underline{opposite Burns rotation} is defined similarly:
$$
\langle \rho\op(\zeta),b_1\otimes\cdots\otimes b_n\rangle = \langle \zeta, b_n\otimes b_1\otimes\cdots\otimes b_{n-1}\rangle.
$$
\end{defn}

\begin{rem}\label{rem:RotationInverse}
Note that if such a $\rho$ exists, it is unique, and $\rho^n = \id_{P_n}$. In this case, $\rho\op=\rho^{-1}$.
\end{rem}

\begin{thm}[Essentially due to \cite{burns}]
If $\rho=\sum_{\beta} L_\beta R_\beta^*$ converges strongly on $P_n$, then $\rho$ is a Burns rotation. Similarly, if $\rho\op=\sum_{\alpha} R_\alpha L_\alpha^*$ converges strongly on $P_n$, then $\rho\op$ is an opposite Burns rotation.
\end{thm}
\begin{proof}
We must show that $\rho$ preserves $P_n$ and that $\rho$ satisfies Equation \eqref{eq:rotation}. The latter follows from:
\begin{align*}
\langle \rho(\zeta),b_1\otimes\cdots\otimes b_n\rangle 
&=\sum_{\beta}\langle \zeta,  R_\beta L_\beta^* (b_1\otimes \cdots\otimes b_n )\rangle\\
&=\sum_{\beta}\langle \zeta, \langle \beta|b_1\rangle_A b_2\otimes \cdots\otimes b_n\otimes \beta \rangle\\
&=\sum_{\beta}\langle \langle \beta|b_1\rangle_A^*\zeta, b_2\otimes \cdots\otimes b_n\otimes \beta \rangle\\
&=\sum_{\beta}\langle \zeta \langle \beta|b_1\rangle_A^* , b_2\otimes \cdots\otimes b_n\otimes \beta \rangle\\
&=\sum_{\beta}\langle \zeta, b_2\otimes \cdots\otimes b_n\otimes \beta \langle \beta|b_1\rangle_A\rangle\\
&= \langle \zeta, b_2\otimes\cdots\otimes b_n\otimes b_1\rangle.
\end{align*}
Now $\rho$ is independent of the choice of $\{\beta\}$. In particular, for any $u\in U(A)$, $\{u\beta\}$ is an $H_A$-basis, and
$$
u \rho(\zeta) u^* = u \left(\sum_\beta L_\beta R_\beta^* \zeta\right) u^*= \sum_\beta L_{u\beta} R_{u\beta}^* \zeta= \rho(\zeta)\in P_n.
$$
\end{proof}

%%%%%%%%%%%%%%%%%%%%%%%%%%%%%%%%%%%%%%%%%%%%%%%%%%
\subsection{Diagrammatic representation of the Burns rotation}\label{sec:DiagramsOfRotation}
For this section, we assume the Burns rotation $\rho$ exists on $P_n$ for all $n\geq 0$.  Recall for all $k\geq 0$, $\rho^{-k}=(\rho\op)^k$.

\begin{nota}
For $\zeta\in P_{m+n}$, we denote $\rho^m(\zeta)=(\rho\op)^n(\zeta)\in P_{m+n}$ by moving $m$ strings around the bottom counterclockwise or by moving $n$ strings around the bottom clockwise:
$$
\Rotation{m}{n}{\zeta}=\rho^m(\zeta)=(\rho\op)^n(\zeta)=\RotationOp{n}{m}{\zeta}
$$
\end{nota}

\begin{prop}\label{prop:RotateCentral}
If $\eta\in P_m$ and $\xi\in P_n$, then $\rho^n (\eta\otimes \xi) = \xi\otimes\eta$:
$$
\Switch{n}{\xi}{m}{\eta}=\TensorPn{n}{\xi}{\,\, m}{\eta}.
$$
\end{prop}
\begin{proof}
Suppose $\alpha\in B^m$ and $\beta\in B^n$. Then by (1) of Lemma \ref{lem:ZetaRelations},
$$
\langle \rho^n(\eta\otimes\xi), \beta\otimes\alpha\rangle = \langle \eta\otimes \xi, \alpha\otimes \beta \rangle = \langle \langle \alpha| \eta\rangle_A \xi,\beta\rangle = \langle \xi{\sb{A}\langle} \eta,\alpha\rangle,\beta\rangle =\langle \xi\otimes \eta, \beta\otimes \alpha\rangle.
$$
\end{proof}

\begin{defn}
For $0\leq j< m$, define $\mu_j\colon P_m\times P_n \to P_{m+n}$ by $\mu_j (\eta,\xi)= \rho^{-j} ( \rho^j(\eta)\otimes \xi)$. We represent $\mu_{j}$ diagrammatically as follows:
$$
\mu_j(\eta,\xi)=\Insert{n}{\xi}{m-j}{j}{\eta}\,.
$$
Well-definition of this diagram relies on the following proposition.
\end{defn}

\begin{prop}\label{prop:insert}
The $\mu_i$'s are associative, i.e., if $\sigma\in P_\ell$, $\eta\in P_m$, and $\xi\in P_n$, and $i\leq \ell$, $j\leq m$, then
$$
\mu_i(\kappa,\mu_j(\eta,\xi))=\mu_{i+j}(\mu_i(\kappa,\eta),\xi).
$$
\end{prop}
\begin{proof}
Suppose $\alpha\in B^{\ell-i}$, $\beta\in B^{m-j}$, $\gamma\in B^n$, $\delta\in B^j$, and $\varepsilon\in B^i$. Then 
{\fontsize{10}{10}{
\begin{align*}
\langle \mu_i(\kappa,\mu_j(\eta,\xi)), \alpha\otimes\beta\otimes\gamma\otimes \delta\otimes\varepsilon\rangle 
&=\left\langle \rho^{-i} \big(\rho^i(\kappa) \otimes \rho^{-j}(\rho^j(\eta)\otimes \xi) \big),\alpha\otimes\beta\otimes\gamma\otimes \delta\otimes\varepsilon\right\rangle \\
&=\left\langle  \rho^i(\kappa) \otimes \rho^{-j}\big(\rho^j(\eta)\otimes \xi\big), \varepsilon\otimes\alpha\otimes\beta\otimes\gamma\otimes \delta\right\rangle\\
&=\left\langle  \rho^{-j}\big(\rho^j(\eta)\otimes \xi\big), \langle \rho^i(\kappa) |\varepsilon\rangle_A\alpha\otimes\beta\otimes\gamma\otimes \delta\right\rangle\\
&=\left\langle  \rho^j(\eta)\otimes \xi, \delta\otimes\langle \rho^i(\kappa) |\varepsilon\rangle_A\alpha\otimes\beta\otimes\gamma\right\rangle\\
&=\left\langle  \rho^j(\eta), \delta\otimes\langle \rho^i(\kappa) |\varepsilon\rangle_A\alpha\otimes\beta{\sb{A}\langle}\gamma,\xi\rangle\right\rangle\\
&=\left\langle  \eta,\langle \rho^i(\kappa) |\varepsilon\rangle_A\alpha\otimes\beta{\sb{A}\langle}\gamma,\xi\rangle\otimes \delta\right\rangle\\
&=\left\langle  \rho^i(\kappa)\otimes \eta,\varepsilon\otimes\alpha\otimes\beta{\sb{A}\langle}\gamma,\xi\rangle\otimes \delta\right\rangle\\
&= \left\langle\rho^{j} \big( \rho^i(\kappa)\otimes\eta\big),\delta\otimes\varepsilon\otimes\alpha\otimes\beta{\sb{A}\langle}\gamma,\xi\rangle\right\rangle\\
&= \left\langle\rho^{j} \big( \rho^i(\kappa)\otimes\eta\big)\otimes \xi,\delta\otimes\varepsilon\otimes\alpha\otimes\beta\otimes\gamma\right\rangle\\
&=\left\langle \rho^{-i-j} \big(\rho^{i+j} ( \rho^{-i}(\rho^i(\kappa)\otimes\eta))\otimes \xi\big),\alpha\otimes\beta\otimes\gamma\otimes \delta\otimes\varepsilon\right\rangle\\
&=\langle \mu_{i+j}(\mu_i(\kappa,\eta),\xi), \alpha\otimes\beta\otimes\gamma\otimes \delta\otimes\varepsilon\rangle.
\end{align*}}}
\end{proof}

\begin{cor}\label{cor:POperad}
$P_\bullet$ naturally forms an algebra over the operad generated by the unshaded, oriented tangles 
$$
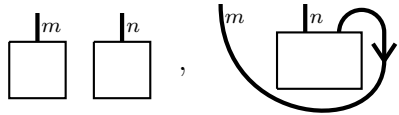
\TensorPn{m}{}{n}{},\Rotation{m}{n}{}
$$
for $m,n\geq 0$ up to planar isotopy.
\end{cor}

The Burns rotation is also compatible with the $\B\P$-algebra $\widehat{Q_\bullet^+}$.

\begin{thm}\label{thm:MoveAround}
\item[(1)]
For all $\zeta\in P_{m+n}$ and $x\in Q_m$, and $y\in Q_n$, $\rho^n((x\otimes_A y) \zeta) = (y\otimes_A x) \rho^n(\zeta)$:
$$
\begin{tikzpicture}[rectangular,baseline=-1cm]
	\clip (-1.65,.5) --(2.2,.5) -- (2.2,-3) -- (-1.65,-3);
	\draw[ultra thick] (-.6,-1.5) arc (0:180:.4cm) -- (-1.4,-2) .. controls ++(270:1.4cm) and ++(270:2.2cm) ..  (1.5,-1);
	\draw[ultra thick] (0,1)--(0,-2);
	\draw[ultra thick] (1.5,1)--(1.5,-1);
	\draw[ultra thick] (-1.2,-1.7)--(-1.4,-2)--(-1.6,-1.7);
	\filldraw[thick, unshaded] (-.5,0)--(-.5,-1)--(.5,-1)--(.5,0)--(-.5,0);
	\filldraw[thick, unshaded] (2,0)--(2,-1)--(1,-1)--(1,0)--(2,0);
	\filldraw[thick, unshaded] (.5,-2.5)--(.5,-1.5)--(-1,-1.5)--(-1,-2.5)--(.5,-2.5);
	\node at (-.25,-2) {$\zeta$};
%	\draw[ultra thick, unshaded] (0,.5) circle (.25cm);
	\node at (0,-.5) {$y$};
	\node at (0.2,.25) {{\scriptsize{$n$}}};
%	\draw[ultra thick, unshaded] (1.5,.5) circle (.25cm);
	\node at (1.5,-.5) {$x$};
	\node at (1.75,.25) {{\scriptsize{$m$}}};
\end{tikzpicture}
=
\begin{tikzpicture}[rectangular,baseline=-1cm]
	\clip (1.9,.5) --(-1.85,.5) -- (-1.85,-3.2) -- (1.9,-3.2);
	\draw[ultra thick] (-1.7,-1)--(-1.5,-1.3)--(-1.3,-1);
	\draw[ultra thick] (.7,1)--(.7,-2);
	\draw[ultra thick] (-.7,0)--(-.7,-2);
	\draw[ultra thick] (-.7,0) arc (0:180:.4cm) -- (-1.5,-2) .. controls ++(270:1.4cm) and ++(270:1.4cm) ..  (1.5,-2)--(1.5,1);
	\filldraw[thick, unshaded] (-.2,0)--(-.2,-1)--(-1.2,-1)--(-1.2,0)--(-.2,0);
	\filldraw[thick, unshaded] (.2,0)--(.2,-1)--(1.2,-1)--(1.2,0)--(.2,0);
	\filldraw[thick, unshaded] (1.2,-2.5)--(1.2,-1.5)--(-1.2,-1.5)--(-1.2,-2.5)--(1.2,-2.5);
	\node at (0,-2) {$\zeta$};
%	\draw[ultra thick, unshaded] (.7,.5) circle (.25cm);
	\node at (.7,-.5) {$y$};
	\node at (.9,.25) {{\scriptsize{$n$}}};
%	\draw[ultra thick, unshaded] (1.5,.5) circle (.25cm);
	\node at (-.7,-.5) {$x$};
	\node at (1.75,.25) {{\scriptsize{$m$}}};
\end{tikzpicture}
=
\begin{tikzpicture}[rectangular,baseline=-1cm]
	\clip (1.85,.5) --(-1.85,.5) -- (-1.85,-3.2) -- (1.85,-3.2);
	\draw[ultra thick] (1.7,-1)--(1.5,-1.3)--(1.3,-1);
	\draw[ultra thick] (-.7,1)--(-.7,-2);
	\draw[ultra thick] (.7,0)--(.7,-2);
	\draw[ultra thick] (.7,0) arc (180:0:.4cm) -- (1.5,-2) .. controls ++(270:1.4cm) and ++(270:1.4cm) ..  (-1.5,-2)--(-1.5,1);
	\filldraw[thick, unshaded] (-.2,0)--(-.2,-1)--(-1.2,-1)--(-1.2,0)--(-.2,0);
	\filldraw[thick, unshaded] (.2,0)--(.2,-1)--(1.2,-1)--(1.2,0)--(.2,0);
	\filldraw[thick, unshaded] (1.2,-2.5)--(1.2,-1.5)--(-1.2,-1.5)--(-1.2,-2.5)--(1.2,-2.5);
	\node at (0,-2) {$\zeta$};
%	\draw[ultra thick, unshaded] (-.7,.5) circle (.25cm);
	\node at (-.7,-.5) {$x$};
	\node at (-.45,.25) {{\scriptsize{$m$}}};
%	\draw[ultra thick, unshaded] (-1.5,.5) circle (.25cm);
	\node at (.7,-.5) {$y$};
	\node at (-1.3,.25) {{\scriptsize{$n$}}};
\end{tikzpicture}
=
\begin{tikzpicture}[rectangular,baseline=-1cm]
	\clip (1.65,.5) --(-2.2,.5) -- (-2.2,-3) -- (1.65,-3);
	\draw[ultra thick] (.6,-1.5) arc (180:0:.4cm) -- (1.4,-2) .. controls ++(270:1.4cm) and ++(270:2.2cm) ..  (-1.5,-1);
	\draw[ultra thick] (0,1)--(0,-2);
	\draw[ultra thick] (-1.5,1)--(-1.5,-1);
	\draw[ultra thick] (1.2,-1.7)--(1.4,-2)--(1.6,-1.7);
	\filldraw[thick, unshaded] (.5,0)--(.5,-1)--(-.5,-1)--(-.5,0)--(.5,0);
	\filldraw[thick, unshaded] (-2,0)--(-2,-1)--(-1,-1)--(-1,0)--(-2,0);
	\filldraw[thick, unshaded] (-.5,-2.5)--(-.5,-1.5)--(1,-1.5)--(1,-2.5)--(-.5,-2.5);
	\node at (.25,-2) {$\zeta$};
%	\draw[ultra thick, unshaded] (0,.5) circle (.25cm);
	\node at (0,-.5) {$x$};
	\node at (.25,.25) {{\scriptsize{$m$}}};
%	\draw[ultra thick, unshaded] (-1.5,.5) circle (.25cm);
	\node at (-1.5,-.5) {$y$};
	\node at (-1.3,.25) {{\scriptsize{$n$}}};
\end{tikzpicture}.
$$
\item[(2)] If $\rho$ is unitary, then for all $\zeta\in P_{m+n}$ and $x\in \widehat{Q_m^+}$, and $y\in \widehat{Q_n^+}$, $(y\otimes_A x)(\omega_{\rho^n \zeta}) = (x\otimes_A y)(\omega_\zeta)$:
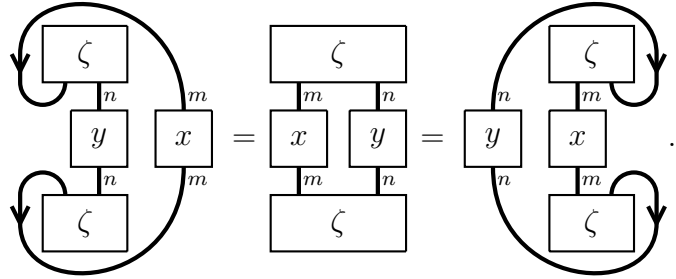
$$
\begin{tikzpicture}[rectangular]
	\clip (-1.6,2.6) --(2.2,2.6) -- (2.2,-2.6) -- (-1.6,-2.6);
	\draw[ultra thick] (-.6,1) arc (0:-180:.4cm) -- (-1.4,1.5) .. controls ++(90:1.4cm) and ++(90:2.2cm) ..  (1.5,.5)--(1.5,-.5);
	\draw[ultra thick] (-.6,-1) arc (0:180:.4cm) -- (-1.4,-1.5) .. controls ++(270:1.4cm) and ++(270:2.2cm) ..  (1.5,-.5);
	\draw[ultra thick] (0,2)--(0,-2);
	\draw[ultra thick] (-1.25,1.5)--(-1.4,1.2)--(-1.55,1.5);
	\draw[ultra thick] (-1.25,-1.2)--(-1.4,-1.5)--(-1.55,-1.2);
	\filldraw[thick, unshaded] (-.5,.5)--(-.5,-.5)--(.5,-.5)--(.5,.5)--(-.5,.5);
	\filldraw[thick, unshaded] (2,.5)--(2,-.5)--(1,-.5)--(1,.5)--(2,.5);
	\filldraw[thick, unshaded] (.5,2)--(.5,1)--(-1,1)--(-1,2)--(.5,2);
	\filldraw[thick, unshaded] (.5,-2)--(.5,-1)--(-1,-1)--(-1,-2)--(.5,-2);
	\node at (-.25,1.5) {$\zeta$};
	\node at (-.25,-1.5) {$\zeta$};
%	\draw[ultra thick, unshaded] (0,1) circle (.25cm);
%	\draw[ultra thick, unshaded] (0,-1) circle (.25cm);
	\node at (0,0) {$y$};
	\node at (0.2,.75) {{\scriptsize{$n$}}};
	\node at (0.2,-.75) {{\scriptsize{$n$}}};
%	\draw[ultra thick, unshaded] (1.5,1) circle (.25cm);
%	\draw[ultra thick, unshaded] (1.5,-1) circle (.25cm);
	\node at (1.5,0) {$x$};
	\node at (1.75,.75) {{\scriptsize{$m$}}};
	\node at (1.75,-.75) {{\scriptsize{$m$}}};
\end{tikzpicture}
=
\begin{tikzpicture}[rectangular]
	\clip (1.3,2.6) --(-1.3,2.6) -- (-1.3,-2.6) -- (1.3,-2.6);
	\draw[ultra thick] (.7,1.5)--(.7,-1.5);
	\draw[ultra thick] (-.7,1.5)--(-.7,-1.5);
	\filldraw[thick, unshaded] (-.2,.5)--(-.2,-.5)--(-1.2,-.5)--(-1.2,.5)--(-.2,.5);
	\filldraw[thick, unshaded] (.2,.5)--(.2,-.5)--(1.2,-.5)--(1.2,.5)--(.2,.5);
	\filldraw[thick, unshaded] (1.2,2)--(1.2,1)--(-1.2,1)--(-1.2,2)--(1.2,2);
	\filldraw[thick, unshaded] (1.2,-2)--(1.2,-1)--(-1.2,-1)--(-1.2,-2)--(1.2,-2);
	\node at (0,1.5) {$\zeta$};
	\node at (0,-1.5) {$\zeta$};
%	\draw[ultra thick, unshaded] (.7,1) circle (.25cm);
%	\draw[ultra thick, unshaded] (.7,-1) circle (.25cm);
	\node at (.7,0) {$y$};
	\node at (.9,.75) {{\scriptsize{$n$}}};
	\node at (.9,-.75) {{\scriptsize{$n$}}};
%	\draw[ultra thick, unshaded] (-.7,1) circle (.25cm);
%	\draw[ultra thick, unshaded] (-.7,-1) circle (.25cm);
	\node at (-.7,0) {$x$};
	\node at (-.45,.75) {{\scriptsize{$m$}}};
	\node at (-.45,-.75) {{\scriptsize{$m$}}};
\end{tikzpicture}
=
\begin{tikzpicture}[rectangular]
	\clip (1.6,2.6) --(-2.2,2.6) -- (-2.2,-2.6) -- (1.6,-2.6);
	\draw[ultra thick] (.6,1) arc (180:360:.4cm) -- (1.4,1.5) .. controls ++(90:1.4cm) and ++(90:2.2cm) ..  (-1.5,.5)--(-1.5,-.5);
	\draw[ultra thick] (.6,-1) arc (180:0:.4cm) -- (1.4,-1.5) .. controls ++(270:1.4cm) and ++(270:2.2cm) ..  (-1.5,-.5);
	\draw[ultra thick] (0,2)--(0,-2);
	\draw[ultra thick] (1.25,1.5)--(1.4,1.2)--(1.55,1.5);
	\draw[ultra thick] (1.25,-1.2)--(1.4,-1.5)--(1.55,-1.2);
	\filldraw[thick, unshaded] (-.5,.5)--(-.5,-.5)--(.5,-.5)--(.5,.5)--(-.5,.5);
	\filldraw[thick, unshaded] (-2,.5)--(-2,-.5)--(-1,-.5)--(-1,.5)--(-2,.5);
	\filldraw[thick, unshaded] (-.5,2)--(-.5,1)--(1,1)--(1,2)--(-.5,2);
	\filldraw[thick, unshaded] (-.5,-2)--(-.5,-1)--(1,-1)--(1,-2)--(-.5,-2);
	\node at (.25,1.5) {$\zeta$};
	\node at (.25,-1.5) {$\zeta$};
%	\draw[ultra thick, unshaded] (0,1) circle (.25cm);
%	\draw[ultra thick, unshaded] (0,-1) circle (.25cm);
	\node at (0,0) {$x$};
	\node at (0.25,.75) {{\scriptsize{$m$}}};
	\node at (0.25,-.75) {{\scriptsize{$m$}}};
%	\draw[ultra thick, unshaded] (-1.5,1) circle (.25cm);
%	\draw[ultra thick, unshaded] (-1.5,-1) circle (.25cm);
	\node at (-1.5,0) {$y$};
	\node at (-1.3,.75) {{\scriptsize{$n$}}};
	\node at (-1.3,-.75) {{\scriptsize{$n$}}};
\end{tikzpicture}.
$$
\end{thm}
\begin{proof}
\item[(1)]
For $\eta\in B^n$ and $\xi\in B^m$,
\begin{align*}
\langle \rho^n((x\otimes_A y) \zeta) , \eta\otimes\xi \rangle 
& = \langle (x\otimes_A y) \zeta , \xi\otimes \eta\rangle
= \langle  \zeta, (x^*\otimes_A y^*)( \xi\otimes \eta) \rangle \\
& = \langle  \zeta,(x^*\xi)\otimes (y^*\eta)\rangle
= \langle  \rho^n (\zeta), (y^*\eta)\otimes (x^*\xi) \rangle \\
& = \langle (y\otimes_A x) \rho^n(\zeta), \eta\otimes\xi\rangle.
\end{align*}
\item[(2)] Pick $(x_i)\subset Q_m^+$ and $(y_j)\subset Q_n^+$ with $x_i\nearrow x$ and $y_j\nearrow y$. Then by (1), for all $i$,
\begin{align*}
(y_j\otimes_A x_i)(\omega_{\rho^n \zeta})
& = \| (y_j^{1/2}\otimes_A x_i^{1/2}) \rho^n\zeta\|_2^2
= \|\rho^n((x_i^{1/2}\otimes_A y_j^{1/2})\zeta)\|_2^2 \\
& = \|(x_i^{1/2}\otimes_A y_j^{1/2})\zeta\|_2^2
= (x_i\otimes_A y_j)(\omega_\zeta).
\end{align*}
We are finished by Theorem \ref{thm:increase}, since $x_i\otimes_A y_j \nearrow x\otimes_A y$ and $y_j\otimes_A x_i\nearrow y\otimes_A x$.
\end{proof}

\begin{rem}
When the operads for $P_\bullet$ and $\widehat{Q_\bullet^+}$ interact as in Theorem \ref{thm:action}, we may remove closed subdiagrams and multiply by the appropriate scalar in $[0,\I]$ by Corollary \ref{cor:POperad} and Theorem \ref{thm:MoveAround}.
\end{rem}

%%%%%%%%%%%%%%%%%%%%%%%%%%%%%%%%%%%%%%%%%%%%%%%%%%
\subsection{Extremality implies the existence of the Burns rotation}\label{sec:ExtremalityImpliesRotations}

We will show in the next lemma and theorem that (approximate) extremality implies the existence of a Burns rotation. The intuition comes from the bimodule planar calculus. In diagrams, for the extremal case, we have:
$$
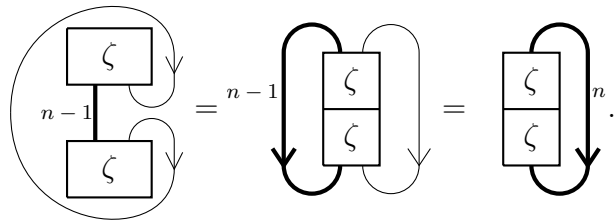
\begin{tikzpicture}[rectangular,baseline=-.8cm]
	\clip (1.6,1) --(-2.2,1) -- (-2.2,-3) -- (1.6,-3);
	\draw[] (.6,-.5) arc (180:360:.4cm) -- (1.4,0) .. controls ++(90:1.4cm) and ++(90:2.2cm) ..  (-1.5,-1);
	\draw[] (.6,-1.5) arc (180:0:.4cm) -- (1.4,-2) .. controls ++(270:1.4cm) and ++(270:2.2cm) ..  (-1.5,-1);
	\draw[ultra thick] (0,-2)--(0,0);
	\draw[] (1.25,0)--(1.4,-.3)--(1.55,0);
	\draw[] (1.25,-1.7)--(1.4,-2)--(1.55,-1.7);

	\filldraw[thick, unshaded] (-.5,.5)--(-.5,-.5)--(1,-.5)--(1,.5)--(-.5,.5);
	\filldraw[thick, unshaded] (-.5,-2.5)--(-.5,-1.5)--(1,-1.5)--(1,-2.5)--(-.5,-2.5);
	\node at (.25,0) {$\zeta$};
	\node at (.25,-2) {$\zeta$};
%	\draw[ultra thick, unshaded] (0,-1) ellipse (.65cm and .25cm);
	\node at (-.5,-1) {{\scriptsize{$n-1$}}};
\end{tikzpicture}
=
\begin{tikzpicture}[rectangular,baseline=1.4cm]
	\clip (-2.2,3.7) --(1.45,3.7) -- (1.45,-.2) -- (-2.2,-.2);
	\draw[ultra thick] (-1,1.3)--(-1.2,1)--(-1.4,1.3);
	\draw[] (1,1.3)--(1.2,1)--(1.4,1.3);
	\draw[ultra thick] (-.2,3) arc (0:180:.5cm) --(-1.2,1) arc (-180:0:.5cm);
	\draw[] (.2,3) arc (180:0:.5cm) --(1.2,1) arc (0:-180:.5cm);
	\filldraw[thick, unshaded] (-.5,1) --(-.5,3) -- (.5,3) -- (.5,1)--(-.5,1);
	\draw[thick] (-.5,2)--(.5,2);
	\node at (0,2.5) {$\zeta$};
	\node at (0,1.5) {$\zeta$};
%	\draw[ultra thick, unshaded] (-1.4,2.3) ellipse (.65cm and .25cm);
	\node at (-1.75,2.3) {{\scriptsize{$n-1$}}};
\end{tikzpicture}
=
\begin{tikzpicture}[rectangular,baseline=1.4cm]
	\clip (-1,3.7) --(1.35,3.7) -- (1.35,-.2) -- (-1,-.2);
	\draw[ultra thick] (.8,1.3)--(1,1)--(1.2,1.3);
	\draw[ultra thick] (0,3) arc (180:0:.5cm) --(1,1) arc (0:-180:.5cm);
	\filldraw[thick, unshaded] (-.5,1) --(-.5,3) -- (.5,3) -- (.5,1)--(-.5,1);
	\draw[thick] (-.5,2)--(.5,2);
	\node at (0,2.5) {$\zeta$};
	\node at (0,1.5) {$\zeta$};
%	\draw[ultra thick, unshaded] (1,2.3) circle (.25cm);
	\node at (1.2,2.3) {{\scriptsize{$n$}}};
\end{tikzpicture}.
$$
Although these diagrams are not yet well-defined, they tell us how to proceed. They become well-defined after the Burns rotation exists by Theorems \ref{thm:CloseOffZetas} and \ref{thm:MoveAround}.

\begin{lem}\label{lem:RotationLemma} Let $p_n$ be the projection in $B(H^n)$ with range $P_n$.
\item[(1)]
If $H$ is approximately extremal with constant $\lambda>0$, then 
$$\left(\sum_\beta  p_nR_\beta R_\beta^*p_n\right)\leq \lambda^{n-1} p_n \text{ and }\left(\sum_\alpha p_nL_\alpha L_\alpha^* p_n\right) \leq \lambda^{n-1} p_n.$$
\item[(2)]
If $H$ is extremal, then on $P_n$, $\sum_\beta p_nR_\beta R_\beta^*p_n = p_n=\sum_\alpha p_nL_\alpha L_\alpha^*p_n$.
\end{lem}
\begin{proof}
\item[(1)]
We prove the first inequality. Note that $R_{\beta}^*\zeta\in D(\sb{A}H^{n-1})$, and $R(R_\beta^*\zeta)=R_\beta^* R(\zeta)\colon L^2(A)\to H^{n-1}$. Since $H$ is (approximately) extremal, so is $H^{n-1}$ with constant $\lambda^{n-1}$, and
\begin{align*}
\left\langle\left( \sum_\beta  p_nR_\beta R_\beta^*p_n\right) \zeta,\zeta\right\rangle_{P_n}
& = \sum_{\beta} \|R_{\beta}^* \zeta\|_2^2 
= \sum_{\beta} \tr_{A}\left({\sb{A}\langle}R_{\beta}^* \zeta,R_{\beta}^* \zeta\rangle \right) \\
& = \sum_{\beta} \Tr_{n-1}\op\left(R_{\beta}^* R(\zeta)R(\zeta)^*R_{\beta} \right)  
= \Tr_{n-1}\op T_{n-1} (R(\zeta)R(\zeta)^*)\\
& \leq \lambda^{n-1} \Tr_{n-1} T_{n-1} (L(\zeta)L(\zeta)^*)
= \lambda^{n-1}\Tr_n(L(\zeta)L(\zeta)^*)\\
& = \lambda^{n-1}\|\zeta\|_2^2
= \langle (\lambda^{n-1}p_n) \zeta,\zeta\rangle_{P_n}.
\end{align*}
\item[(2)] As $\lambda=1$, by (1),
$$
\left\langle \left(\sum_\beta p_n R_\beta R_\beta^* p_n\right) \zeta, \zeta\right\rangle = \langle \zeta, \zeta\rangle
$$
for all $\zeta\in P_n$, and the result follows from polarization.
\end{proof}

\begin{thm}\label{thm:rotation}
Suppose $H$ is approximately extremal. Then $\rho=\sum_{\beta} L_\beta R_\beta^* $ converges strongly on $P_n$. Moreover if $H$ is extremal, $\rho$ is unitary. A similar result holds for $\rho\op=\sum_\alpha R_\alpha L_\alpha^*$.
\end{thm}
\begin{proof}
We begin as in the proof of Proposition 3.3.19 of \cite{burns}, but as we do not have Jones projections, we use Lemma \ref{lem:RotationLemma}.

Suppose $\zeta\in P_n$, and enumerate $\{\beta\}=\{\beta_i\}_{i\in\N}$. We will show 
$$
\left\|\sum_{i=r}^s L_{\beta_i}R_{\beta_i}^* \zeta\right\|_2^2\to 0 \text{ as } r,s\to \I.
$$
First note that the infinite matrix $(L_{\beta_j}^* L_{\beta_i})$ is a projection, so it is dominated by $1=\delta_{i,j}$. Hence each corner $(L_{\beta_j}^* L_{\beta_i})_{i,j=r}^s$ is dominated by $1=\delta_{i,j}$, and 
$$
 \left\| \sum_{i=r}^s L_{\beta_i} R_{\beta_i}^* \zeta\right\|_2^2 
 = \sum_{i,j=r}^s\left\langle (L_{\beta_j}^* L_{\beta_i}) R_{\beta_i}^*\zeta, R_{\beta_j}^*\zeta\right\rangle
 \leq \sum_{i=r}^s \langle R_{\beta_i}^* \zeta, R_{\beta_i}^* \zeta\rangle.
$$
We need to show that the right hand side tends to zero, which is certainly true if the infinite sum $\sum_{\beta} \|R_{\beta}^* \zeta\|_2^2$ converges. But this follows immediately from Lemma \ref{lem:RotationLemma}. Hence $\rho$ converges and $\|\rho\|\leq \sqrt{\lambda^{n-1}}$ (where $\lambda$ is the approximate extremality constant). If $\lambda = 1$, then $\|\rho\|\leq 1$ and $\rho^n=\id_{P_n}$, so $\rho$ is necessarily isometric and thus unitary.
\end{proof}

%%%%%%%%%%%%%%%%%%%%%%%%%%%%%%%%%%%%%%%%%%%%%%%%%%
\subsection{Symmetric bimodules and a converse of Theorem \ref{thm:rotation}}\label{sec:symmetric}
To prove a converse of Theorem \ref{thm:rotation}, we need additional structure on $H$ due to Example \ref{ex:NoCentralVectors}.

\begin{rem}
For the rest of this section, we assume $H$ is symmetric (see Remark \ref{rem:symmetric}).
\end{rem}

\begin{lem}\label{lem:switch}
For all $\eta,\xi\in B^n$, $\langle \eta |\xi\rangle_A = {\sb{A}\langle} J\eta, J\xi\rangle$.
\end{lem}
\begin{proof}
Suppose $a_1,a_2\in A$. Then 
\begin{align*}
\big\langle {\sb{A}\langle} J\eta,J\xi\rangle \widehat{a_1},\widehat{a_2}\big\rangle 
& = \langle J R(J\eta)^* R(J\xi) J \widehat{a_1}, \widehat{a_2}\rangle
= \langle \widehat{a_2^*} , R(J\eta)^* R(J\xi) \widehat{a_1^*}\rangle = \langle a_2^* J\eta , a_1^* J\xi\rangle\\
& = \langle J(\eta a_2) , J (\xi a_1)\rangle 
= \langle \xi a_1, \eta a_2\rangle 
= \big\langle \langle \eta| \xi\rangle_A \widehat{a_1} , \widehat{a_2}\big\rangle.
\end{align*}
\end{proof}

\begin{defn}
Using Lemma \ref{lem:switch}, we define an algebra structure on $B^n\otimes_A B^n$ as follows: if $\eta_1,\eta_2,\xi_1,\xi_2\in B^n$, then
$$
(\eta_1\otimes \xi_1)( \eta_2\otimes \xi_2) = \eta_1 \langle J\xi_1 | \eta_2\rangle_A \otimes \xi_2 = \eta_1 {\sb{A}\langle} \xi_1,J\eta_2\rangle \otimes \xi_2.
$$
\end{defn}

\begin{prop}[\cite{MR703809,MR1055223}]\label{prop:ModuleIso}
The map $B^n\otimes_A B^n\to C_n$ by $\eta\otimes J_n\xi \mapsto L(\eta)L(\xi)^*$ gives a $*$-algebra isomorphism onto its image, and it extends to a $C_n-C_n$ bimodule isomorphism $\theta_n\colon H^{2n}\to L^2(C_n,\Tr_n)$. The same result holds swapping $\OP$.
\end{prop}
\begin{proof}
The map is well defined as it is $A$-middle linear:
\begin{align*}
\eta a \otimes J_n\xi 
& \mapsto L(\eta a)L(\xi)^* 
= L(\eta) a L(\xi)^* 
= L(\eta) L(\xi a^*)^*
\text{ and }\\
\eta\otimes a J_n\xi 
& \mapsto L(\eta) L(J_n(aJ_n\xi))^* = L(\eta)L(\xi a^*)^*.
\end{align*}
The map clearly preserves the multiplicative structure and is isometric by construction. If $\eta_1,\eta_2,\xi_1,\xi_2\in B^n$, then
\begin{align*}
\langle L(\eta_1) L(\xi_1)^* , L(\eta_2)L(\xi_2)^* \rangle_{L^2(C_n,\Tr_n)}
& = \Tr_n \left(L(\xi_2)L(\eta_2)^*L(\eta_1) L(\xi_1)^*\right)\\
& = \Tr_n \left(L(\xi_2)\langle \eta_2| \eta_1\rangle_A L(\xi_1)^*\right)\\
& = \Tr_n \left(L(\xi_2\langle \eta_2| \eta_1\rangle_A) L(\xi_1)^*\right)\\
& = \langle \xi_2\langle \eta_2| \eta_1\rangle_A,\xi_1\rangle_{H^n} \\
&= \langle J_n\xi_1 , J_n(\xi_2\langle \eta_2| \eta_1\rangle_A)\rangle_{H^n}\\
& = \langle J_n\xi_1 ,  \langle \eta_1| \eta_2\rangle_AJ_n\xi_2\rangle_{H^n}\\
&= \langle \eta_1\otimes J_n\xi_1 , \eta_2\otimes J_n \xi_2 \rangle_{H^{2n}}.
\end{align*}
Hence it clearly extends to a $C_n-C_n$ bilinear bimodule isomorphism.
\end{proof}

\begin{cor}
$C_{n-k}\subseteq C_n\subseteq C_{n+k}$ is standard (isomorphic to the basic construction) for all $n,k\geq 0$.
\end{cor}
\begin{proof}
By Remark \ref{rem:FiberProduct} and Proposition \ref{prop:ModuleIso},
$$
J_{2n}(C_{n-k}\otimes_A \id_{n+k})'J_{2n}=J_{2n}(\id_{n-k}\otimes_A C_{n+k}\op)J_{2n}=C_{n+k}\otimes_A \id_{n-k}.
$$
\end{proof}

\begin{lem}[\cite{burns}, Theorem 3.3.13]\label{lem:BurnsL2Lemma}
Let $N$ be a von Neumann subalgebra of a semifinite von Neumann algebra $M$ with n.f.s. trace $\Tr_M$. Then
\item[(1)] $N'\cap L^2(M)=\overline{N'\cap \n_{\Tr_M}}^{\|\cdot\|_2}$
\item[(2)] $(N'\cap L^2(M))^\perp = \overline{[N,\n_{\Tr_M}]}^{\|\cdot\|_2}$, the closure of the span of the commutators in $L^2(M)$.
\end{lem}

\begin{rem}
By Proposition \ref{prop:ModuleIso} and Lemma \ref{lem:BurnsL2Lemma}, $\theta_n$ yields an isomorphsim
$$P_{2n} = A'\cap H^{2n}\cong A'\cap L^2(C_n,\Tr_n) = \overline{ A'\cap \n_{\Tr_n}}^{\|\cdot\|_2}=\overline{C_n\op \cap \n_{\Tr_n}}^{\|\cdot\|_2}=L^2(Q_n,\Tr_n)$$ 
of $Q_n- Q_n$ bimodules. A similar result holds swapping $\OP$.
\end{rem}

\begin{thm}\label{thm:RotationConverse}
If $\rho$ exists on $P_{2n}$, then $H^n$ is approximately extremal. If $\rho$ is unitary, then $H^n$ is extremal.
\end{thm}
\begin{proof} 
The main step is to show the following lemma, whose proof is essentially the same as in \cite{burns}.

\begin{lem}[3.3.21.(ii) of \cite{burns}]\label{lem:BurnsLemma}
If $\rho$ exists on $P_{2n}$, then for all $x\in C_n\op\cap \n_{\Tr_n}$, $\rho^n(\theta_n^{-1}(\widehat{x}))=\theta_n^{-1}(\widehat{j_n(x)})\in C_n\op\cap \n_{\Tr_n}$. In particular, $C_n\op\cap \n_{\Tr_n}=\n_{\Tr_n\op}\cap \n_{\Tr_n}$. A similar result holds swapping $\OP$.
\end{lem}

Using this lemma, Burns' proof shows $\Tr_n\op\leq \|\rho^n\| \Tr_n$ on $Q_n^+$. Suppose $z\in Q_n$. If $\Tr_n(z^*z)=\I$, we are finished. Otherwise, $z\in C_n\op\cap \n_{\Tr_n}=\n_{\Tr_n\op}\cap \n_{\Tr_n}$, and
\begin{align*}
\Tr_n\op (z^*z) 
& = \Tr_n\circ j_n (z^*z) 
= \Tr_n (j_n(z)^* j_n(z)) 
= \left\langle \widehat{j_n(z)},\widehat{j_n(z)}\right\rangle_{L^2(Q_n,\Tr_n)} \\
& = \left\langle \theta_n^{-1}(\widehat{j_n(z)}),\theta_n^{-1}(\widehat{j_n(z)})\right\rangle_{P_n}
= \left\langle \rho^n (\theta_n^{-1}(\widehat{z})),  \rho^n (\theta_n^{-1}(\widehat{z}))\right\rangle_{P_n}\\
& = \|\rho^n(\theta_n^{-1}(\widehat{z}))\|^2_{P_n}
\leq \|\rho^n\|^2 \|\theta_n^{-1}(\widehat{z})\|^2_{P_n} 
= \|\rho^n\|^2 \|\widehat{z}\|^2_{L^2(Q_n,\Tr_n)} \\
& = \|\rho^n\|^2 \Tr_n(z^*z).
\end{align*}
Similarly $\Tr_n \leq \|\rho^n\|^2 \Tr_n\op$ on $Q_n^+$, and $H^n$ is approximately extremal. In particular, if $\|\rho\|=1$, $H^n$ is extremal.
\end{proof}

\begin{rem}
Theorem \ref{thm:MainTheorem} now follows immediately from Theorems \ref{thm:extremal}, \ref{thm:rotation}, and \ref{thm:RotationConverse}.
\end{rem}

%%%%%%%%%%%%%%%%%%%%%%%%%%%%%%%%%%%%%%%%%%%%%%%%%%
\section{Examples}\label{sec:Examples}

\begin{ex}[Bifinite bimodules]
In the case that $H$ is a symmetric, bifinite $A-A$ bimodule, then the $\B\P$-algebra structure encodes the $C^*$-tensor category whose objects are the sub-bimodules of $H^n$ for some $n$ and whose morphisms are intertwiners.
\end{ex}

\begin{ex}\label{ex:InfiniteIndex}
Suppose $A_0=A\subset B=A_1$ is an infinite index inclusion of $II_1$-factors. Then $H=L^2(B)$ gives an $A-A$ bimodule. In this case, letting $A_{n+1}$ be the $n\thh$ iterated basic construction of $A_{n-1}\subset A_n$, we have
\item[$\bullet$] $H^n\cong L^2(A_n,\Tr_n)$,
\item[$\bullet$] $C_n,C_n\op$ is the left,right action respectively of $A_{2n}$, and
\item[$\bullet$] $Q_n = A_0'\cap A_{2n}$.
\item
Theorem \ref{thm:MainTheorem} was proven for this case by \cite{burns}.
\end{ex}

\begin{ex}\label{ex:NoCentralVectors}
Suppose $A$ is a $II_1$-factor, and $\sigma\in \Aut(A)$. Define $H_\sigma = {\sb{A}L^2}(A)_{\sigma(A)}$ by $a\widehat{b}c = \widehat{ab\sigma(c)}$ for all $a,b,c,\in A$. Suppose that $\sigma$ is outer and not periodic, and $\sigma^n$ is outer for all $n\in\N$. Then $H_\sigma^n \cong H_{\sigma^n}$ is extremal and $P_n=(0)$ for all $n\geq 1$.
\end{ex}

\begin{ex}[Group actions]\label{ex:GroupActions}
Suppose $G$ is a countable i.c.c. group, and $\pi\colon G\to U(K)$ is a unitary representation. We can define two bimodules:
\item[(1)] $H=K\otimes_\C \ell^2(G)$ where the left action is given by the diagonal action $\pi\otimes \lambda$ and the right action is given by $1\otimes \rho$ where $\lambda,\rho$ are the left,right regular representation of $G$ on $\ell^2(G)$. Hence $K\otimes_\C \ell^2(G)$ gives an $A-A$ bimodule where $A=LG$. Then we may identify
$$
H^n = K^n\otimes_\C \ell^2(G)
$$
where we write $K^n=K^{\otimes_\C n}$, and the left action is the diagonal action $\pi^n\otimes \lambda$ and the right action is $1_n\otimes \rho$. It is clear that projections in $Q_n$ correspond to $LG-LG$ invariant subspaces of $H^n$. Every $G$-invariant subspace of $K^n $ yields such a subspace, but in general, they do not exhaust all possible subspaces.
\item[(2)] To fix this problem, we use an idea of Richard Burstein and add a copy of the hyperfinite $II_1$-factor $R$. Suppose $\alpha\colon G\to \Aut(R)$ is an outer action, so $A=R\rtimes_\alpha G$ is a $II_1$-factor. Set $H=K\otimes_\C L^2(R) \otimes_\C \ell^2(G)$, and consider the left and right actions where 
\begin{align*}
r_1 (k\otimes \widehat{r_2} \otimes \delta_g) r_3 &= k \otimes \widehat{r_1 r_2 \alpha_{g}^{-1}(r_3)}\otimes \delta_{g} \\
g_1(k\otimes \widehat{r} \otimes \delta_{g_2}) g_3 &= (\pi_{g_1}k)\otimes \widehat{\alpha_{g_1}(r)} \otimes \delta_{g_1g_2g_3}
\end{align*}
for $r,r_i\in R$ and $g,g_i\in G$ for $i=1,2,3$. Hence $g\in G$ acts on the left by $\pi_g\otimes \alpha_g \otimes \lambda_g$ and on the right by $1\otimes 1\otimes \rho_g$.  Then similarly we may identify
$$H^n=K^n  \otimes_\C L^2(R)\otimes_\C \ell^2(G).$$
\end{ex}

\begin{thm}\label{thm:invariants}
For $A=R\rtimes_\alpha G$ and $H^n$ as above, $A-A$ invariant subspaces of $H^n$ correspond to $G$-invariant subspaces of $K^n$.
\end{thm}
\begin{proof}
First, if $L_0\subset K^n$ is a $G$-invariant subspace, then $L_0\otimes L^2(A)$ is an $A-A$ invariant subspace of $H^n$.

Now suppose $L\subset H^n$ is an $A-A$ invariant subspace, and let $p\in Q_n$ be the projection onto $L$. Note that 
\begin{align*}
p&\in \bigg(1_{K^n}\otimes R\bigg)' \cap \bigg(1_{K^n} \otimes A\op\bigg)' \\
&= \bigg(B(K^n)\otimes (R'\cap B(L^2(A)))\bigg)\cap \bigg(B(K^n)\otimes A\bigg)\\
&=B(K^n)\otimes ( R'\cap A)=B(K^n)\otimes 1_{L^2(A)}.
\end{align*}
Hence there is a $q\in B(K^n)$ such that $p=q\otimes 1_{L^2(A)}$. But since $q$ commutes with the left $G$-action on $H^n$, we have $q\in \pi(G)'\cap B(K^n)$.
\end{proof}

\begin{cor}
$A-A$ invariant vectors of $H^n$ correspond to $G$-invariant vectors of $K^n$.
\end{cor}

\begin{ex}[Group-subgroup]\label{ex:GroupSubgroup}
Suppose $G_0\subseteq G_1$ is an inclusion of countable i.c.c. groups, and let $K=\ell^2(G_1/G_0)$. As in Example \ref{ex:GroupActions}, we consider two cases:
\item[(1)] $A_0=LG_0$, $A_1= LG_1$, and $H=K\otimes_\C \ell^2(G_1)$.
\item[(2)] $A_0= R\rtimes G_0$, $A_1= R\rtimes G_1$, and $H=K\otimes_\C L^2(R) \otimes \ell^2(G_1)$.
\item Note that in either case, $H^n \cong L^2(A_{n+1})$, where $A_{n+1}=J_n A_{n-1}' J_n$ is the basic construction of $A_{n-1}\subset A_{n}$. As in the usual subfactor treatment, we can consider $H^n$ as an $A_i-A_j$ bimodule for $i,j\in \{0,1\}$.
\end{ex}

\begin{thm}
Let $G_1=S_\I$, the group of finite permutations of $\N$, and let $G_0=\Stab(1)$ be the permutations which fix $1$. Let $A_0=R\rtimes G_0$ and $A_1=R\rtimes G_1$, and let $H=K\otimes_\C L^2(R)\otimes \ell^2(G_1)$ as in (2) of Example \ref{ex:GroupSubgroup}. Then considering $H^n$ as an $A_0-A_0$ or as an $A_1-A_1$ bimodule, we have that $\dim(Q_n)<\I$ for all $n\in\N$.
\end{thm}
\begin{proof}
Since $A_i'\cap A_j \cong A_{i+2}'\cap A_{j+2}$ for all $i,j\geq 0$ by \cite{MR1387518}, it suffices to show that $\dim(A_1'\cap A_{2n+1})<\I$ for all $n\geq 0$. Also by \cite{MR1387518}, 
$$A_1'\cap A_{2n+1}\cong \End_{A_1-A_1}(L^2(A_{n+1}))\cong \End_{A_1-A_1}(H^n).$$
By Theorem \ref{thm:invariants}, $A_1-A_1$ invariant subspaces of $H^n$ correspond to $G_1$-invariant subspaces of $K^{n}$. The result now follows by \cite{MR0286940}.
\end{proof}

\begin{cor}\label{cor:FiniteDimensional}
The infinite index $II_1$-subfactor $R\rtimes G_0 \subset R\rtimes G_1$ for $G_0=\Stab(1)\subset S_\I=G_1$ has finite dimensional higher relative commutants.
\end{cor}

\begin{thm}
Suppose $G_0\subset G_1$ and $K$ are as in Example \ref{ex:GroupSubgroup} such that 
$[G_1\colon G_0]=\I$ and
$\#G_0\backslash G_1/G_0=2$.
Then 
\item[(1)] the space of $G_0$-invariant vectors in $K^n$ is one dimensional, and
\item[(2)] zero is the only $G_1$-invariant vector in $K^n$.
\end{thm}
\begin{proof}
Let $\{g_i\}_{i\geq 0}$ be a set of coset representatives for $G_1/G_0$ with $g_0=e$. Since $\# G_0\backslash G_1/ G_0 = 2$, for $i,j\geq 1$, there are $h_{i,j}\in G_0$ such that $h_{i,j} g_i G_0 = g_j G_0$.
\item[(1)]
Suppose 
$$
\xi = \sum_{i_1,\dots, i_n} \lambda_{i_1,\dots, i_n} \delta_{g_{i_1}G_0}\otimes\cdots \otimes \delta_{g_{i_n}G_0} \in K^n
$$ 
is $G_0$-invariant. Then since $\pi_{h_{i,j}}\xi = \xi$ for all $i,j\geq 1$, we must have $\lambda_{i_1,\dots, i_n}=0$ unless $i_j=0$ for all $j=1,\dots, n$. (Otherwise, there would be infinitely many coefficients which would be nonzero and equal, a contradiction to $\xi\in K^n \cong \ell^2((G_1/G_0)^n)$.) Hence $\xi\in \spann\{ \delta_{G_0}\otimes\cdots \otimes\delta_{G_0}\}$.
\item[(2)]
Since $\delta_{G_0}\otimes\cdots \otimes\delta_{G_0}$ is not $G_1$-invariant, the result follows from (1).
\end{proof}

\begin{cor}\label{cor:InfiniteSymmetric}
Let $G_0=\Stab(1)\subset S_\I= G_1$. Let $A_i= R\rtimes G_i$ for $i=0,1$, and let $K=\ell^2(G_1/G_0)$.
\item[(1)] When we consider $H=K\otimes_C L^2(R)\otimes_\C \ell^2(G_1)$ as an $A_1-A_1$ bimodule, $P_n=(0)$. 
\item[(2)] When we consider $H=L^2(A_1)=L^2(R)\otimes_\C \ell^2(G_1)$ as an $A_0-A_0$ bimodule, 
$$
H^n\cong L^2(A_n)\cong K^{n-1}\otimes_\C L^2(R)\otimes_\C\ell^2(G_1),
$$
and for all $n\geq 0$, $P_n$ is one-dimensional and spanned by 
$$
\widehat{1}\otimes \cdots \otimes \widehat{1}\in \bigotimes_{A_0}^n L^2(A_1) \cong L^2(A_{n}).
$$ 
\end{cor}

In joint work with Steven Deprez, we have shown an even stronger result:
\begin{thm}\label{thm:InfiniteSymmetricGroup}
The algebras $Q_n$ for the bimodules in (1) and (2) in Example \ref{ex:GroupSubgroup} are finite dimensional, and the dimensions grow super-factorially. 
\end{thm}

\begin{cor}
The infinite index $II_1$-subfactor $LG_0\subset LG_1$ where $G_0=\Stab(1)\subset S_\I = G_1$ has finite dimensional higher relative commutants.
\end{cor}

%%%%%%%%%%%%%%%%%%%%%%%%%%%%%%%%%%%%%%%%%%%%%%%%%
\appendix
%%%%%%%%%%%%%%%%%%%%%%%%%%%%%%%%%%%%%%%%%%%%%%%%%%
\section{Relative tensor products of extended positive cones}\label{sec:TensorUnbounded}

\begin{nota}
For this section, let $H_A$ be a right Hilbert $A$-module, $\sb{A}K_B$ be a Hilbert $A-B$ bimodule, and $\sb{B}L$ be a left Hilbert $B$-module where $A,B$ are finite von Neumann algebras. We write:
\item[$\bullet$] $X=(A\op)'\cap B(H)$,
\item[$\bullet$] $\sb{A}K$ when we ignore the right $B$-action,
\item[$\bullet$] $Y_0 = A'\cap B(K)$,
\item[$\bullet$] $Y=A'\cap (B\op)'\cap B(K)$, 
\item[$\bullet$] $Z=B'\cap B(L)$,
\item[$\bullet$] $X\otimes_A Y_0 = \set{x\otimes_A y}{x\in X\text{ and } y\in Y_0}''$, and
\item[$\bullet$] $X\otimes_A Y \otimes_B Z = \set{x\otimes_A y\otimes_B z}{x\in X,\, y\in Y, \text{ and } z\in Z}''$.
\end{nota}

The goal of this section is to define the operator $x\otimes_A y \in \widehat{(X\otimes_A Y_0)^+}$ for $x\in \widehat{X^+}$ and $y\in \widehat{Y_0^+}$ such that certain properties, e.g., associativity, are satisfied.

The next three lemmata are straightforward, but we include some proofs for completeness and for the convenience of the reader.

\begin{lem}\label{lem:WeakStrong}
Suppose $x\in M^+$ and $(x_i)_{i\in I}\subset M^+$ is a directed net, with $x_i\leq x$ for all $i\in I$. The following are equivalent:
\item[(1)] $x_i\to x$ strongly (if and only if $\sigma$-strongly as $\|x_i\|_\I\leq \|x\|_\I$ for all $i$)
\item[(2)] $x_i\to x$ weakly (if and only if $\sigma$-weakly as $\|x_i\|_\I\leq \|x\|_\I$ for all $i$)
\item[(3)] $x_i\nearrow x$, i.e., $x_i(\omega_\xi)\nearrow x(\omega_\xi)$ for all $\xi\in H$,
\item[(4)] $x_i(\omega_\xi)\nearrow x(\omega_\xi)$ for all $\xi$ in a dense subspace $D$ of $H$.
\end{lem}
\begin{proof}
Clearly $(1)\Rightarrow(2)\Rightarrow(3)\Rightarrow(4)$. 
\itm{(3)\Rightarrow (1)}
Suppose $(x-x_i)(\omega_\xi)\to 0$ for all $\xi \in H$. Then $\|\sqrt{x-x_i}\xi\|_2\to 0$, so $\sqrt{x-x_i}\to 0$ strongly. Hence $x_i\to x$ strongly as multiplication is strongly continuous on bounded sets.
\itm{(4)\Rightarrow(3)}
Choose an orthonormal basis $\{e_n\}_{n\geq 1}\subset D$ for $H$. Suppose $\xi=\sum_n \lambda_n e_n \in H\setminus\{0\}$, and let $\varepsilon>0$. Then there is an $N>0$ such that
$$
\xi_N := \sum_{n>N} \lambda_n e_n \Longrightarrow \|\xi_N\|_2^2 = \sum_{n>N} |\lambda_n|^2 <\frac{\varepsilon^2}{16\|x\|_\I^2 \|\xi\|_2^2}.
$$
For $n=1,\dots,N$, there are $i_n\in I$ such that $i > i_n$ implies 
$$
|\langle (x-x_i) \lambda_n e_n,\xi\rangle|\leq \|(x-x_i) \lambda_n e_n\|_2\|\xi\|_2< \frac{\varepsilon}{2^{n+1}}. 
$$
Now choose $i'>i_n$ for all $n=1,\dots,N$. We calculate that for $i>i'$, 
\begin{align*}
(x-x_i)(\omega_\xi)
&= \langle (x-x_i) \xi,\xi\rangle 
\leq \sum_{n=1}^N |\langle (x-x_i) \lambda_n e_n,\xi\rangle| + \left|\left\langle (x-x_i)\xi_N,\xi\right\rangle\right|\\
& \leq \sum_{n=1}^N \frac{\varepsilon}{2^{n+1}} + \left|\left\langle x\xi_N,\xi\right\rangle\right|+\left|\left\langle x_i\xi_N,\xi\right\rangle\right|
\leq \sum_{n=1}^N \frac{\varepsilon}{2^{n+1}} +2\|x\|_\I \|\xi_N\|_2 \|\xi\|_2\\
&<  \frac{\varepsilon}{2}+2\mu \frac{\varepsilon}{4\|x\|_\I \|\xi\|_2} \|\xi\|_2 = \varepsilon.
\end{align*}
As $\varepsilon$ was arbitrary, we are finished.
\end{proof}

\begin{lem}\label{lem:CommuteStrong}
If $x,y \in M^+$, and $(x_i)_{i\in I},(y_j)_{j\in J}\subset M^+$ are directed nets of increasing operators such that 
\item[$\bullet$] any two elements in $\{x,y\}\cup \set{x_i}{i\in I}\cup \set{y_j}{j\in J}$ commute and
\item[$\bullet$] $x_i\nearrow x$ and $y_j\nearrow y$, 
\item then $x_iy_j\nearrow xy$ (and Lemma \ref{lem:WeakStrong} applies).
\end{lem}

\begin{lem}\label{lem:binormal}
Suppose $x\in X$ and $y\in Y_0$. Then $x\otimes_A y\colon H\otimes_A K\to H\otimes_A K$ given by the unique extension of $\xi\otimes \eta\mapsto (x\xi)\otimes(y\eta)$ where $\xi\in D(H_A)$ and $\eta\in D(\sb{A}K)$ is well-defined and bounded, and $\|x\otimes_A y\|_\I\leq \|x\|_\I\|y\|_\I$. Hence the $*$-algebra map $x\odot_\C y\mapsto x\otimes_A y$ is a binormal representation of $X\odot_\C Y_0$ on $H\otimes_A K$.
\end{lem}
\begin{proof}
\item[(1)]
Fix $\xi_1,\dots,\xi_k\in D(H_A)$ and $\eta_1,\dots,\eta_k\in D(\sb{A}K)$, and let $\xi=(\xi_1,\dots,\xi_k)$ and $\eta=(\eta_,\dots,\eta_k)$. Since the matrices $m=({\sb{A}\langle} y\eta_i,y\eta_j\rangle)_{i,j}, n=(\langle \xi_j,\xi_i\rangle_A)_{i,j}\in M_k(A)$ are positive (see Lemma 1.8 of \cite{MR1424954}), we have
\begin{align*}
\left\| \sum_{i=1}^k (x\xi_i)\otimes(y\eta_i)\right\|_2^2
& = \sum_{i,j=1}^k \langle (x\xi_i)\otimes(y\eta_i), (x\xi_j)\otimes(y\eta_j)\rangle \\
& = \sum_{i,j=1}^k \langle (x\xi_i){\sb{A}\langle} y\eta_i,y\eta_j\rangle, (x\xi_j)\rangle  
= \langle (x\xi) n,( x\xi)\rangle 
= \|(x\xi)n^{1/2}\|_2^2\\
& =\| x(\xi n^{1/2})\|_2^2\leq \|x\|_\I^2 \|\xi n^{1/2}\|_2^2
= \|x\|_\I^2 \sum_{i,j=1}^k \langle \xi_i{\sb{A}\langle} y\eta_i,y\eta_j\rangle, \xi_j\rangle \\
& = \|x\|_\I^2 \sum_{i,j=1}^k \langle \langle\xi_j,\xi_i\rangle_A ( y\eta_i),(y\eta_j)\rangle 
= \|x\|_\I^2  \|m^{1/2}(y\eta)\|_2^2\\
& = \|x\|_\I^2  \|y(m^{1/2}\eta)\|_2^2
\leq \|x\|_\I^2 \|y\|_\I^2 \|m^{1/2}\eta\|_2^2\\
& = \|x\|_\I^2\|y\|_\I^2\left\| \sum_{i=1}^k \xi_i\otimes\eta_i\right\|_2^2.
\end{align*}
\item[(2)]
That $x\mapsto x\otimes_A 1_K$ is a normal representation of $X$ follows from the density of $D(H_A)\otimes_A K$ and (4) of Lemma \ref{lem:WeakStrong}. Similar for $y\mapsto 1_H\otimes_A y$. 
\end{proof}

\begin{nota}\label{nota:spectral}
Let $\cB$ be the Borel $\sigma$-algebra of subsets of $[0_\R,\I_\R]$. For a spectral measure $E \colon \cB\to P(H)$, we use the conventions $E_\lambda=E([0,\lambda])$, so $E_\I=1$, and $E^\I=E(\{\I\})$ (in general, our spectral measures on $\cB$ have non-trivial mass at $\I$). 
\end{nota}

\begin{lem}\label{lem:SpectralConvergence}
Suppose $E \colon \cB\to P(X)\subset B(H_A) $ is a spectral measure. Suppose $f\colon [0,\I]\to [0,\I)$ is a bounded Borel-measurable function, and $(\varphi_n)$ is a sequence of positive simple functions increasing pointwise to $f$. Then
$$
\int_0^\I f(\lambda)\, dE_\lambda := \sup_n\int_0^\I \varphi_n(\lambda) \, dE_\lambda
$$
is well defined.
\end{lem}
\begin{proof}
Suppose $\xi\in H$. Then as $\omega_\xi$ is normal,  $\omega_\xi \circ E$ is a Borel measure, and
$$
\int_0^\I f(\lambda)\, d(\omega_\xi(E_\lambda))=\sup_n \int_0^\I \varphi_n(\lambda) \, d(\omega_\xi(E_\lambda))
$$
is independent of the choice of positive simple functions $\varphi_n$ increasing to $f$.
\end{proof}

\begin{prop}\label{prop:TensorSpectral}
Suppose 
\begin{align*}
E \colon \cB&\longrightarrow P(X)\subset B(H_A) \text{ and }\\
F \colon \cB&\longrightarrow P(Y_0)\subset B(\sb{A}K)
\end{align*}
are spectral measures.
\item[(1)] The map $E \otimes_A F\colon \cB\otimes \cB\longrightarrow P(X\otimes_A Y_0)$ by
$$
(I_1,I_2)\longmapsto \int_{I_1\times I_2} \, d(E_\lambda\otimes_A F_\mu) := E(I_1)\otimes_A F (I_2)
$$
extends uniquely to a spectral measure by countable additivity.
\item[(2)] If $\varphi,\psi\colon [0,\I]\to [0,\I)$ are positive simple functions, then
{\fontsize{10}{10}{
$$
\int_{0}^\I\int_0^\I \varphi(\lambda) \psi(\mu) \, d(E_\lambda\otimes_A F_\mu)
= \left(\int_0^\I \varphi(\lambda) \, dE_\lambda\right) \otimes_A \left( \int_0^\I \psi(\mu) \, dF_\mu \right)\in X\otimes_A Y_0.
$$
}}
\item[(3)]
If $f,g$ are bounded, $\cB$-measurable functions and $(\varphi_m),(\psi_n)$ are sequences of positive simple functions increasing to $f,g$, then
{\fontsize{10}{10}{
$$
\sup_{m,n} \int_{0}^\I\int_0^\I \varphi_m(\lambda) \psi_n(\mu) \, d(E_\lambda\otimes_A F_\mu)
= \left(\int_0^\I f(\lambda) \, dE_\lambda\right) \otimes_A \left( \int_0^\I g(\mu) \, dF_\mu \right)\in X\otimes_A Y_0.
$$
}}
\end{prop}
\begin{proof}
\item[(1)]
One simply needs to check countable additivity (pointwise on $H\otimes_A K$), which follows from countably additivity on products of intervals, which is straightforward.
\item[(2)] Obvious.
\item[(3)] Immediate from (2) together with Lemmas \ref{lem:CommuteStrong} and \ref{lem:SpectralConvergence}.
\end{proof}

\begin{lem}\label{lem:TensorSpectralAssociative}
The relative tensor product of spectral measures as in Proposition \ref{prop:TensorSpectral} is associative, i.e., if
\begin{align*}
E \colon \cB&\longrightarrow P(X)\subset B(H_A), \\
F\colon \cB&\longrightarrow P(Y)\subset B(\sb{A}K_B), \text{ and }\\
G \colon \cB&\longrightarrow P(Z)\subset B(\sb{B}L)
\end{align*}
are spectral measures on $\cB$, then $(E\otimes_A F)\otimes_B G=E\otimes_A(F\otimes_B G)$. Moreover, if $f,g,h\colon [0,\I]\to [0,\I)$ are bounded $\cB$-measurable functions, and $(\varphi_m),(\psi_n),(\gamma_k)$ are positive simple functions increasing to $f,g,h$ respectively, then
{\fontsize{10}{10}{
\begin{align*}
&\sup_{m,n,k} \int_{0}^\I\int_0^\I\int_0^\I \varphi_m(\lambda) \psi_n(\mu)  \gamma_\ell(\nu) \, d(E_\lambda\otimes_A F_\mu\otimes_B G_\nu)=\\
&\hspace{.5in}  = \left(\int_0^\I f(\lambda) \, dE_\lambda\right) \otimes_A \left( \int_0^\I g(\mu) \, dF_\mu \right)\otimes_B \left( \int_0^\I h(\nu) \, dG_\nu \right)\in X\otimes_A Y\otimes_B Z.
\end{align*}
}}
\end{lem}
\begin{proof}
Immediate from associativity of the relative tensor product and Proposition \ref{prop:TensorSpectral}.
\end{proof}

\begin{defn}\label{defn:TensorCones}
Suppose $x\in \widehat{X^+}$ and $y\in \widehat{Y_0^+}$ have spectral resolutions
$$
x = \int_{[0,\I)} \lambda \, d E_\lambda + \I E^\I\text{ and } y = \int_{[0,\I)} \mu \, dF_\mu+\I F^\I
$$
(recall Notation \ref{nota:spectral}). Then
\begin{align*}
E \colon \cB&\longrightarrow P(X)\subset B(H_A) \text{ and }\\
F \colon \cB&\longrightarrow P(Y_0)\subset B(\sb{A}K)
\end{align*}
are two spectral measures as in Proposition \ref{prop:TensorSpectral}. For $m,n\in\N$, set
$$
x_m = \int_{[0,m]} \lambda \, dE_\lambda +mE^\I \text{ and } y_n = \int_{[0,n]} \mu \, dF_\mu + nF^\I.
$$
Applying Lemma \ref{lem:QuadraticForm} to the directed set
$$
\cF=\set{x_m\otimes_A y_n}{m,n\in \N}\subset (X\otimes_A Y_0)^+,
$$
we get a positive, self-adjoint operator affiliated to $X\otimes_A Y_0$ and densely-defined in an affiliated subspace of $X\otimes_A Y_0$. We denote this operator as $x\otimes_A y$.
\end{defn}

\begin{rem}\label{rem:ThreeProjections} Assume the notation of Definition \ref{defn:TensorCones}. When we work with $x\otimes_A y$, it helps to consider the following $3$ projections:
\begin{align*}
p_0 & = (E_0\otimes_A 1_K)\vee (1_H\otimes F_0),\\
p_\I & = \bigg((1-E_0)\otimes_A F^\I\bigg)+\bigg(E^\I\otimes_A(1-F_0)\bigg)+E^\I\otimes_A F^\I,\text{ and} \\
p_f &= \sup_{\lambda,\mu<\I} E_\lambda\otimes_A F_\mu = (1-E^\I)\otimes_A (1-F^\I),
\end{align*}
which we should think of as having the following ``supports" given by the shaded areas in $[0_\R,\I_\R]^2$ below:
$$
p_0=
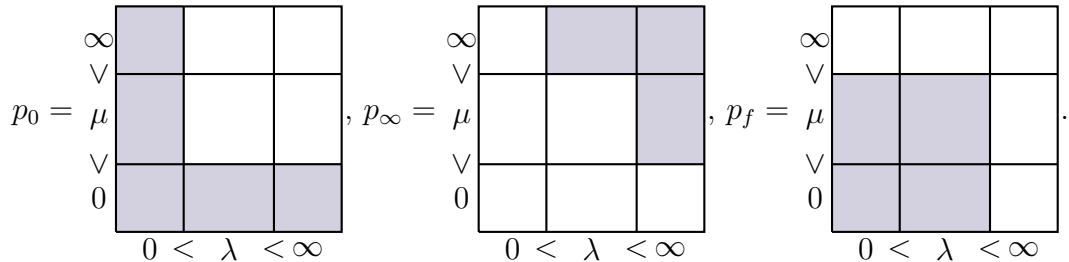
\begin{tikzpicture}[rectangular,baseline=0cm]
	\clip (2.05,2.05) --(-2.6,2.05) -- (-2.6,-2.6) -- (2.05,-2.6);
	\filldraw[shaded] (-2,-2)--(-2,2)--(-.8,2)--(-.8,-.8)--(2,-.8)--(2,-2);
	\filldraw[shaded] ;
	\draw[thick] (-2,-.8)--(2,-.8);
	\draw[thick] (-2,.8)--(2,.8);
	\draw[thick] (.8,-2)--(.8,2);
	\draw[thick] (-.8,-2)--(-.8,2);
	\draw[thick] (-2,-2)--(-2,2)--(2,2)--(2,-2)--(-2,-2);
	\node at (-2.3,-1.4) {$0$};
	\node at (-2.3,-.8) {$\vee$};	
	\node at (-2.3,0) {$\mu$};
	\node at (-2.3,.8) {$\vee$};
	\node at (-2.3,1.4) {$\I$};
	\node at (-1.4,-2.3) {$0$};
	\node at (-.8,-2.3) {$<$};	
	\node at (0,-2.3) {$\lambda$};
	\node at (.8,-2.3) {$<$};
	\node at (1.4,-2.3) {$\I$};
\end{tikzpicture}, \, p_\I=
\begin{tikzpicture}[rectangular,baseline=0cm]
	\clip (2.05,2.05) --(-2.6,2.05) -- (-2.6,-2.6) -- (2.05,-2.6);
	\filldraw[shaded] (.8,-.8)--(.8,.8)--(-.8,.8)--(-.8,2)--(2,2)--(2,-.8);
	\filldraw[shaded] ;
	\draw[thick] (-2,-.8)--(2,-.8);
	\draw[thick] (-2,.8)--(2,.8);
	\draw[thick] (.8,-2)--(.8,2);
	\draw[thick] (-.8,-2)--(-.8,2);
	\draw[thick] (-2,-2)--(-2,2)--(2,2)--(2,-2)--(-2,-2);
	\node at (-2.3,-1.4) {$0$};
	\node at (-2.3,-.8) {$\vee$};	
	\node at (-2.3,0) {$\mu$};
	\node at (-2.3,.8) {$\vee$};
	\node at (-2.3,1.4) {$\I$};
	\node at (-1.4,-2.3) {$0$};
	\node at (-.8,-2.3) {$<$};	
	\node at (0,-2.3) {$\lambda$};
	\node at (.8,-2.3) {$<$};
	\node at (1.4,-2.3) {$\I$};
\end{tikzpicture},\, p_f=
\begin{tikzpicture}[rectangular,baseline=0cm]
	\clip (2.05,2.05) --(-2.6,2.05) -- (-2.6,-2.6) -- (2.05,-2.6);
	\filldraw[shaded] (-2,-2)--(-2,.8)--(.8,.8)--(.8,-2);
	\filldraw[shaded] ;
	\draw[thick] (-2,-.8)--(2,-.8);
	\draw[thick] (-2,.8)--(2,.8);
	\draw[thick] (.8,-2)--(.8,2);
	\draw[thick] (-.8,-2)--(-.8,2);
	\draw[thick] (-2,-2)--(-2,2)--(2,2)--(2,-2)--(-2,-2);
	\node at (-2.3,-1.4) {$0$};
	\node at (-2.3,-.8) {$\vee$};	
	\node at (-2.3,0) {$\mu$};
	\node at (-2.3,.8) {$\vee$};
	\node at (-2.3,1.4) {$\I$};
	\node at (-1.4,-2.3) {$0$};
	\node at (-.8,-2.3) {$<$};	
	\node at (0,-2.3) {$\lambda$};
	\node at (.8,-2.3) {$<$};
	\node at (1.4,-2.3) {$\I$};
\end{tikzpicture}.
$$
\item[$\bullet$] These three projections commute with $x\otimes_A y$. 
\item[$\bullet$] $\Dom((x\otimes_A  y)^{1/2})\subset (1-p_\I)(H\otimes_A K)$, and $(x\otimes_A y)(1-p_\I)$ is densely defined on $(1-p_\I)(H\otimes_A K)$.
\item[$\bullet$] $(x\otimes_A y)p_f = \sup_{m,n<\I} \int_{[0,m]}\int_{[0,n]} \lambda \mu \, d(E_\lambda\otimes_A F_\mu)$.
\item[$\bullet$]$(x\otimes_A y)p_0=0$.
\end{rem}

\begin{lem}\label{lem:LessThan}
Let $x\in\widehat{X^+}$ and $y\in\widehat{Y_0^+}$, and assume the notation of Definition \ref{defn:TensorCones} and Remark \ref{rem:ThreeProjections}. Suppose $x'\in X^+$, $y'\in Y_0^+$ with $x'\leq x$ and $y'\leq y$. Then 
\item[(1)] $(x'\otimes_A y')p_0=p_0(x'\otimes_A y')=0$,
\item[(2)] for all $\xi \in H\otimes_A K$, $(x\otimes_A y)(\omega_\xi)=(x\otimes_A y)(\omega_{(1-p_0)\xi})$, and
\item[(3)] $x'\otimes_A y'\leq x\otimes_A y$.
\end{lem}
\begin{proof}
\item[(1)] 
Suppose $\eta\in D((E_0H)_A)$ and $\kappa\in D(\sb{A} K)$ (recall $E_0\in X$ and $F^\I\in Y_0$). Then since $x'\leq x$, we must have 
$$
\|(x')^{1/2}\eta\|_H^2=\langle x'\eta,\eta\rangle = x'(\omega_{\eta})\leq x(\omega_\eta)=x(\omega_{E_0\eta})=xE_0(\omega_\eta)=0.
$$
But this implies $x'\eta=0$. Hence we have
$$
(x'\otimes_A y') (\eta\otimes\kappa) = 0.
$$
Similarly, for all $\eta\in D(H_A)$ and $\kappa\in D(\sb{A}(F_0 K))$, $(x'\otimes_A y')(\eta\otimes\kappa)=0$. By density of $D(H_A)\otimes_A D(\sb{A}K)$, we have $(x'\otimes_A y') p_0=0$. Taking adjoints gives $p_0(x'\otimes_A y')=0$.
\item[(2)]
By (1), for all $m,n>0$, $p_0(x_m\otimes_A y_n)=(x_m\otimes_A y_n)p_0=0$, so
\begin{align*}
(x\otimes_A y)(\omega_\xi)
& = \sup_{m,n} (x_m\otimes_A y_n)(\omega_\xi)\\
& = \sup_{m,n} \bigg( (x_m\otimes_A y_n)(\omega_{(1-p_0)\xi})+\langle (x_m\otimes_A y_n)p_0\xi,p_0\xi\rangle\\
& \hspace{.2in}+\langle (x_m\otimes_A y_n) p_0 \xi,\xi\rangle + \langle (x_m\otimes_A y_n) \xi,p_0\xi\rangle\bigg)\\
& = \sup_{m,n} (x_m\otimes_A y_n )(\omega_{(1-p_0)\xi})
= (x\otimes_A y)(\omega_{(1-p_0)\xi}).
\end{align*}
\item[(3)]
By (2), it suffices to show that for all $\xi \in \Dom((x\otimes_A y)^{1/2})$ with $\xi=p_f\xi$,
$$
\bigg(p_f(x'\otimes_A y')p_f\bigg)(\omega_\xi)=(x'\otimes_A y')(\omega_\xi) 
\leq (x\otimes_A y) (\omega_\xi)=\bigg(p_f(x\otimes_A y)p_f\bigg)(\omega_{\xi}).
$$
Fix such a $\xi$, and let $\varepsilon>0$. As $E_\lambda\otimes_A F_\mu \to p_f$ strongly as $\lambda,\mu\to \I$ from below, there is an $N>0$ such that for all $\lambda,\mu >N$,
$$
\bigg(p_f(x'\otimes_A y')p_f - (E_\lambda x' E_{\lambda}\otimes_A F_\mu y' F_{\mu} )\bigg) (\omega_{\xi}) < \varepsilon.
$$
Since $x'\leq x$ and $y'\leq y$, we have $E_N x'E_N\leq xE_N$, $F_N y'F_N\leq yF_N$ by Lemma \ref{lem:VectorStates}, so $E_N x' E_N \otimes_A F_N y' F_N\leq xE_N \otimes_A yF_N$ as all these operators mutually commute. Hence 
\begin{align*}
\bigg(p_f(x'\otimes y')p_f\bigg)(\omega_{\xi}) 
& = \bigg(p_0(x_m\otimes_A y_n)p_0 -(E_Nx' E_{N}\otimes_A F_N y' F_{N} )\bigg)(\omega_{\xi})\\
&\hspace{1in}+(E_N x' E_{N}\otimes_A F_N y' F_{N} )(\omega_\xi) \\
& < \varepsilon +(x E_{N}\otimes_A yF_{N} )(\omega_{\xi}) \leq \varepsilon +(x \otimes_A y )(\omega_{\xi}).
\end{align*}
Since $\varepsilon$ was arbitrary, the result follows.
\end{proof}

\begin{lem}\label{lem:corners}
Suppose $(x_j')_{j\in J}\subset \widehat{X^+}$ increases to $x\in \widehat{X^+}$. Suppose $p,q\in P(X)$ are spectral projections of $x$ such that $p+q=1$. Then $\langle x_j' p\xi,q\xi\rangle \to 0$ for all $\xi\in \Dom(x^{1/2})$.
\end{lem}
\begin{proof}
For $k=0,1,2,3$, $p\xi + i^k q\xi \in \Dom(x^{1/2})\subseteq \Dom((x_j')^{1/2})$ for all $j\in J$. Since $x_j'$ increases to $x$, by polarization 
\begin{align*}
\lim_{j\in J} \langle (x_j')^{1/2} p\xi,(x_j')^{1/2}q\xi\rangle 
&= \lim_{j\in J} \frac{1}{4}\sum_{k=0}^3 i^k x_j'(\omega_{p\xi + i^k q\xi})
= \frac{1}{4} \sum_{k=0}^3 i^k x(\omega_{p\xi + i^k q\xi})\\
& = \langle x^{1/2} p\xi, x^{1/2} q\xi\rangle 
= 0
\end{align*}
as $p,q$ commute with $x^{1/2}$.
\end{proof}

\begin{thm}\label{thm:increase}
Let $x\in\widehat{X^+}$ and $y\in\widehat{Y_0^+}$, and assume the notation of Definition \ref{defn:TensorCones} and Remark \ref{rem:ThreeProjections}. Suppose there are sequences $(x_m')\subset X^+$, $(y_n')\subset Y_0^+$ which increase to $x, y$ respectively. Then $x_m'\otimes_A y_n'$ increases to $x\otimes_A y$.
\end{thm}
\begin{proof}
\itt{Case 1} 
Suppose $\xi \notin  \Dom((x\otimes_A y)^{1/2})$ and $M>0$. Since $\sup_{m,n} x_m\otimes_A y_n=x\otimes_A y$, there is an $N_0\in\N$ such that for all $m,n\geq N_0$, $(x_m\otimes_A y_n)(\omega_\xi)>M$. Since $p_0\xi\neq \xi$ by Lemma \ref{lem:LessThan}, we must have
$$
(1_H\otimes_A (1_K-F_0)) \xi \neq 0 \text{ and } ((1_H-E_0)\otimes_A 1_K)\xi \neq 0.
$$ 
\itt{Claim} There is an $N_1>N_0$ such that $(x_m'\otimes 1_K)\xi\neq 0 \neq (1_H\otimes_A y_n')\xi$ for all $m,n>N_1$. 
\begin{proof}
We prove the second non-equality. Suppose not. Then for each $n>0$, there is an $k>n$ such that $(1\otimes_A y_k')\xi = 0$. But then
$$
(1_H\otimes_A y_n')(\omega_\xi)\leq (1_H\otimes_A y_k')(\omega_\xi)=0,
$$
so $(1_H\otimes_A y_n')\xi=0$ for all $n\in\N$. Since $(1_H\otimes_A (1-F_0))\xi\neq 0$, and $D(H_A)\otimes_A D(\sb{A}((1-F_0)K))$ is dense in $H\otimes_A ((1_K-F_0)K)$, there is an $\eta\in D(H_A)$ such that $L_\eta^*\xi \in ((1_K-F_0)K)\setminus\{0\}$ and $L_\eta L_\eta^* \leq 1_H\otimes_A 1_K$. Now since $y_n'$ increases to $y$, and $y(\omega_{L_\eta^* \xi})>0$, there is an $N'>0$ such that for all $n>N'$,
\begin{align*}
0 &< y_n'(\omega_{L_\eta^* \xi})
= (L_\eta y_n L_\eta^*)(\omega_\xi)
= \bigg(L_\eta L_\eta^* (1_H\otimes_A y_n')\bigg)(\omega_\xi) 
\leq (1_H\otimes_A y_n')(\omega_\xi) =0,
\end{align*}
a contradiction.
\end{proof}
Choose $N_1$ as in the claim, and suppose $n>N_1$. Let $\{\alpha_i\}\subset D(\sb{A}K)$ be an $\sb{A}K$-basis, and let $\eta = (1_H\otimes_A (y_{N_1})^{1/2})\xi\neq 0$, and note $(x_{N_1}\otimes_A 1_K)(\omega_\eta)>M$. Then 
\begin{align*}
M & < (x_{N_1}\otimes1_K)(\omega_\eta) 
= \left((x_{N_1}\otimes_A 1_K) \left( \sum_i R_{\alpha_i} R_{\alpha_i}^*\right)\right)(\omega_\eta)
= \sum_i (R_{\alpha_i} (x_{N_1}) R_{\alpha_i}^*)(\omega_\eta),
\end{align*}
so there is an $N_2>0$ such that
$$
M 
< \sum_{i=1}^{N_2} (R_{\alpha_i} x_{N_1} R_{\alpha_i}^*)(\omega_\eta)
= \sum_{i=1}^{N_2} x_{N_1} (\omega_{R_\alpha^* \eta})
\leq \sum_{i=1}^{N_2} x (\omega_{R_\alpha^* \eta}).
$$
Now as $x_m'$ increases to $x$, there is an $N_3>N_1$ such that $m>N_3$ implies
\begin{align*}
M
& < \sum_{i=1}^{N_2} x_m' (\omega_{R_\alpha^* \eta})
= \sum_{i=1}^{N_2} (R_{\alpha_i} x_m' R_{\alpha_i}^*)(\omega_\eta)
\leq \sum_{i} (R_{\alpha_i} x_m' R_{\alpha_i}^*)(\omega_\eta) \\
& = \left((x_m'\otimes_A 1_K) \left(\sum_i R_{\alpha_i} R_{\alpha_i}^*\right)\right)(\omega_\eta)
= (x_m'\otimes y_{N_1})(\omega_\xi).
\end{align*}
Repeating the above argument for $y_n'$ yields an $N_4$ such that $m,n>N_4$ implies $M < (x_m'\otimes_A y_n') (\omega_\xi)$.
\itt{Case 2}
Suppose $\xi\in \Dom((x\otimes_A y)^{1/2})$. Then $\xi=(1-p_\I)\xi$. We want to show
$$
\sup_{m,n}( x_m'\otimes_A y_n' )(\omega_\xi) = (x\otimes_A y)(\omega_\xi) =\sup_{m,n} (x_m\otimes_A y_n)(\omega_\xi),
$$
so by Lemma \ref{lem:LessThan}, we may assume $\xi=(1-p_0)\xi$, and thus $\xi=p_f\xi$. Let $\varepsilon>0$. Since
$$
p_f(x\otimes_A y)p_f = \sup_{\lambda,\mu<\I} x E_{\lambda}\otimes_A y F_{\mu},
$$
there is an $N_0\in\N$ such that for all $\lambda,\mu\geq N_0$,
$$
\bigg((x\otimes_A y) - (x E_{\lambda}\otimes_A y F_{\mu} )\bigg) (\omega_\xi) < \frac{\varepsilon}{4}.
$$
By Lemma \ref{lem:LessThan}, $x_m'\otimes_A y_n' \leq x\otimes_A y$ for all $m,n$, so using Lemma \ref{lem:VectorStates}, we have
\begin{align*}
\bigg((x_m'\otimes_Ay_n') - (E_{N_0} x_m' E_{N_0})\otimes_A (F_{N_0} y_n' F_{N_0})]\bigg)
& \leq \bigg((x\otimes_A y) - (x E_{N_0}\otimes_A y F_{N_0} )\bigg) \text{ and }\\
E_{N_0} x_m' E_{N_0}\otimes_A F_{N_0} y_n'F_{N_0}
& \leq x E_{N_0}\otimes_A y F_{N_0} 
\end{align*}
by multiplying on either side by $1_{H\otimes_A K}-(E_{N_0}\otimes_A F_{N_0})$ and $E_{N_0}\otimes_A F_{N_0}$ respectively.
Now since $x_m',y_n'$ increase to $x,y$ respectively,  by Lemma \ref{lem:VectorStates}, $E_{N_0} x_m' E_{N_0}, F_{N_0} y_n' F_{N_0}$ increases to $x E_{N_0}, y F_{N_0}$ respectively. Thus $E_{N_0} x_m' E_{N_0}\otimes_A F_{N_0} y_n' F_{N_0}$ increases to $x E_{N_0}\otimes_A y F_{N_0}$ by Lemma \ref{lem:CommuteStrong}, and there is an $N_1>N_0$ such that for all $m,n\geq N_1$,
$$
\bigg((x E_{N_0}\otimes_A y F_{N_0}) - (E_{N_0} x_m' E_{N_0}\otimes_A F_{N_0} y_n' F_{N_0}) \bigg)(\omega_\xi) < \frac{\varepsilon}{4}.
$$
By Lemma \ref{lem:corners}, there is an $N_2>N_1$ such that for all $m,n>N_2$, 
$$
\left|\bigg\langle (x_m'\otimes y_n')(1_{H\otimes_A K}-E_{N_0}\otimes_A F_{N_0}) \xi,(E_{N_0}\otimes_A F_{N_0})\xi\bigg\rangle\right|< \frac{\varepsilon}{4}.
$$
Now we calculate that for all $m,n> N_2$,
{\fontsize{10}{10}{
\begin{align*}
(x\otimes_A y - x_m'\otimes y_n')(\omega_\xi)
& = 
(1-E_{N_0}\otimes_A F_{N_0})(x\otimes_A y - x_m'\otimes y_n')(1-E_{N_0}\otimes_A F_{N_0})(\omega_\xi)\\
&\hspace{.4in} +(1_{H\otimes_A K}-E_{N_0}\otimes_A F_{N_0})(x\otimes_A y - x_m'\otimes y_n')(E_{N_0}\otimes_A F_{N_0})(\omega_\xi)\\
&\hspace{.4in} +(E_{N_0}\otimes_A F_{N_0})(x\otimes_A y - x_m'\otimes y_n')(1_{H\otimes_A K}-E_{N_0}\otimes_A F_{N_0})(\omega_\xi)\\
&\hspace{.4in} +(E_{N_0}\otimes_A F_{N_0})(x\otimes_A y - x_m'\otimes y_n')(E_{N_0}\otimes_A F_{N_0})(\omega_\xi)\\
& \leq \bigg((x\otimes_A y) - (x E_{N_0}\otimes_A y F_{N_0} )\bigg) (\omega_\xi)\\
&\hspace{.4in} +|((1_{H\otimes_A K}-E_{N_0}\otimes_A F_{N_0})(x_m'\otimes_A y_n')(E_{N_0}\otimes_A F_{N_0})(\omega_\xi)|\\
&\hspace{.4in} +|(E_{N_0}\otimes_A F_{N_0})(x_m'\otimes_A y_n')(1-E_{N_0}\otimes_A F_{N_0})(\omega_\xi)|\\
&\hspace{.4in} +\bigg((x E_{N_0}\otimes_A y F_{N_0}) - (E_{N_0} x_m'E_{N_0}\otimes_A F_{N_0} y_n' F_{N_0}) \bigg)(\omega_\xi) \\
& < \frac{\varepsilon}{4}+\frac{\varepsilon}{4}+\frac{\varepsilon}{4}+\frac{\varepsilon}{4} = \varepsilon.
\end{align*}}}
\end{proof}

\begin{cor}\label{cor:associative}
If $x\in \widehat{X^+}$, $y\in \widehat{Y^+}$, and $z\in \widehat{Z^+}$, then $(x\otimes_A y)\otimes_B z = x\otimes_A (y\otimes_B z)$.
\end{cor}
\begin{proof}
Take sequences $(x_m)\subset X^+$, $(y_n)\subset Y^+$, and $(z_\ell)\subset Z^+$ which increase to $x,y,z$ respectively. Then
$$
(x\otimes_A y)\otimes_B z =\sup_{m,n,\ell} (x_m\otimes_A y_n) \otimes_B z_\ell=\sup_{m,n,\ell} x_m\otimes_A (y_n \otimes_B z_\ell) = x\otimes_A (y\otimes_B z).
$$
\end{proof}

\begin{cor}\label{cor:TensorConeMap}
If $x,w \in \widehat{X^+}$, $y\in \widehat{Y_0^+}$, and $\lambda\in [0,\I]$, then $(\lambda x+w)\otimes_A y = \lambda (x\otimes_A y)+(w\otimes_A y)$.
\end{cor}
\begin{proof}
Choose $X^+ \ni x_m,w_n\nearrow x,w \in \widehat{X^+}$ respectively and $\widehat{Y_0^+}\ni y_\ell \nearrow y\in \widehat{Y_0^+}$. Then $(\lambda x_m+w_n)\otimes_A y_\ell = \lambda (x_m\otimes_A y_\ell)+(w_n\otimes_A y_\ell)$, and the result follows by Remark \ref{rem:sup} and Theorem \ref{thm:increase}.
\end{proof}

By taking sups appropriately, and with a little more care, Lemma \ref{lem:LessThan} and Theorem \ref{thm:increase} can be generalized to prove:
\begin{thm}\label{thm:UnboundedIncrease}
Let $x\in\widehat{X^+}$ and $y\in\widehat{Y_0^+}$. Suppose there are nets $(x_i)_{i\in I}\subset \widehat{X^+}$, $(y_j)_{j\in J}\subset \widehat{Y_0^+}$ which increase to $x, y$ respectively. Then $x_i\otimes_A y_j\nearrow x\otimes_A y$.
\end{thm}

%%%%%%%%%%%%%%%%%%%%%%%%%%%%%%%%%%%%%%%%%%%%%%%%%%%%%
\section{The operad $\B\P$}\label{sec:BP}
To show the action of $\B\P$ is well-defined in Theorem \ref{thm:ConePA}, we show each connected/internally connected $\B\P$-tangle has a unique standard form that behaves well under composition, analogous to the methods of \cite{0912.1320}. The existence of this standard form is thoroughly sketched, and the behavior under composition is briefly sketched.

\begin{defn} 
We will again define $\B\P$, an operad of unshaded, oriented tangles up to planar isotopy, but in more detail. First, we require for tangles $\cT\in\B\P$
\item[$\bullet$] $\cT$ has an external disk $D_0$ and internal disks $D_1,\dots, D_s$, each with an even number $2k_i$ of market boundary points and a distinguished interval marked $*$.
\item[$\bullet$] Each boundary point of $\cT$ is connected to exactly one oriented string. Each oriented string is either a closed loop, or it is attached to two distinct boundary points.
\item[$\bullet$] For $i=1,\dots,s$, reading counter-clockwise from $*$, the strings attached to the first $k_i$ boundary points of $D_i$ are oriented away from $D_i$, and the second $k_i$ strings are oriented toward $D_i$,
\item[$\bullet$] Reading counter-clockwise from $*$, the strings attached to the first $k_0$ boundary points of $D_0$ are oriented toward $D_0$, and the second $k_0$ strings are oriented away from $D_0$,
\item
$\B\P$ is the operad generated by the following tangles:
\item[(1)] For $n\geq 0$, the ``Temperley-Lieb" tangle $1_n$ with no inputs and $2n$ boundary points:
$$1_n = \idn{n},$$
\item[(2)] For $n\geq 0$, the unique tangles $t_{n+1},t_{n+1}\op$ with $2n+2$ internal boundary points and $2n$ external boundary points and only one right, left cap respectively:
$$
t_{n+1} =\OperatorValuedWeight{n}{}\,\text{ and }t_{n+1}\op=\OperatorValuedWeightOp{n}{},
$$
\item[(3)] For $m,n\geq 0$, the tangles $\otimes_{m,n}$ with internal disks $D_1,D_2$ with $2m,2n$ internal boundary points and $2(m+n)$ external boundary points as follows:
$$\otimes_{m,n}=\tensor{m}{}{n}{}.$$ 
\item[(4)]
For $n\geq 1$, the tangles  $\tau_n,\tau_n\op$ with two input disks, each with $2n$ internal boundary points, and no external boundary points such that boundary point $m$ of input disk $D_1$ is connected to boundary point $2n-m+1$ of input disk $D_2$ for each $m=1,\dots, 2n$ as follows: 
$$
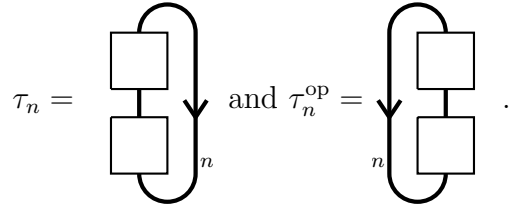
\tau_n = \TraceOfTwo{n}{}{}\, \text{ and } \tau_n\op = \TraceOfTwoOp{n}{}{}.
$$
\item
A tangle $\cT\in \B\P$ with disks $\{D_i\}_{i=0}^s$ and strings $\{S_j\}_{j=1}^t$ is called: 
\item[$\bullet$] \underline{connected} if $\{D_i\}_{i=0}^s\cup\{S_j\}_{j=1}^t$ is a connected in $\R^2$, and
\item[$\bullet$] \underline{internally connected} if $\cT$ has no external boundary points and $\{D_i\}_{i=1}^s\cup\{S_j\}_{j=1}^t$ is  connected in $\R^2$.
\end{defn}

\begin{thm}\label{thm:relations}
The following relations hold in $\B\P$ for $m,n\geq 0$ (compare with Theorem \ref{thm:BPRelations}):
\item[(1)] $t_m t_{m+1}\op = t_{m}\op t_{m+1}$,
\item[(2)] $\otimes_{\ell,m+n}(-,\otimes_{m,n}(-,-))=\otimes_{\ell+m,n}(\otimes_{\ell,m}(-,-),-)$,
\item[(3)] $\otimes_{m,n-1}(-,t_{n}(-))=t_{m+n}(\otimes_{m,n}(-,-))$ and $\otimes_{m-1,n}(t_m\op(-),-)=t_{m+n}\op(\otimes_{m,n}(-,-))$, and
\item[(4)] $\tau_n(\cT_1(-),\cT_2(-))=\tau_n(\cT_2(-),\cT_1(-))$ and $\tau_n\op(\cT_1(-),\cT_2(-))=\tau_n\op(\cT_2(-),\cT_1(-))$ for all $\cT_1,\cT_2\in\B\P$ up to reindexing internal disks.
\end{thm}
\begin{proof}
Clear by drawing pictures.
\end{proof}

\begin{prop}\label{prop:BPprime}
If $\cT\in \B\P$ with external disk $D_0$ and internal disks $\{D_i\}_{i=1}^s$, then the strings can only connect the $D_i$'s in the following ways:
\item[(1)] If $D_i$ is connected by a string to $D_0$ where $1\leq i\leq s$, then any other string connected to $D_i$ must only be connected to $D_i$ or $D_0$.  
\item[(2)] If $D_i$ and $D_j$ are connected by a string where $1\leq i,j\leq s$, then no string of $D_i$ or $D_j$ connects to $D_0$.
\item[(3)] For each $i=0,\dots,s$, if the string $S$ connects boundary points $m$ and $n$ of $D_i$, then $m=2k_i-n+1$. Such a string is called an \underline{$i$-cap} of $\cT$. We call the $i$-cap a \underline{left $i$-cap} if when we connect boundary points $n$ and $2k_i-n+1$ by an imaginary string $S'$ inside $D_i$, the loop $S\cup S'$ contains the distinguished interval of $D_i$. The $i$-cap is a \underline{right $i$-cap} otherwise.
\item[(4)] If the string $S$ connects boundary point $m$ of $D_i$ to boundary point $n$ of $D_j$ where $0\leq i<j\leq s$, then there is a string $S'$ connecting boundary point $2k_i-m+1$ of $D_i$ to boundary point $2k_j-n+1$ of $D_j$. If $i>0$, we call $S\cup S'$ an \underline{$i,j$-cap} of $\cT$. In this case, if we connect boundary points $m$ and $2k_i-m+1$ of $D_i$ and boundary points $n$ and $2k_j-n+1$ of $D_j$ by imaginary strings $S_i, S_j$ inside $D_i,D_j$ respectively, then if the loop $S\cup S'\cup S_i\cup S_j$ contains the distinguished intervals of $D_i$ and $D_j$, then we call the $i,j$-cap a \underline{left $i,j$-cap}. Otherwise, the distinguished intervals of $D_i,D_j$ must lie outside $S\cup S'\cup S_i\cup S_j$, and we have a \underline{right $i,j$-cap}.
\item[(5)] Suppose there is a right (respectively left) $i,j$-cap where $i\neq j$. If $\cS$ is the collection of all disks and strings which can be connected to $D_i$ or $D_j$ through other disks and strings, then we may consider $\cS$ as an internally connected tangle, and all $k,\ell$-caps in $\cS$ ($k\neq \ell$) must form concentric circles. (We will show in Theorem \ref{thm:IJcaps} that all such $k,\ell$-caps are also right (respectively left) caps.) 
\end{prop}
\begin{proof}
Clear by drawing pictures.
\end{proof}

\begin{ex}
The tangle on the left is in $\B\P$ (see the proof of Theorem \ref{thm:IJcaps}), but the tangle on the right is not:
$$
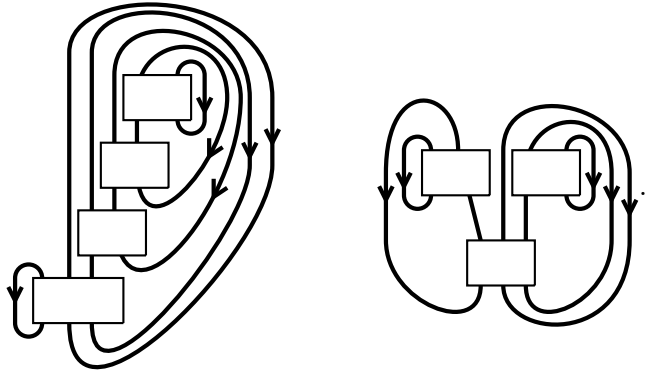
\begin{tikzpicture}[scale=.6,baseline=.5cm]
	\clip (-2.1,5.2) --(4,5.2) -- (4,-3.1) -- (-2.1,-3.1);

	\draw[ultra thick] (-.7,-2)--(-.7,4).. controls ++(90:1.6cm) and ++(90:2.5cm) ..  (3.8,3)--(3.8,1.5).. controls ++(270:1.8cm) and ++(270:3cm) .. (-.7,-2);
	\draw[ultra thick] (-.2,-2)--(-.2,4).. controls ++(90:1.4cm) and ++(90:2.2cm) ..  (3.3,3)--(3.3,1.5).. controls ++(270:1.4cm) and ++(270:2.2cm) .. (-.2,-2);
	\draw[ultra thick] (.3,.5)--(.3,3.5).. controls ++(90:1.6cm) and ++(90:1.6cm) ..  (3.1,3).. controls ++(270:2.4cm) and ++(270:3.2cm) .. (.3,.5);
	\draw[ultra thick] (.8,1.5)--(.8,3).. controls ++(90:1.4cm) and ++(90:1.6cm) ..  (2.8,3) .. controls ++(270:1.4cm) and ++(270:2.2cm) .. (.8,1.5);
	\draw[ultra thick] (1.7,3.5) arc (180:0:.3cm) -- (2.3,2.5) arc (0:-180:.3cm);	
	\draw[ultra thick] (2.15,3)--(2.3,2.7)--(2.45,3);
	\draw[ultra thick] (3.65,2.3)--(3.8,2)--(3.95,2.3);
	\draw[ultra thick] (3.15,2)--(3.3,1.7)--(3.45,2);
	\draw[ultra thick] (2.4,2.1)--(2.4,1.7)--(2.7,1.9);
	\draw[ultra thick] (2.5,1.2)--(2.5,.8)--(2.8,1);
	\draw[ultra thick] (-1.3,-1) arc (0:180:.3cm) -- (-1.9,-2) arc (180:360:.3cm);
	\draw[ultra thick] (-1.75,-1.2)--(-1.9,-1.5)--(-2.05,-1.2);

	\filldraw[thick, unshaded] (.5,-2)--(.5,-1)--(-1.5,-1)--(-1.5,-2)--(.5,-2);
	\filldraw[thick, unshaded] (1,-.5)--(1,.5)--(-.5,.5)--(-.5,-.5)--(1,-.5);
	\filldraw[thick, unshaded] (1.5,1)--(1.5,2)--(0,2)--(0,1)--(1.5,1);
	\filldraw[thick, unshaded] (2,2.5)--(2,3.5)--(.5,3.5)--(.5,2.5)--(2,2.5);
\end{tikzpicture}
\hspace{.5in}
\begin{tikzpicture}[scale=.6,baseline=1.5cm]
	\clip (-2.5,5.2) --(3.3,5.2) -- (3.3,-1) -- (-2.5,-1);
	\draw[ultra thick] (-.2,1.5)--(-.7,3.5).. controls ++(90:1.4cm) and ++(90:2.2cm) ..  (-2.3,3)--(-2.3,1.5).. controls ++(270:1.4cm) and ++(270:1.2cm) .. (-.2,.5);
	\draw[ultra thick] (-2.15,2.7)--(-2.3,2.4)--(-2.45,2.7);
	\draw[ultra thick] (.3,.5)--(.3,3.5).. controls ++(90:1.6cm) and ++(90:1.6cm) ..  (3.1,3)--(3.1,1.5).. controls ++(270:2.4cm) and ++(270:1.2cm) .. (.3,.5);
	\draw[ultra thick] (2.55,2.7)--(2.7,2.4)--(2.85,2.7);
	\draw[ultra thick] (.8,1.5)--(.8,3).. controls ++(90:1.4cm) and ++(90:1.6cm) ..  (2.7,3)--(2.7,1.5) .. controls ++(270:1.4cm) and ++(270:1.2cm) .. (.8,.5);
	\draw[ultra thick] (2.95,2.4)--(3.1,2.1)--(3.25,2.4);
	\draw[ultra thick] (1.7,3.5) arc (180:0:.3cm) -- (2.3,2.5) arc (0:-180:.3cm);	
	\draw[ultra thick] (2.15,3)--(2.3,2.7)--(2.45,3);
	\draw[ultra thick] (-1.3,3.5) arc (0:180:.3cm) -- (-1.9,2.5) arc (180:360:.3cm);
	\draw[ultra thick] (-1.75,3)--(-1.9,2.7)--(-2.05,3);

	\filldraw[thick, unshaded] (0,2.5)--(0,3.5)--(-1.5,3.5)--(-1.5,2.5)--(0,2.5);
	\filldraw[thick, unshaded] (1,.5)--(1,1.5)--(-.5,1.5)--(-.5,.5)--(1,.5);
	\filldraw[thick, unshaded] (2,2.5)--(2,3.5)--(.5,3.5)--(.5,2.5)--(2,2.5);
\end{tikzpicture}.
$$
\end{ex}

\begin{rems}\label{rem:StandardForm}
The following are trivial observations about tangles $\cT\in\B\P$:
\item[(1)] We may draw $\cT$ such that each disk $D_i$ where $i=0,\dots,s$ is a rectangle with $k_i$ points on the top, $k_i$ points on the bottom, and the distinguished interval on the left. When drawing diagrams, we omit the disk $D_0$.
\item[(2)] We may draw all strings which are not part of a cap of $\cT$ (strings that meet $D_0$) as vertical lines, oriented upward. We will assume this orientation in the sequel.
\item[(3)] If $\cT\in\B\P$ is connected, then no two internal disks of $\cT$ are connected by a string. Hence it is possible to draw $\cT$ such that each internal disk $D_i$, $i>0$, is the same vertical size and is on the same horizontal level. Moreover, the ordering of the internal disks from left to right is unique.
\end{rems}

\begin{thm}\label{thm:ConnectedStandardForm}
Suppose $\cT\in\B\P$ is a connected tangle with $s>0$ input disks $D_j$, each with $2k_j$ boundary points. Suppose further that $\cT$ is in the form afforded by Remarks \ref{rem:StandardForm}, $\cT$ has no strings which connect $D_0$ to $D_0$, and that the $D_j$'s are numbered from left to right in $\cT$. Then $\cT$ can be written in a unique standard form 
\begin{align*}
&\otimes_{m_1,\sum_{j>1} m_j}( t_{m_1+1}\op\cdots t_{m_1+v_1}\op t_{m_1+v_1+1}\cdots t_{m_1+v_1+w_1}(-),\\
&\otimes_{m_2,\sum_{j>2} m_j}(t_{m_2+1}\op\cdots t_{m_2+v_2}\op t_{m_2+v_2+1}\cdots t_{m_2+v_2+w_2}(-),\cdots\\
&\otimes_{m_{s-1},m_s}(t_{m_{s-1}+1}\op\cdots t_{m_{s-1}+v_{s-1}}\op t_{m_{s-1}+v_{s-1}+1}\cdots t_{m_{s-1}+v_{s-1}+w_{s-1}}(-),\\
&\hspace{1.3in}t_{m_s+1}\op\cdots t_{m_s+v_s}\op t_{m_s+v_s+1}\cdots t_{m_s+v_s+w_s}(-))\cdots))
\end{align*}
such that for all $j=1,\dots,s$,
\item[$\bullet$] $m_j>0$ is half the number of strings connecting $D_j$ to $D_0$,
\item[$\bullet$] $v_j\geq 0$ is the number of left caps on $D_j$,
\item[$\bullet$] $w_j\geq 0$ is the number of right caps on $D_j$, and
\item[$\bullet$] $m_j+v_j+w_j=k_j$.
\item
Moreover, using the relations in Theorem \ref{thm:relations}, any composite of $t_k, t_\ell\op,\otimes_{m,n}$ 
for $k,\ell,m,n>0$ can be written uniquely in the above form.
\end{thm}
\begin{proof}
Clearly each tangle/composite can be written in such a form. For uniqueness, note that $m_j,u_j,v_j$ are completely determined by $\cT$, and the order of the $\otimes_{m,n}$ is  given by ``parenthesizing" the $D_j$'s from right to left (use relation (2) of Theorem \ref{thm:relations}).
\end{proof}

\begin{cor}\label{cor:DecomposeTangle}
Each connected tangle $\cT\in\B\P$ in the form of Remarks \ref{rem:StandardForm} can be written in a similar unique standard form where instead of some of the tangles
$$
t_{m_j+1}\op\cdots t_{m_j+v_j}\op t_{m_j+v_j+1}\cdots t_{m_j+v_j+w_j}(-),
$$
we have tangles $1_{k_j}$ with the condition that we never see two $1_{k_j}$'s in a row, i.e.,
$$
\otimes_{\ell,m}(1_\ell,\otimes_{m,n}(1_m,\cdots)) \text{ or } \otimes_{s-1,s}(1_{s-1},1_s).
$$
This amounts to grouping as many vertical strands connecting $D_0$ to $D_0$ as possible, and treating them as a ``labelled" input disk (as in \cite{math/9909027}).
\end{cor}

\begin{thm}\label{thm:IJcaps}
Suppose $\cT\in\B\P$ is internally connected and has at least two internal disks. Then the $i,j$-caps of $\cT$ form concentric circles by (5) in Proposition \ref{prop:BPprime}. Let $C_1$ be the outermost $i,j$-cap of $\cT$. There is a unique smallest $n\in\N$ and two unique connected tangles $\cT_1,\cT_2\in\B\P$ up to swapping such that:
\item[(1)] if $C_1$ is a right $i,j$-cap, $\cT=\tau_n(\cT_1(-),\cT_2(-))$, and
\item[(2)] if $C_1$ is a left $i,j$-cap, $\cT=\tau_n\op(\cT_1(-),\cT_2(-))$.
\end{thm}
\begin{proof}
We prove (1). We give an algorithm to build $\cT_1$ and $\cT_2$ by partitioning the internal disks of $\cT$ into two sets $U$ and $L$, standing for ``upper" and ``lower." All $i,j$-caps of $\cT$ will be between a $D_i\in U$ and a $D_j\in L$, and we will see they are all right caps. We form $\cT_1$ by putting a box around the $D_i\in U$ together with all ``contractible" $i$-caps, and we form $\cT_2$ by doing the same to the $D_j\in L$.

Before we describe the algorithm, we give an example:
$$
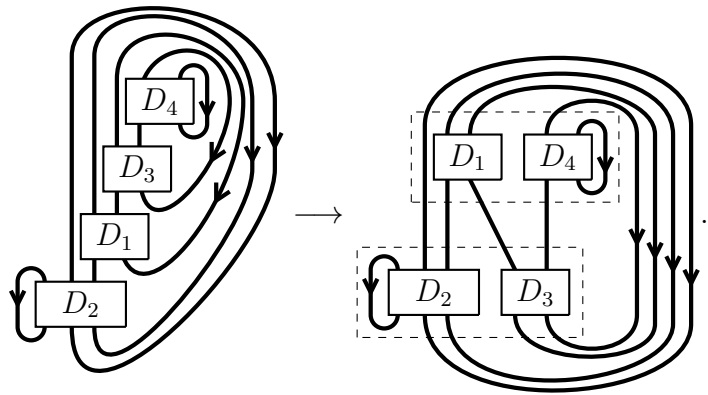
\begin{tikzpicture}[scale=.6,baseline=5]
	\clip (-2.1,5.2) --(4,5.2) -- (4,-3.1) -- (-2.1,-3.1);

	\draw[ultra thick] (-.7,-2)--(-.7,4).. controls ++(90:1.6cm) and ++(90:2.5cm) ..  (3.8,3)--(3.8,1.5).. controls ++(270:1.8cm) and ++(270:3cm) .. (-.7,-2);
	\draw[ultra thick] (-.2,-2)--(-.2,4).. controls ++(90:1.4cm) and ++(90:2.2cm) ..  (3.3,3)--(3.3,1.5).. controls ++(270:1.4cm) and ++(270:2.2cm) .. (-.2,-2);
	\draw[ultra thick] (.3,.5)--(.3,3.5).. controls ++(90:1.6cm) and ++(90:1.6cm) ..  (3.1,3).. controls ++(270:2.4cm) and ++(270:3.2cm) .. (.3,.5);;
	\draw[ultra thick] (.8,1.5)--(.8,3).. controls ++(90:1.4cm) and ++(90:1.6cm) ..  (2.8,3) .. controls ++(270:1.4cm) and ++(270:2.2cm) .. (.8,1.5);
	\draw[ultra thick] (1.7,3.5) arc (180:0:.3cm) -- (2.3,2.5) arc (0:-180:.3cm);	
	\draw[ultra thick] (2.15,3)--(2.3,2.7)--(2.45,3);
	\draw[ultra thick] (3.65,2.3)--(3.8,2)--(3.95,2.3);
	\draw[ultra thick] (3.15,2)--(3.3,1.7)--(3.45,2);
	\draw[ultra thick] (2.4,2.1)--(2.4,1.7)--(2.7,1.9);
	\draw[ultra thick] (2.5,1.2)--(2.5,.8)--(2.8,1);
	\draw[ultra thick] (-1.3,-1) arc (0:180:.3cm) -- (-1.9,-2) arc (180:360:.3cm);
	\draw[ultra thick] (-1.75,-1.2)--(-1.9,-1.5)--(-2.05,-1.2);

	\filldraw[thick, unshaded] (.5,-2)--(.5,-1)--(-1.5,-1)--(-1.5,-2)--(.5,-2);
	\filldraw[thick, unshaded] (1,-.5)--(1,.5)--(-.5,.5)--(-.5,-.5)--(1,-.5);
	\filldraw[thick, unshaded] (1.5,1)--(1.5,2)--(0,2)--(0,1)--(1.5,1);
	\filldraw[thick, unshaded] (2,2.5)--(2,3.5)--(.5,3.5)--(.5,2.5)--(2,2.5);
	\node at (-.5,-1.5) {$D_2$};
	\node at (.25,0) {$D_1$};
	\node at (.75,1.5) {$D_3$};
	\node at (1.25,3) {$D_4$};
\end{tikzpicture}
\longrightarrow
\begin{tikzpicture}[scale=.6,baseline=.5]
	\clip (-2.3,4) --(5.4,4) -- (5.4,-4) -- (-2.3,-4);

	\draw[ultra thick] (-.7,-2)--(-.7,2.2).. controls ++(90:2cm) and ++(90:2cm) ..  (5.2,2.2)--(5.2,-2.2).. controls ++(270:1.8cm) and ++(270:2.4cm) .. (-.7,-2);
	\draw[ultra thick] (-.2,-2)--(-.2,2).. controls ++(90:1.8cm) and ++(90:1.8cm) ..  (4.8,2)--(4.8,-2.2).. controls ++(270:1.4cm) and ++(270:2.2cm) .. (-.2,-2);
	\draw[ultra thick] (1.3,-1)--(.3,1)--(.3,2).. controls ++(90:1.6cm) and ++(90:1.2cm) ..  (4.4,2)--(4.4,-2.2).. controls ++(270:1cm) and ++(270:1.2cm) .. (1.3,-2);;
	\draw[ultra thick] (2,-1)--(2,2).. controls ++(90:1cm) and ++(90:1cm) ..  (4,2)--(4,-2) .. controls ++(270:1cm) and ++(270:1cm) .. (2,-2);
	\draw[ultra thick] (2.7,2) arc (180:0:.3cm) -- (3.3,1) arc (0:-180:.3cm);	
	\draw[ultra thick] (3.15,1.5)--(3.3,1.2)--(3.45,1.5);
	\draw[ultra thick] (3.85,-.1)--(4,-.4)--(4.15,-.1);
	\draw[ultra thick] (4.25,-.4)--(4.4,-.7)--(4.55,-.4);
	\draw[ultra thick] (4.65,-.7)--(4.8,-1)--(4.95,-.7);
	\draw[ultra thick] (5.05,-1)--(5.2,-1.3)--(5.35,-1);
	\draw[ultra thick] (-1.3,-1) arc (0:180:.3cm) -- (-1.9,-2) arc (180:360:.3cm);
	\draw[ultra thick] (-1.75,-1.2)--(-1.9,-1.5)--(-2.05,-1.2);
	%lower
	\filldraw[thick, unshaded] (.5,-2)--(.5,-1)--(-1.5,-1)--(-1.5,-2)--(.5,-2);
	\filldraw[thick, unshaded] (2.5,-2)--(2.5,-1)--(1,-1)--(1,-2)--(2.5,-2);	
	%upper
	\filldraw[thick, unshaded] (1,1)--(1,2)--(-.5,2)--(-.5,1)--(1,1);
	\filldraw[thick, unshaded] (3,1)--(3,2)--(1.5,2)--(1.5,1)--(3,1);
	%boxes
	\draw[dashed] (-2.2,-2.5)--(-2.2,-.5)--(2.8,-.5)--(2.8,-2.5)--(-2.2,-2.5);
	\draw[dashed] (-1,.5)--(-1,2.5)--(3.6,2.5)--(3.6,.5)--(-1,.5);
	%labels
	\node at (-.5,-1.5) {$D_2$};
	\node at (1.75,-1.5) {$D_3$};
	\node at (.25,1.5) {$D_1$};
	\node at (2.25,1.5) {$D_4$};
\end{tikzpicture}.
$$

Assume all input disks of $\cT$ are rectangles as in Remark \ref{rem:StandardForm}. Reindexing the internal disks, we suppose $C_1$ is a right $1,2$-cap. Set $U=\{D_1\}$ and $L=\{D_2\}$. Isotope the tangle so 
\item[$\bullet$] $D_1$ appears above $D_2$,
\item[$\bullet$] all strings connecting $D_1$ and $D_2$ either go upward from $D_2$ to $D_1$ with no critical points or are large arcs with only two  critical points, and
\item[$\bullet$] all $1$-caps and $2$-caps which enclose $C_1$ are large arcs with only two critical points, and all $1$-caps and $2$-caps which do not enclose $C_1$ are close to $D_1$ and $D_2$ respectively.
\item
We work our way inside to the next $i,j$-cap that is not a $1,2$-cap (clearly all the $1,2$-caps are right caps). If there are no other $i,j$-caps, we are finished. 

Otherwise, after reindexing, the next innermost $i,j$-cap $C_2$ is either a $1,3$-cap or $2,3$-cap, and all other strings connected to $D_2$ or $D_1$ respectively are $2$-caps or $1$-caps which can be isotoped so they are close to $D_2$ or $D_1$ respectively. We consider the case where $C_2$ is a $1,3$-cap, with the $2,3$-cap case being similar. We set $L=\{D_2,D_3\}$ (in the $2,3$-cap case, $U=\{D_1,D_3\}$), and we note that $C_2$ must also be a right cap as it is contained in $C_1$. We can also isotope the tangle so that 
\item[$\bullet$] $D_2$ and $D_3$ appear on the same horizontal level by moving $D_3$ around $C_2$,
\item[$\bullet$] all strings connecting disks in $U$ and $L$ either go upward from a disk in $L$ to a disk in $U$ with no critical points or are large arcs with only two critical points, and
\item[$\bullet$] all $1$-caps which enclose $C_2$ are large arcs with only two critical points (note that no $3$-cap can enclose $C_2$), and all $1$-caps and $2$-caps which do not enclose $C_2$ are close to $D_1$ and $D_2$ respectively.
\item
We work our way inside to the next $i,j$-cap that is not a $1,3$-cap (clearly all the $1,3$-caps are right caps). If there are no other $i,j$-caps, we are finished. 

Otherwise, after reindexing, the next $i,j$-cap $C_3$ is either a $1,4$-cap or a $3,4$-cap, and all other strings connected to $D_3$ or $D_1$ respectively are $3$-caps or $1$-caps which can be isotoped so they are close to $D_3$ or $D_1$ respectively. We consider the case where the next cap is a $3,4$-cap, with the $1,4$-cap case being similar. We set $U=\{D_1,D_4\}$ (in the $1,4$-case, $L=\{D_1,D_3,D_4\}$), and we note that $C_3$ must also be a right cap as it is contained in $C_2$. We can also isotope the tangle so that 
\item[$\bullet$] all disks in $U$ appear on the same horizontal level by moving $D_4$ around $C_3$, and
\item[$\bullet$] all strings connecting disks in $U$ and $L$ either travel upward from a disk in $L$ to a disk in $U$ with no critical points or travel in large arcs with only two critical points.
\item[$\bullet$] all $3$-caps which enclose $C_3$ are large arcs with only two critical points (note that no $1$-cap can enclose $C_3$), and all $1$-caps and $3$-caps which do not enclose $C_3$ are close to $D_1$ and $D_3$ respectively.
\item
We work our way inside once again. Iterating this process until we run out of $i,j$-caps, we see that all $i,j$-caps must be right caps. Moreover, we have partitioned our internal disks into two sets, $U,L$, and we can isotope $\cT$ so that:
\item[$\bullet$] all disks in $U$ and $L$ appear on the same horizontal levels,
\item[$\bullet$] any string connecting a disk $D_i\in U$ to a disk $D_j\in L$ travels upward from $D_j$ to $D_i$ with no critical points, or travels in a large arc from $D_i$ to $D_j$ with only two critical points, and
\item[$\bullet$] all $i$-caps for $D_i\in U$ which do not enclose a $k,\ell$-cap are close to $D_i$ and all $j$-caps for $D_j\in L$ which do not enclose a $k,\ell$-cap are close to $D_j$.
\item
It is now clear that we can put boxes around the disks and caps in $U,L$ as desired, and we are left with $\tau_n(\cT_1(-),\cT_2(-))$ for some $n\in\N$ and some connected tangles $\cT_1,\cT_2\in\B\P$. Note that the $n$ is determined by the $i,j$-caps and the $k$-caps which enclose an $i,j$-cap, and this $n$ is minimal when all other $\ell$-caps are contracted so they are close to $D_\ell$. Moreover, the only choice we made was the initial choice $U=\{D_1\}$ and $L=\{D_2\}$, but if we swapped $U$ and $L$, we would have ended up with $\tau_n(\cT_2(-),\cT_1(-))$. Hence $\cT_1,\cT_2$ are unique up to swapping.
\end{proof}

\begin{cor}\label{cor:subtangles}
Each $\cT\in\B\P$ with no external boundary points and at least one input disk contains an internally connected subtangle of the standard form:
\item[(1)] $t_1\op\cdots t_{r}\op t_{r+1} t_{r_i+2}\cdots t_{k}$ for some $0\leq r\leq k$, or 
\item[(2)] $\tau_{n_1+n_2}(\cT_1(-),\cT_2(-))$ or $\tau_{n_1+n_2}\op(\cT_1(-),\cT_2(-))$ for some connected $\cT_1,\cT_2\in\B\P$.
\end{cor}

\begin{defn}[Action of tangles in $\B\P$]\label{defn:TangleAction}
We may now describe the action of a tangle $\cT\in \B\P$ on a tuple 
$$
(z_1,\dots,z_s)\in \prod_{i=1}^s\widehat{Q_{n_i}^+}.
$$
If $\cT$ is connected, we put $\cT$ in the standard form afforded by Corollary \ref{cor:DecomposeTangle}, label the inputs with the $z_i$'s, and replace $1_n$ with $\id_{H_n}$, $t_n,t_n\op$ with $T_n,T_n\op$, and $\otimes_{m,n}$ with $\otimes_A$.

If $\cT$ is not connected, then there are closed subtangles as in Corollary \ref{cor:subtangles}. These closed subtangles will act as scalars in $\widehat{Q_0^+}=\widehat{Z(A)^+}=[0_\R,\I_R]$, and the order of scalar multiplication does not matter, so it suffices to define the scalar given by a single internally connected subtangle. 

First, closed loops count for a multiplicative factor: 
$$
\dim_{-A}(H)=T_1(1)=\ClosedLoop{1}\, \text{ and } \dim_{A-}(H)=T_1\op(1)=\ClosedLoopOp{1}.
$$
Suppose $\cS$ is a closed, internally connected subtangle of $\cT$ with only one input disk. Then we put $\cS$ in the standard form of (1) in Corollary \ref{cor:subtangles}, label the tangle by $z_i$, and replace $t_n,t_n\op$ with $T_n,T_n\op$. 

If $\cS$ is a closed, internally connected subtangle with more than one input disk. Then there is a right (respectively left) $i,j$-cap, we have unique connected $\cS_1,\cS_2\in\B\P$ such that $\cS=\tau_n(\cS_1(-),\cS_2(-))$ (respectively $\cS=\tau_n\op(\cS_1(-),\cS_2(-))$) by Theorem \ref{thm:IJcaps}. Now we replace $\tau_n$ with $\Tr_n$ (respectively $\tau_n\op$ with $\Tr_n\op$), and we get the action of $\cS_1,\cS_2$ as above to get $w_1\in \widehat{Q_{k_1}^+},w_2\in \widehat{Q_{k_2}^+}$, and we use Theorem \ref{thm:BilinearExtension} to get the value $\Tr_n(w_1\cdot w_2)$ (respectively $\Tr_n\op(w_1\cdot w_2)$).

Hence we have defined the action of any tangle in $\B\P$. To show the action is well-defined (well-behaved under composition of tangles), one uses methods of \cite{0912.1320} to show that the standard forms of connected and internally connected tangles given in Corollaries \ref{cor:DecomposeTangle} and \ref{cor:subtangles} and the maps given in Subsection \ref{sec:TowersOfBimodules} behave the same under composition by Theorems \ref{thm:BPRelations} and \ref{thm:relations}. We briefly sketch such an argument.

First, it suffices to consider the composites $\cR\circ_1 \cT$, $\cT'\circ_i \cT$, and $\cS\circ_j \cT$ of $\cR,\cS,\cT',\cT\in\B\P$ such that $\cR$ is internally connected with one input disk (see (1) of Corollary \ref{cor:subtangles}), $\cS$ is internally connected with two or more input disks (see (2) of Corollary \ref{cor:subtangles} and Theorem \ref{thm:IJcaps}), and $\cT,\cT'$ are connected. We may assume the respective connectivity properties because if any of $\cR,\cS,\cT',\cT$ had an internally connected subtangle $\cU$ which is not involved with the composition, the scalar that $\cU$ would contribute when acting remains unchanged by the composition. Note we cannot compose two nontrivial internally connected tangles.

Now if $\cR\circ_1 \cT, \cT'\circ_i \cT, \cS\circ_j \cT$ is still internally connected, internally connected, connected respectively, then we are finished by the existence of the standard forms and Theorems \ref{thm:BPRelations} and \ref{thm:relations}. In the cases of $\cR\circ_1 \cT$ and $\cT'\circ_i\cT$, we can only get internally connected tangles of the standard form (1) in Corollary \ref{cor:subtangles}, and once again, Theorems \ref{thm:BPRelations} and \ref{thm:relations} and (4) in Corollary \ref{cor:BPRelationsCor} are sufficient to show the well-definition. 

One must treat the case $\cS\circ_j \cT$ more carefully, as we may create internally connected tangles of the standard form (1) or (2) in Corollary \ref{cor:subtangles}. First, use that there are connected $\cS_1,\cS_2\in\B\P$ such that $\cS=\tau_n(\cS_1(-),\cS_2(-))$ (or $\OP$) by Theorem \ref{thm:IJcaps}. Then $\cT$ is inserted into $S_1$ or $S_2$, so we look at $S_k\circ_\ell \cT$. Now Theorems \ref{thm:BPRelations} and \ref{thm:relations} and (4) in Corollary \ref{cor:BPRelationsCor} are once again sufficient.
\end{defn}

%%%%%%%%%%%%%%%%%%%%%%%%%%%%%%%%%%%%%%%%%%%%%%%%%%%%%
\section{Extended positive cones}\label{sec:cones}
For the bimodule planar calculus, we need to make multiplication by $\I_\R$ rigorous. We do so by generalizing the notion of an extended positive cone.

\begin{defn}\label{defn:ConeDefn}
An \underline{extended positive cone} is a set $V$ together with a partial order $\leq$, an addition $+\colon V\times V\to V$, and a scalar multiplication $\cdot \colon [0_\R,\I_\R] \times V\to V$ such that
\itt{Additivity axioms} 
\item[$\bullet$] (Zero) There is a $0_V\in V$ such that $0_V+v=v+0_V=v$ for all $v\in V$.
\item[$\bullet$] (Infinity) There is an $\I_V\in V\setminus\{0\}$ such that $v+\I_V = \I_V+v = \I_V$ for all $v\in V$.
\item[$\bullet$] (Associativity) $v_1+(v_2+v_3)=(v_1+v_2)+v_3$ for all $v_1,v_2,v_3\in V$.
\item[$\bullet$] (Commutativity) $v_1+v_2=v_2+v_1$ for all $v_1,v_2\in V$.
\itt{Multiplicative axioms} 
\item[$\bullet$] (Unit) $1_\R v=v$ for all $v\in V$.
\item[$\bullet$] (Associativity) $(\lambda\mu) v=\lambda(\mu v)$ for all $\lambda,\mu\in [0_\R,\I_\R]$ and $v\in V$.
\item[$\bullet$] (Zero) $0_\R v=0_V$ for all $v\in V$.
\item[$\bullet$] (Infinity) $\lambda \I_V=\I_V$ for all $\lambda>0_\R$.
\itt{Distributivity}
\item[$\bullet$] (Scalars distribute) 
 $\lambda (v_1+v_2)=\lambda v_1+\lambda v_2$ for all $\lambda\in [0_\R,\I_\R]$ and $v_1,v_2\in V$.
\item[$\bullet$] ($V$ distributes)
 $(\lambda_1+\lambda_2)v = \lambda_1 v + \lambda_2 v$ for all $\lambda_1,\lambda_2\in [0_\R,\I_\R]$ and $v\in V$.
\itt{Partial order axioms}
\item[$\bullet$] (Non-degeneracy) $0_V\leq x\leq \I_V$ for all $x\in V$.
\item[$\bullet$] (Linearity) if $x_i\leq y_i$ for $i=0,1$ and $\lambda\in [0_\R,\I_\R]$, then $\lambda x_0 +x_1\leq \lambda y_0+y_1$.
\end{defn}

\begin{rem} 
\item[(1)] $0_V,\I_V\in V$ are unique.
\item[(2)] If $\lambda v = 0_V$, then $v=0_V$ or $\lambda=0_\R$.
\end{rem}

\begin{exs}\label{ex:EPC}
\item[(1)]
The set $[0_\R,\I_\R]$ with the usual ordering and the convention that $\lambda \I_\R = \I\lambda=\I_\R$ for all $\lambda\in\R_{>0}$ and $0_\R\I_\R=\I_\R0_\R=0_\R$ is an extended positive cone.
\item[(2)]
Let $X$ be a nonempty set. The space of functions $\{f\colon X\to [0_\R,\I_\R]\}$ is an extended positive cone with pointwise addition and scalar multiplication, where $f\leq g$ if $f(x)\leq g(x)$ for all $x\in X$. Similarly, the space of extended positive measurable functions on a measure space is an extended positive cone.
\item[(3)]
If $M$ is a von Neumann algebra, $\omega(M)$, the set of normal weights $\omega\colon M^+\to [0_\R,\I_\R]$, is an extended positive cone where $\I_{\omega(M)}$ is the map which sends $0_M$ to $0_\R$ and all other elements of $M$ to $\I_\R$, and $\varphi\leq \psi$ if $\varphi(x)\leq \psi(x)$ for all $x\in M^+$.
\item[(4)]
If $M$ is a von Neumann algebra, $\widehat{M^+}$ is an extended positive cone where $\I_{\widehat{M^+}}$ is the unbounded operator affiliated to $M$ with domain $(0)$, and $m_1\leq m_2$ if $m_1(\phi)\leq m_2(\phi)$ for all $\phi\in M_*^+$.
\item[(5)]
If $V,W$ are extended positive cones, then so is $V\times W$ where $(v_1,w_1)+(v_2,w_2)=(v_1+v_2,w_1+w_2)$, $\lambda (v_1,w_1)=(\lambda v_1,\lambda w_1)$, $0_{V\times W} = (0_{V},0_{W})$, $\I_{V\times W}=(\I_{V},\I_{W})$, and $(v_1,w_1)\leq (v_2,w_2)$ if $v_1\leq v_2$ and $w_1\leq w_2$.
\end{exs}

\begin{defn}
Let $V,W$ be extended positive cones. A function $T\colon V\to W$ is a linear map (of extended positive cones) if 
\item[$\bullet$] $T(\lambda u+v)=\lambda Tu+Tv$ for all $u,v\in V$ and $\lambda\in [0_\R,\I_\R]$, and
\item[$\bullet$] if $u,v\in V$ with $u\leq v$, then $Tu\leq Tv$.
\item
We define a multi-linear map of extended positive cones $V_1\times \cdots \times V_n\to V_0$ similarly.
\end{defn}

\begin{exs}\label{ex:MapExamples}
\item[(1)]
For a fixed scalar $\lambda\in [0_\R,\I_\R]$, multiplication by $\lambda$ is a map of extended positive cones.
\item[(2)]
Suppose $\omega\colon M^+\to [0_\R,\I_\R]$ is a normal weight. Then its unique extension to a normal weight $\omega\colon \widehat{M^+}\to [0_\R,\I_\R]$ is a map of extended positive cones.
\item[(3)]
If $m\in \widehat{M^+}$, then $m\colon \omega(M)\to[0_\R,\I_\R]$ given by$ \varphi\mapsto m(\varphi)$ is a map of extended positive cones.
\item[(4)]
Suppose $N\subset M$ is an inclusion of von Neumann algebras, $i\colon \widehat{N^+}\to \widehat{M^+}$ is the inclusion (well-defined by Equation \eqref{eq:KS}), and $T\colon \widehat{M^+}\to \widehat{N^+}$ is the unique extension of an operator valued weight $M^+\to \widehat{N^+}$. Then $i,T$ are maps of extended positive cones.
\item[(5)] 
Using the notation of Appendix \ref{sec:TensorUnbounded}, the map $\widehat{X^+} \times \widehat{Y_0^+}\to \widehat{X\otimes_A Y_0^+}$ given by $(x,y)\mapsto x\otimes_A y$ is a multilinear map of extended positive cones by Lemma \ref{cor:TensorConeMap}.
\end{exs}

\begin{defn}
An increasing net $(x_i)_{i\in I}\subset V$ converges to $x\in V$ if $x$ is the unique least upper bound for $(x_i)_{i\in I}$. We denote this convergence by $\sup_{i\in I} x_i=x$ or $x_i\nearrow x$. 
\item[$\bullet$] 
$V$ is \underline{complete} if each increasing net $(x_i)_{i\in I}$ has a unique least upper bound.
\item[$\bullet$] 
A map $T\colon V\to W$ is \underline{normal} if $x_i\nearrow x$ implies $Tx_i \nearrow Tx$.
\end{defn}

\begin{rem}\label{rem:normal}
The maps in Examples \ref{ex:MapExamples} are all normal.
\end{rem}

\begin{defn}
The \underline{dual space} of $V$, denoted $V^*$, is the set of all normal maps $V\to [0_\R,\I_\R]$. Note that $V^*$ is a complete extended positive cone with
\item[(1)] $(\lambda \varphi+\psi)(v)=\lambda \varphi(v)+\psi(v)$ for all $v\in V$, $\lambda\in [0_\R,\I_\R]$, and $\varphi,\psi\in V^*$, with the convention that $0_\R\cdot \I_\R=0_\R$,
\item[(2)] $0_{V^*}$ is the zero map,
\item[(3)] $\D \I_{V^*}(v)=\begin{cases}
0 & \text{if } v=0\\
\I_V & \text{else, and}
\end{cases}$
\item[(4)] $(\sup_{i\in I} \varphi_i)(v) := \sup_{i\in I} \varphi_i(v)$.
\item[$\bullet$] There is a natural inclusion $V\to V^{**}$ by $x\mapsto (\ev_x\colon \varphi\mapsto \varphi(x))$.
\item[$\bullet$] The \underline{completion} of $V$ is the set of sups of increasing nets in the image of $V$ in $V^{**}$.
\end{defn}

\begin{thm}
Let $M$ be a semifinite von Neumann algebra with n.f.s. trace $\Tr_M$. Let $\omega(M)$ be the set of normal weights on $M^+$.
\item[(1)] $\widehat{M^+}$ is the dual extended positive cone of $\omega(M)$ (the ordering on each is given in Examples \ref{ex:EPC}). 
\item[(2)] The map $\widehat{M^+}\ni x\mapsto \Tr_M(x\,\cdot \, )\in \omega(M)$ is a normal isomorphism of extended positive cones.
\end{thm}
\begin{proof}
This is a rewording of Theorem \ref{thm:BilinearExtension} into the language of this subsection.
\end{proof}

\begin{defn}\label{defn:ExtendedConeAdjoint}
If $T\colon V\to W$ is a normal map of extended positive cones, we get a map of dual spaces $T^*\colon W^*\to V^*$ by $T^*(\phi)=\phi\circ T$ for all $\phi\in W^*$. 
We can characterize it as the unique map satisfying
$$
\langle T(v), \varphi\rangle_W = \varphi(T(v))=\langle v, T^*(\varphi)\rangle_V
$$
for all $v\in V$ and $\varphi\in W$.
\end{defn}

\begin{prop}\label{prop:adjoint}
Suppose $N\subset M$ is an inclusion of semifinite von Neumann algebras with n.f.s. traces $\Tr_N,\Tr_M$ respectively. Let $i\colon \omega(N)\cong \widehat{N^+}\to \widehat{M^+}\cong \omega(M)$ be the inclusion, and let $T\colon \widehat{M^+}\to \widehat{N^+}$ be the unique extension to $\widehat{M^+}$ of the unique trace-preserving operator valued weight. Then $i,T$ are normal and $T=i^*$, $T^*=i$.
\end{prop}
\begin{proof}
Clearly $i,T$ are normal. Suppose $n\in \widehat{N^+}$ and $m\in (\widehat{M^+})^*=\widehat{M^+}$. Then
$$
\langle i(n),m\rangle_{\widehat{M^+}}=\Tr_M(m\cdot n)=\Tr_N(T(m)\cdot n) = \langle n, T(m)\rangle_{\widehat{N^+}},
$$
so $T=i^*$. Since $\Tr_M(m\cdot n) =\Tr_M(n\cdot m)$, $i=T^*$. 
\end{proof}

%%%%%%%%%%%%%%%%%%%%%%%%%%%%%%%%%%%%%%%%%%%%%%%%%
\bibliographystyle{amsalpha}
\bibliography{../bibliography}
\end{document}